\documentclass[a4paper,11pt]{article}

\usepackage{amsmath,amssymb,amsthm,bbm,esint,times,url,enumerate,fullpage,enumitem,wasysym}
\usepackage{epsf,graphicx,fancybox,subfig}
\usepackage[bookmarks=false,unicode,colorlinks,urlcolor=red,citecolor=red,linkcolor=blue]{hyperref}

\newtheorem{theorem}{Theorem}[section]
\newtheorem{proposition}[theorem]{Proposition}

\newtheorem{remark}[theorem]{Remark}

\newtheorem{conjecture}[theorem]{Conjecture}

\newcommand{\doi}[1]{%
\href{http://dx.doi.org/#1}{doi:\nolinkurl{#1}}}
\newcommand{\arxiv}[1]{%
 \href{http://front.math.ucdavis.edu/#1}{ArXiv:#1}}
\newcommand{\mref}[1]{%
\href{http://www.ams.org/mathscinet-getitem?mr=#1}{#1}}

\begin{document}

\title{Nodal domains, spectral minimal partitions and their relation to Aharonov-Bohm operators}

\author{V. Bonnaillie-No\"el\footnote{D\'epartement de Math\'ematiques et Applications, ENS Paris, CNRS, PSL Research University, 45 rue d'Ulm, 75005 Paris, France.
\href{mailto:virginie.bonnaillie@ens.fr}{bonnaillie\@@math.cnrs.fr}}, B. Helffer\footnote{Laboratoire de Math\'ematique Jean Leray  (UMR 6629), Universit\'e de Nantes, CNRS, 2 rue de la Houssini\`ere, BP 92208, F-44322 Nantes cedex 3, France and Laboratoire de Math\'ematiques (UMR 8628), Universit\'e Paris Sud, CNRS, B\^at 425, F-91405 Orsay cedex, France.
\href{mailto:bernard.helffer@univ-nantes.fr}{bernard.helffer\@@univ-nantes.fr}},  and T. Hoffmann-Ostenhof\footnote{Vienna University, Department of Theoretical Chemistry, A 1090 Wien, W\"ahringerstrasse 17, Austria. \href{mailto:thoffmann@tbi.univie.ac.at}{thoffmann\@@tbi.univie.ac.at}}
}
\date{}
\maketitle
\begin{abstract}
This survey is  a short version of a chapter written by the first two authors in the book \cite{HenrotBook} (where more details and references are given) but we have decided here to emphasize more on the role of the Aharonov-Bohm operators which appear to be a useful tool coming from physics for understanding a problem motivated either by spectral geometry or dynamics of population. Similar questions appear also in Bose-Einstein theory. Finally some open problems which might be of interest are mentioned.
\end{abstract}

\section{Introduction}

In this survey, we mainly consider  the Dirichlet realization of the Laplacian operator  in $\Omega$, when $\Omega$ is a bounded domain in $\mathbb R^2$ with piecewise-$C^{1,+}$ boundary (domains with corners or cracks\footnote{For example a square with a segment removed. By $C^{1,+}$, we mean $C^{1,\alpha}$ for some $\alpha >0\,$.} also permitted). This operator will be denoted by $H(\Omega)$. We would like to analyze the  connections  between the nodal domains of the eigenfunctions of $H(\Omega)$ and the partitions of $\Omega$ by $ k$ open sets $ D_i$ which are minimal in the sense that the maximum over the $D_i$'s of the groundstate energy of the Dirichlet realization of the Laplacian $H(D_i)$ is minimal. This problem can be seen as a strong competition limit of segregating species in population dynamics (see \cite{BHIM},  \cite{MR2151234} and references therein). Similar questions appear also in the analysis of the segregation and the symmetry breaking of a two-component condensate in the Bose-Einstein theory (see \cite{MR3148108} and references therein), with $\Omega =\mathbb R^2$ and $H(\Omega)$ replaced by the harmonic oscillator or more generally by a Schr\"odinger operator.

To be more precise,  we start from the following weak notion of partition:\\
{\it A {\bf partition} (or $k$-partition for indicating the cardinality of the partition) of $\Omega$ is a family $\mathcal D=\{D_i\}_{1\leq i\leq k}$ of $k$ mutually disjoint sets in $\Omega$ (with $k\geq1$ an integer).
}\\
If we denote by $\mathfrak O_k(\Omega)$ the set of partitions of $\Omega$ where the $D_i$'s are domains ({\it i.e.} open and connected), we  introduce the energy of the partition:
\begin{equation}\label{chapBH.LaD}
\Lambda(\mathcal D)=\max_{1\leq i\leq k}\lambda(D_i),
\end{equation}
where  $\lambda(D_{i})$ is the ground state energy ({\it i.e.} the lowest eigenvalue) of $H(D_{i})$.
The optimal problem we are interested in is the determination, for any integer $k\geq1\,,$ of
\begin{equation}\label{chapBH.eq.Lkdef}
\mathfrak L_{k}=\mathfrak L_{k}(\Omega) = \inf_{\mathcal D\in\mathfrak O_k(\Omega)} \Lambda(\mathcal D).
\end{equation}
We can also consider the case of a two-dimensional Riemannian manifold and the Laplacian is then the Laplace-Beltrami operator. 
 We say that $(\varphi,\lambda)$ is a spectral pair for $H(\Omega)$ if $\lambda$ is an eigenvalue of the Dirichlet Laplacian $H(\Omega)$ on $\Omega$ and $\varphi\in E(\lambda)\setminus \{0\}$, where $E(\lambda)$ denotes the eigenspace attached to $\lambda$.
We denote by $\{\lambda_n(\Omega),n\geq 1\}$  the non decreasing sequence of eigenvalues of $H(\Omega)$ and by $\{\varphi_n,n\geq 1\}$ some associated orthonormal basis of eigenfunctions.  The groundstate $\varphi_1$ can be chosen to be strictly positive in $\Omega$, but the other excited eigenfunctions $\varphi_n$ must have zerosets. Here we recall that for $\varphi\in C^0(\overline\Omega)$, the nodal set (or zeroset) of $\varphi$ is defined by~:
\begin{equation}
{\mathcal N}(\varphi)=\overline{\{ {\bf x} \in \Omega\:\big|\: \varphi({\bf x})=0\}}\,.
\end{equation}
In the case when $\varphi$ is an eigenfunction of the Laplacian, the $\mu(\varphi)$ components of $\Omega\setminus {\mathcal N}(\varphi)$ are called the nodal domains\index{nodal domains} of $\varphi$ and define naturally a partition of $\Omega$ by $\mu(\varphi)$ open sets, which is called a {\bf nodal partition}.\\
Our main goal is to discuss the links between the partitions of $\Omega$ associated with these eigenfunctions and the minimal partitions of $\Omega$. We will also describe how these minimal partitions can also be seen as  a nodal partition of an eigenfunction of a suitable Aharonov-Bohm operator. There is of course a strong relation between the analysis of nodal sets of eigenfunctions and nodal domains. Here the surveys of S. Zelditch, \cite{MR3087922, MR3524217} show the growing importance of this field. There mostly the case of the Laplace Beltrami operator  on smooth manifolds is reviewed. The results about nodal sets of general Schr\"odinger operators are rather scattered and we are not able to find a good survey for this. 
We want to stress that these topics represent only a small part of this and related areas in physics and mathematics. We just mention some of the  surveys and papers which demonstrate that this field is huge and there are many different communities which in many cases, unfortunately,  interact only  little. Two instructive reviews are  \cite{MR3102912,1709.03650} written by  mathematical physicists where questions about quantum chaos, nodal domain statistics and many other topcis are discussed. 
Many questions about nodal domains and partitions have their  counterparts in quantum graphs and spectral graph theory, see for instance  \cite{MR1855391, MR3013208, MR2340484}.

\section{Nodal partitions\index{nodal domains}}\label{chapBH.s2}

\subsection{On the local structure of nodal sets}
We refer for this section to the survey of P. B\'erard \cite{MR671608}. We recall that, if $\varphi$ is an eigenfunction associated with $\lambda$ and $D$ is one of its nodal domains, then the restriction of $\varphi$ to $D$ belongs to $H_0^1(D)$ and is an eigenfunction of the Dirichlet Laplacian in $D$. Moreover, $\lambda$ is the ground state energy in $D$.

Using~\cite{MR0075416,MR0397805}, we can prove that nodal sets are regular in the sense:
\begin{itemize}[label=--,itemsep=-2pt]
\item The singular points ${\bf x}_0$ on the nodal lines are isolated. 
\item At the singular points, an even number of half-lines meet with equal angle.
\item At the boundary, this is the same as  adding the tangent line in the picture.
\end{itemize}
The notion of regularity will be defined later for general partitions. 

\subsection{Courant's theorem and Courant sharp\index{Courant sharp} eigenvalues}\label{subsection-Courant-sharp}
The following theorem was established by R. Courant \cite{Cou} in 1923 for the Laplacian with Dirichlet or Neumann conditions.
\begin{theorem}[Courant]\label{chapBH.thm.Courant}\index{Courant nodal Theorem}
The number of nodal components of the $k$-th eigenfunction is not greater than $k$.
\end{theorem}

We say that a spectral pair $(\varphi,\lambda)$ is Courant sharp if $\lambda=\lambda_k$ and $u$ has $k$ nodal domains. We say that an eigenvalue $\lambda_k$ is Courant sharp if there exists an eigenfunction $\varphi$ associated with $\lambda_{k}$ such that $(\varphi,\lambda_k)$ is a Courant sharp spectral pair.

Whereas  the Sturm-Liouville theory shows that in dimension $1$ all the spectral pairs are Courant sharp, we will see below that in higher dimension,
the Courant sharp situation can only occur for a finite number of eigenvalues.

 \subsection{Pleijel's theorem\index{Pleijel Theorem}}\label{chapBH.ss2.5}
Pleijel proves in 1956 the following theorem \cite{MR0080861}:
\begin{theorem}[Weak Pleijel's theorem]\label{chapBH.Pleijelweakform}
If the dimension is $\geq 2$, there is only a finite number of Courant sharp eigenvalues of the Dirichlet Laplacian.
\end{theorem}
This theorem is the consequence of a more precise theorem which gives a link between Pleijel's theorem and partitions. For describing this result and its proof, we first recall the Faber-Krahn inequality:\index{Faber-Krahn (inequality)}
\begin{theorem}[Faber-Krahn inequality]
For any domain $D\subset \mathbb R^2$, we have
\begin{equation}\label{chapBH.eq.FK}
|D|\ \lambda(D) \geq \lambda(\Circle )\,,
\end{equation}
 where $|D|$ denotes the area of $D$ and $\Circle$ is the disk of unit area ${\mathcal B}\Big(0, \frac{1}{\sqrt{\pi}}\Big)\,. $
\end{theorem}

Note that improvements can be useful when $D$ is "far" from a disk. It is then interesting to have a lower bound for $|D|\ \lambda(D) - \lambda(\Circle )$. We refer for example to \cite{MR3357184} and \cite{MR1266215}. These ideas are behind recent improvements by Steinerberger \cite{MR3272823}, Bourgain \cite{MR3340367} and Donnelly \cite{MR3205801} of the strong Pleijel's theorem below. Its proof  is indeed enlightning.
First, by summation of Faber-Krahn's inequalities \eqref{chapBH.eq.FK} applied to each $D_i$ and having in mind the definition of the energy, we deduce that
for any open partition $\mathcal D$ of $\Omega$ we have
\begin{equation}\label{chapBH.fkv2}
 |\Omega|\ \Lambda(\mathcal D) \geq \sharp (\mathcal D) \, \lambda( \Circle )\,,
\end{equation}
where $\sharp (\mathcal D)$ denotes the number of elements of the partition.\\
Secondly, we implement Weyl's formula for the counting function of the Laplacian which reads in dimension $d$
\begin{equation}\label{chapBH.wform1}
 N(\lambda) \sim \frac{\omega_d}{(2\pi)^{d}} |\Omega| \lambda^{\frac d2},\quad \mbox{ as } \lambda \rightarrow +\infty\,,
\end{equation}
where $\omega_d$ denotes the volume of a ball of radius 1 in $\mathbb R^d$ and $|\Omega|$ the volume of $\Omega$. This leads to:
\begin{theorem}[Strong Pleijel's theorem]\label{chapBH.PropPleijel1}
Let $\varphi_{n}$ be an eigenfunction of $H(\Omega)$ associated with $\lambda_{n}(\Omega)$. Then
\begin{equation}\label{chapBH.bourgain3}
\limsup_{n\to +\infty} \frac{\mu(\varphi_n)}{n}\leq \frac{4 \pi }{ \lambda(\Circle)}\,.
\end{equation}
\end{theorem}

\begin{remark} It is natural (see an important motivation for minimal partitions in the next section) to determine the Courant sharp situation for some examples. This kind of analysis has been initiated by \AA. Pleijel for the square, and continued in \cite{MR2483815} for the disk and some rectangles (rational or irrational).
We recall that, according to Theorem~\ref{chapBH.Pleijelweakform}, there are a finite number of Courant sharp eigenvalues. The point is to quantify this number or to find lower bounds or upper bounds for the largest integer $n$ such that $\lambda_{n-1} < \lambda_n$ with $\lambda_n$ Courant sharp. This involves an explicit control of the remainder in Weyl's formula and then a case by case analysis of the remaining eigenspaces.

Other domains have been analyzed by various subgroups of the set of authors (Band, B\'erard, Bersudsky, Charron,  Fajman, Helffer, Hoffmann-Ostenhof, Kiwan, L\'ena, Leydold, Persson-Sundqvist, Terracini, \ldots, see \cite{HenrotBook} for the references): the square and the annulus for the Neumann-Laplacian, the sphere, the irrational and equilateral torus, the triangle (equilateral, hemi-equilateral, right angled isosceles), Neumann $2$-rep-tiles,  the cube, the ball, the 3D-torus.
\end{remark}
\begin{remark}\label{RemPl}
 Pleijel's Theorem extends to bounded domains in $\mathbb{R}^d$, and more generally to compact $d$-manifolds with boundary, with a constant $\gamma(d) <1$ replacing $ {4 \pi} / {\lambda(\Circle)}$
in the right-hand side of \eqref{chapBH.bourgain3} (see Peetre \cite{MR0092917}, B\'{e}rard-Meyer \cite{MR690651}). This constant is independent of the geometry.  It is also true for the Neumann Laplacian in a piecewise analytic bounded domain in $\mathbb R^2$ (see \cite{MR2457442} whose proof is based on a control of the asymptotics of the number of boundary points belonging to the nodal sets of the eigenfunction associated with $\lambda_k$ as $k\to +\infty\,$, a difficult result proved by Toth-Zelditch \cite{MR2487604}).
C. L\'ena \cite{1609.02331} gets the same result for $C^2$ domains,  without any condition on the dimension,  through a very nice decomposition of the nodal domains. In \cite{MR3584180,MR3584181}, the authors determine an upper bound for Courant-sharp eigenvalues, expressed in terms of simple geometric invariants of $\Omega$.   Finally, it is expected that Pleijel's theorem can be extended to Schr\"odinger operators $-\Delta +V$, either for the negative spectrum if $V\rightarrow 0$  as $|{\bf x}| \rightarrow +\infty$ (for instance the potential of the Hydrogen atom) or for the whole spectrum if $V\rightarrow +\infty\,$. This has been proved for the harmonic oscillator by P. Charron and for radial potentials by Charron-Helffer-Hoffmann-Ostenhof (see \cite{1604.08372} and references therein).
\end{remark}

\section{Minimal spectral partitions}\label{chapBH.s4}
\subsection{Definitions}\label{chapBH.ss4.1}
  Spectral minimal partitions were introduced in \cite{MR2483815} within a more general class  depending on $p\in [1,+\infty]$ ($p=1$ and $p=+\infty$  being physically the most interesting).
For any integer $k\geq1$ and $ p\in [1,+\infty[$, we define the $p$-energy of a $k$-partition $\mathcal D=\{D_{i}\}_{1\leq i\leq k}$ by
\begin{equation}\label{chapBH.LAp}
\Lambda_{p}(\mathcal D)=\Big(\frac 1k \sum_{i=1}^k \lambda(D_i)^p\Big)^{\frac 1p}\,.
\end{equation}
The associated optimization problem writes
\begin{equation}\label{chapBH.Lfrakp}
\mathfrak L_{k,p}(\Omega)=\inf_{\mathcal D\in \mathfrak O_k(\Omega)}\Lambda_{p}(\mathcal D)\,,
\end{equation}
and we call $p$-minimal $k$-partition a $k$-partition with $p$-energy $\mathfrak L_{k,p}(\Omega)$.\\
For $p=+\infty$, we write $\Lambda_{\infty}(\mathcal D)=\Lambda(\mathcal D)$ and $\mathfrak L_{k,\infty}(\Omega)=\mathfrak L_k(\Omega)\,$.

The analysis of the properties of minimal partitions leads us to introduce two notions of regularity:
\begin{itemize}[label=--,itemsep=-2pt]
\item A partition $\mathcal D=\{D_i\}_{1\leq i\leq k}$ of $\Omega$ in $\mathfrak O_k(\Omega)$ is called {\bf strong} if
$${\rm{ Int\,}}(\overline{\cup_i D_i}) \setminus \partial \Omega =\Omega\,.$$
\item It is called {\bf nice} if 
$$D_i= {\rm{ Int\,}}(\overline{D_i})\cap \Omega\,,$$
 for any $1\leq i\leq k$.
\end{itemize}
In Figure \ref{chapBH.fig.exvecpABcarre}, only the fourth picture gives a nice partition. Attached to a strong partition, we associate the boundary set 
\begin{equation}\label{chapBH.assclset}
\partial \mathcal D = \overline{ \cup_i \left( \Omega \cap \partial D_i \right)}\,,
\end{equation}
which  plays the role of the nodal set (in the case of a nodal partition). \\
To go further, we introduce the set ${\mathfrak O}^{\mathsf{reg}}_k(\Omega)$ of {\bf regular} partitions, which should satisfy the following properties~:
\begin{enumerate}[label={\rm(\roman*)}]
\item Except  at  finitely many distinct ${\bf x}_i\in\Omega\cap \partial\mathcal D$ in the neigborhood of which $\partial \mathcal D $ is the union of $\nu({\bf x}_i)$ smooth curves ($\nu({\bf x}_i)\geq 2$) with one end at ${\bf x}_i$, $\partial \mathcal D $ is locally diffeomorphic to a regular curve.
\item $\partial\Omega\cap \partial \mathcal D $ consists of a (possibly empty) finite set of points ${\bf y}_j$. Moreover $\partial \mathcal D $ is near ${\bf y}_j$ the union of $\rho({\bf y}_j)$ distinct smooth half-curves which hit ${\bf y}_j$.
\item $\partial \mathcal D $ has the equal angle meeting property, that is the half curves cross with equal angle at each singular interior point of $\partial \mathcal D $ and also at the boundary together with the tangent to the boundary.
\end{enumerate}
 We denote by $X(\partial \mathcal D )$ the set corresponding to the points ${\bf x}_i$ introduced in (i) and by $Y(\partial \mathcal D )$ corresponding to the points ${\bf y}_j$ introduced in (ii).\\

This notion of regularity for partitions is very close to what we have observed for the nodal partition of an eigenfunction. The main difference is that, in the nodal case, there is always an even number of half-lines meeting at an interior singular point.

\subsection{Bipartite partitions}\label{chapBH.ss4.3}
Two sets $ D_i,D_j$ of the partition $\mathcal D$ are neighbors and write $ D_i\sim D_j$, if
$D_{ij}={\rm{ Int\,}}(\overline {D_i\cup D_j})\setminus \partial \Omega $
is connected. A regular partition is bipartite\index{bipartite partition} if it can be colored by two colors (two neighbors having two different colors).\\
Nodal partitions are the main examples of bipartite partitions.
Note that in the case of a simply connected planar domain, we know by graph theory that, if for a regular partition all the $\nu({\bf x}_i)$ are even, then the partition is bipartite. This is no more the case on an annulus or on a surface. 

 \subsection{Main properties of minimal partitions}\label{chapBH.ss4.4}
It has been proved by Conti-Terracini-Verzini  (existence) and Helffer--Hoffmann-Ostenhof--Terracini (regularity) ( \cite{MR2483815} and references therein) the following theorem:
\begin{theorem}\label{chapBH.thstrreg}
For any $ k$, there exists a minimal $k$-partition which is strong and regular. Moreover any minimal $k$-partition has a strong and regular representative\footnote{possibly after a modification of the open sets of the partition by capacity $0$ subsets.}. The same result holds for  the $p$-minimal $k$-partition problem with $p\in [1,+\infty)$. 
\end{theorem}
Note that  the regularity property implies that a minimal  partition is nice.\\

When $p=+\infty$, minimal spectral partitions have two important properties. Let ${\mathcal D}=\{D_{i}\}_{1\leq i\leq k}$ be a minimal $k$-partition, then
\begin{itemize}[label=--,itemsep=-2pt]
\item The minimal partition ${\mathcal D}$ is a spectral equipartition, {\it i.e.} $\mathfrak L_{k}(\Omega)=\lambda(D_{i})$, $\forall 1\leq i\leq k$.
\item For any pair of neighbors $D_i\sim D_j$,
$\lambda_2(D_{ij}) = \mathfrak L_k (\Omega)\,.$
\end{itemize}

For the first property, this can be understood, once the regularity is obtained by pushing the boundary and using the Hadamard formula \cite{MR2251558}. 
For the second property, we can observe that $\{D_i,D_j\}$ is necessarily a minimal $2$-partition of $D_{ij}$. This leads to what we call the {\it  pair compatibility condition}.

Note that in the proof of Theorem~\ref{chapBH.thstrreg}, one obtains on the way an useful construction. Attached to each $D_i$, there is a distinguished ground state $u_i$ such that $u_i >0$ in $D_i$ and such that for each pair of neighbors $\{D_i,D_j\}$, $u_i-u_j$ is the second eigenfunction of the Dirichlet Laplacian in $D_{ij}$. 

Let us now establish two important properties concerning the monotonicity (according to $k$ or the domain $\Omega$):
\begin{itemize}[label=--,itemsep=-2pt]
\item For any $k\geq1\,$, we have $\mathfrak L_k (\Omega)<\mathfrak L_{k+1}(\Omega)\,$.
\item If $\Omega \subset \widetilde \Omega$, then
$\mathfrak L_k (\widetilde \Omega) \leq \mathfrak L_k(\Omega)$ for any $k\geq 1\,.$
\end{itemize}

\subsection{Minimal spectral partitions and Courant sharp property}\label{ss3.4}
A natural question is whether a minimal partition of $\Omega$ is a nodal partition. We have first the following converse theorem (see \cite{MR2483815}):
\begin{theorem}\label{chapBH.partnod}
If a minimal partition is bipartite\index{bipartite partition}, it is a nodal partition\index{nodal domains}.
\end{theorem}
\begin{proof}
Combining the bipartite assumption and the consequence of the pair compatibility condition mentioned after Theorem~\ref{chapBH.thstrreg}, it is immediate to construct some $\varphi\in H_0^1(\Omega)$ such that
$$\varphi_{\vert D_i} = \pm \varphi_i,\quad\forall 1\leq i\leq k,\qquad\mbox{ and }
\qquad-\Delta \varphi =\mathfrak L_k(\Omega) \varphi \quad\mbox{ in }\Omega\setminus X(\partial \mathcal D).$$
 A capacity argument shows that $-\Delta \varphi =\mathfrak L_k(\Omega) \varphi$ in all $\Omega$ and hence $\varphi$ is an eigenfunction of $H(\Omega)$ whose nodal set is $\partial\mathcal D$.
 \end{proof}

The next question is then to determine how general  the previous situation is. Surprisingly, this only occurs in the so called Courant sharp situation.  For the statement, we need another spectral sequence. 
For any $k\geq1$, we denote by $L_k(\Omega)$ (or $L_{k}$ if there is no confusion) the smallest eigenvalue (if any) for which there exists an eigenfunction with $k$ nodal domains. We set $L_k(\Omega) =+\infty$ if there is no eigenfunction with $k$ nodal domains.
In general, one can show, as an easy consequence of the max-min characterization of the eigenvalues, that
\begin{equation} \label{chapBH.compLLL}
\lambda_k(\Omega)\leq\mathfrak L_k(\Omega)\leq L_k(\Omega)\,.
\end{equation}
The following important theorem (due to \cite{MR2483815}) gives the full picture of the equality cases:
\begin{theorem}\label{chapBH.TheoremL=L}
Suppose $\Omega\subset \mathbb R^2$ is regular. If $\mathfrak L_k(\Omega)=L_k(\Omega)$ or $\mathfrak L_k(\Omega)=\lambda_k(\Omega)$, then
\begin{equation}\label{chapBH.LLL}
\lambda_k(\Omega)=\mathfrak L_k(\Omega)=L_k(\Omega)\,.
\end{equation}
In addition, there exists a Courant sharp eigenfunction associated with $\lambda_{k}(\Omega)$.
\end{theorem}

{\noindent\it Sketch of the proof.} 
It is easy to see using a variation of the proof of Courant's theorem that the equality $\lambda_k=\mathfrak L_k$ implies \eqref{chapBH.LLL}. Hence the difficult part is to get \eqref{chapBH.LLL} from the assumption that $L_k=\mathfrak L_k =\lambda_{m(k)}$, that is to prove that $m(k)=k$. This involves a construction of an exhaustive family $\{\Omega(t),\ t \in (0,1)\}$, interpolating between $\Omega(0):=\Omega \setminus {\mathcal N}(\psi_k)$ and $\Omega(1):=\Omega$, where $\psi_k$ is an eigenfunction corresponding to $L_k$ such that its nodal partition is a minimal $k$-partition. This family is obtained by cutting small intervals in each regular component of ${\mathcal N}(\psi_k)$. $L_k$ is  an eigenvalue common to all  $H(\Omega(t))$, but its labelling changes between $t=0$ and $t=1$ at some $t_0$ where the multiplicity of $L_k$ should increase. By a tricky argument which is not detailed here,  we get  a contradiction. $\hfill\square$

Similar results hold in the case of compact Riemannian surfaces when considering the Laplace-Beltrami operator. Typical cases are analyzed like $\mathbb S^2$ in \cite{MR2664708} and $\mathbb T^2$ in \cite{MR3348988}. In the case of dimension~$3$, let us mention that Theorem \ref{chapBH.TheoremL=L} is proved in \cite{MR2644760}. The complete analysis of the topology of minimal partitions in higher dimension is not achieved as for the bidimensional case.

\subsection{Topology of regular partitions}\label{chapBH.s6}
\subsubsection{Euler's formula\index{Euler formula} for regular partitions}\label{chapBH.ss6.1}
In the case of planar domains, one can use a variant of Euler's formula in the following form (see \cite{MR1736932,MR3151084}).
\begin{proposition}\label{chapBH.Euler}
Let $\Omega$ be an open set in $\mathbb R^2$ with piecewise $C^{1,+}$ boundary and $\mathcal D$ be a $k$-partition with $\partial \mathcal D $ the boundary set. Let $b_0$ be the number of components of $\partial \Omega$ and $b_1$ be the number of components of $\partial \mathcal D \cup\partial \Omega$. Denote by $\nu({\bf x}_i)$ and $\rho({\bf y}_i)$ the numbers of curves ending at ${\bf x}_i\in X(\partial \mathcal D)$, respectively ${\bf y}_i\in Y(\partial \mathcal D)$. Then
\begin{equation}\label{chapBH.Emu}
k= 1 + b_1-b_0+\sum_{{\bf x}_i\in X(\partial \mathcal D )}\Big(\frac{\nu({\bf x}_i)}{2}-1\Big)
+\frac{1}{2}\sum_{{\bf y}_i\in Y(\partial \mathcal D )}\rho({\bf y}_i)\,.
\end{equation}
\end{proposition}

This can be applied, together with other arguments to determine upper bounds for the number of singular points of minimal partitions. There is a corresponding result for compact manifolds involving the Euler characteristics.

\subsubsection{Application to regular $3$-partitions}\label{chapBH.ss6.2}
As an application of the Euler formula, we can describe (see \cite{MR2743435})  the possible ``topological'' types of non bipartite minimal $3$-partitions when $\Omega$ is a simply-connected domain in $\mathbb R^2$. 
Then the boundary set $\partial \mathcal D $ has one of the following properties:
\begin{enumerate}[label={\rm[\alph*]},itemsep=-2pt]
\item one interior singular point ${\bf x}_0 \in\Omega$ with $\nu({\bf x}_0)=3$, three points $\{{\bf y}_i\}_{1\leq i\leq3}$ on the boundary $\partial\Omega$ with $\rho({\bf y}_i)=1$;
\item two interior singular points ${\bf x}_0,\ {\bf x}_1\in\Omega$ with $\nu({\bf x}_0)=\nu({\bf x}_1)=3$ and two boundary singular points ${\bf y}_1,\ {\bf y}_2\in\partial\Omega$ with $\rho({\bf y}_1)=1=\rho({\bf y}_2)$;
\item two interior singular points ${\bf x}_0,\ {\bf x}_1\in\Omega$ with $\nu({\bf x}_0)=\nu({\bf x}_1)=3$ and no singular point on the boundary.
\end{enumerate}
This helps us to analyze  (with some success)  the minimal $3$-partitions with some topological type. We actually do not know any example where the minimal $3$-partitions are of type [b] and [c] (see numerical computations in \cite{MR2598097} for the square and the disk, \cite{MR3245081} for circular sectors and see \cite{MR3093548} for complements in the case of  the disk).

\subsubsection{Upper bound for the number of singular points \label{chapBH.ss6.3}}
Euler's formula also  implies
\begin{proposition}\label{chapBH.prop.uppk}
Let $\mathcal D$ be a minimal $k$-partition of a simply connected domain $\Omega$ with $k\geq 2$. Let $X^{\sf odd}(\partial \mathcal D )$ be the subset among the interior singular points $X(\partial\mathcal D)$ for which $\nu({\bf x}_{i})$ is odd. Then the cardinality of $X^{\sf odd}(\partial \mathcal D)$ satisfies
\begin{equation}\label{chapBH.eulergrandk}
\sharp X^{\sf odd}(\partial \mathcal D ) \leq 2k -4 \,.
\end{equation}
\end{proposition}

In the case of $\mathbb S^2$, one can prove that a minimal $3$-partition is not nodal (the second eigenvalue has multiplicity $3$), and as a step towards a characterization, one can show that non-nodal minimal partitions have necessarily two singular \emph{triple points} ({\it i.e.} with $\nu({\bf x})=3$). If we assume, for some $k\geq 12$, that a minimal $k$-partition has only singular triple points and consists only of (spherical) pentagons and hexagons, then Euler's formula in  its historical version for convex polyedra $V-E+F= \chi(\mathbb S^2) =2$ (where $F$ is the number of faces, $E$ the number of edges and $V$ the number of vertices) implies that the number of pentagons is $12$.  It has been proved by Soave-Terracini \cite[Theorem 1.12]{MR3345178} that
$$\mathfrak L_{3}(\mathbb S^d) = \frac 32\left(d+\frac 12\right).$$

\subsection{Examples of minimal $k$-partitions}\label{chapBH.sexamples}

When $\Omega$ is a  disk or a square, one can show the minimal $k$-partition are nodal only for $k=1,2,4$ (see Figure~\ref{chapBH.fig.partminDisk} for $k=2,4$).
\begin{figure}[h!bt]
\begin{center}
\subfloat[{Minimal $k$-partitions, $k=2,4\,$.\label{chapBH.fig.partminDisk}}]{\begin{tabular}{p{1.25cm}p{1.25cm}p{1.25cm}p{1.25cm}p{1.25cm}}
\includegraphics[height=1.6cm]{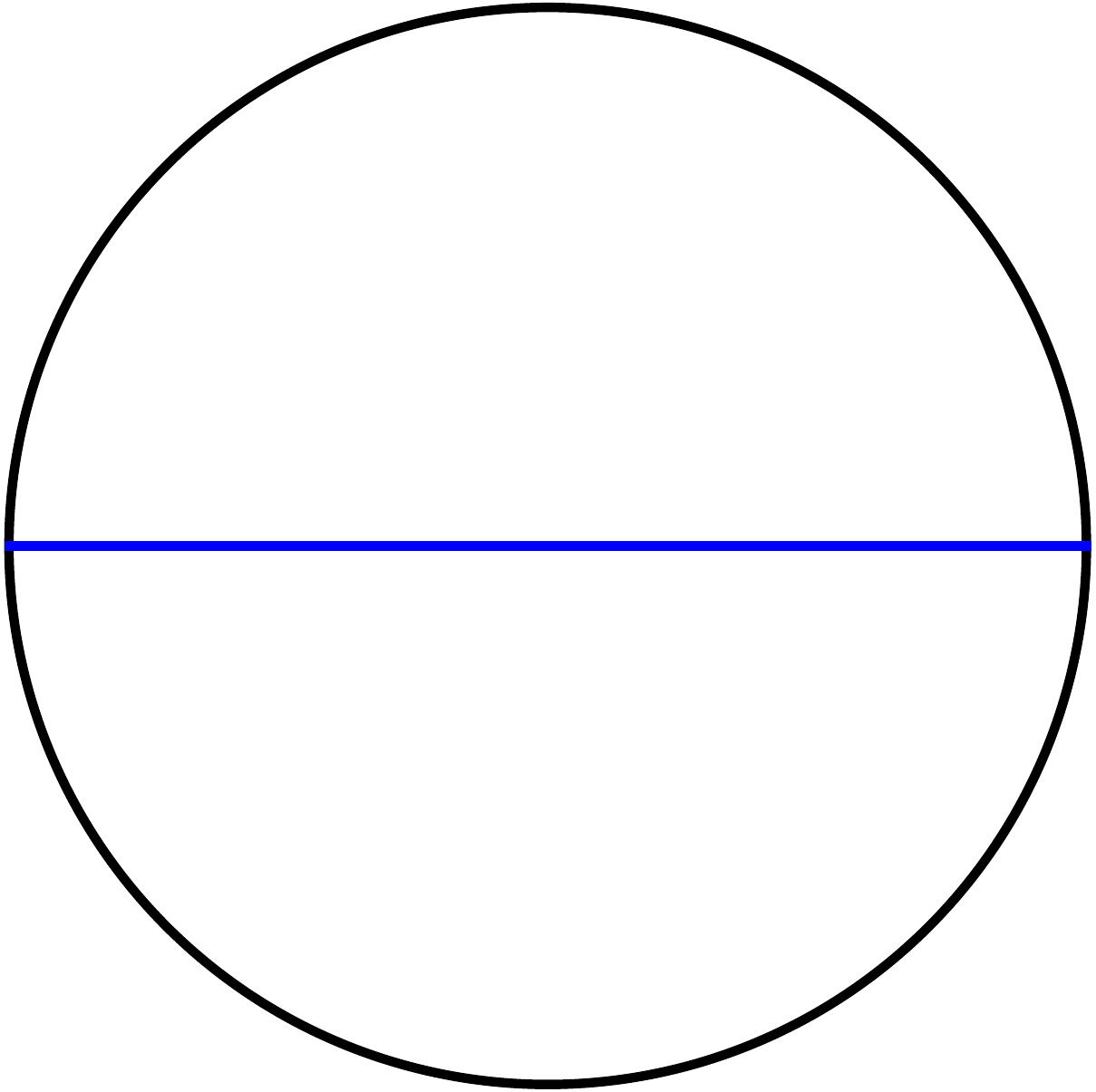}
&\includegraphics[height=1.6cm]{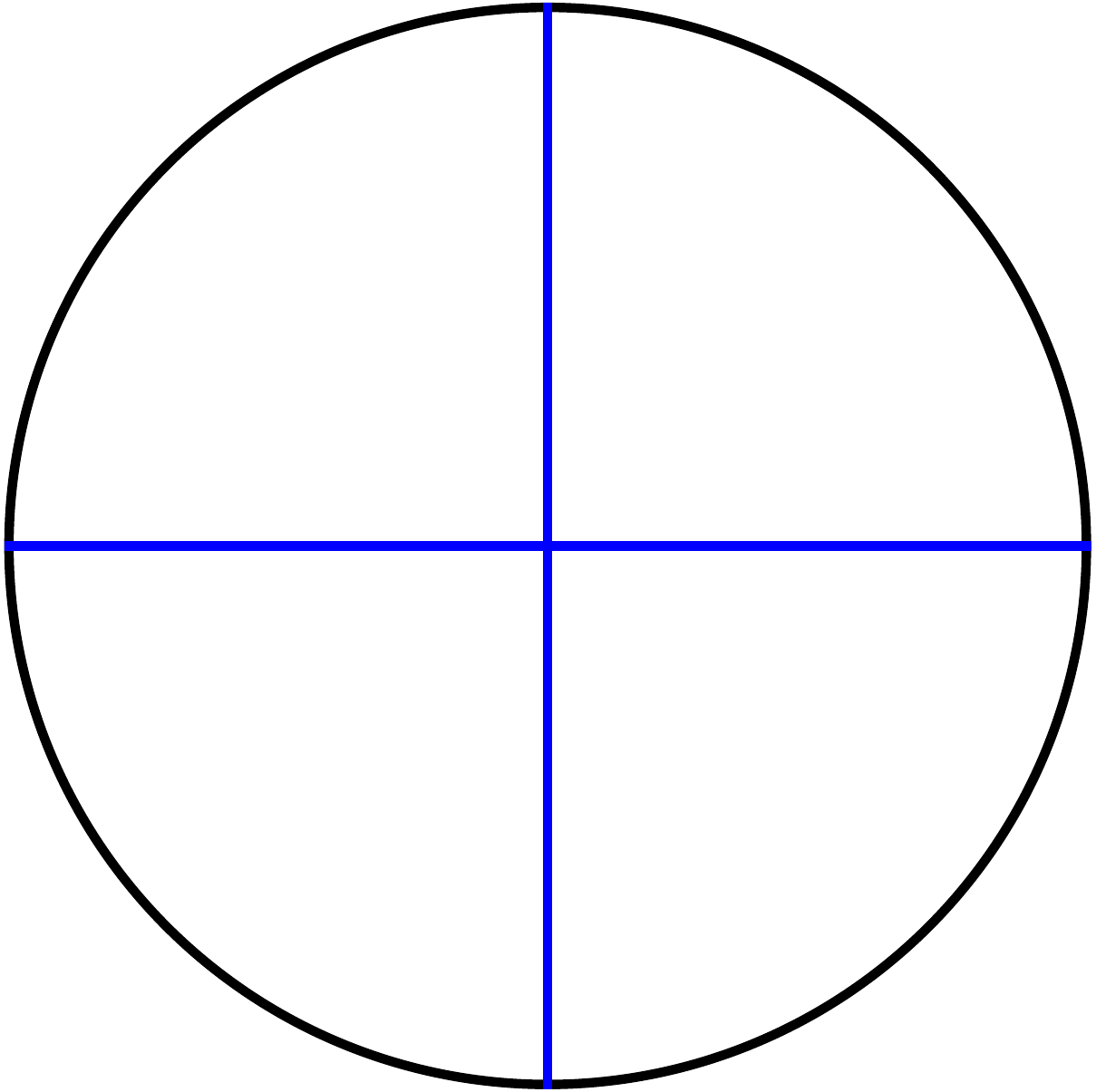}
&\includegraphics[height=1.6cm]{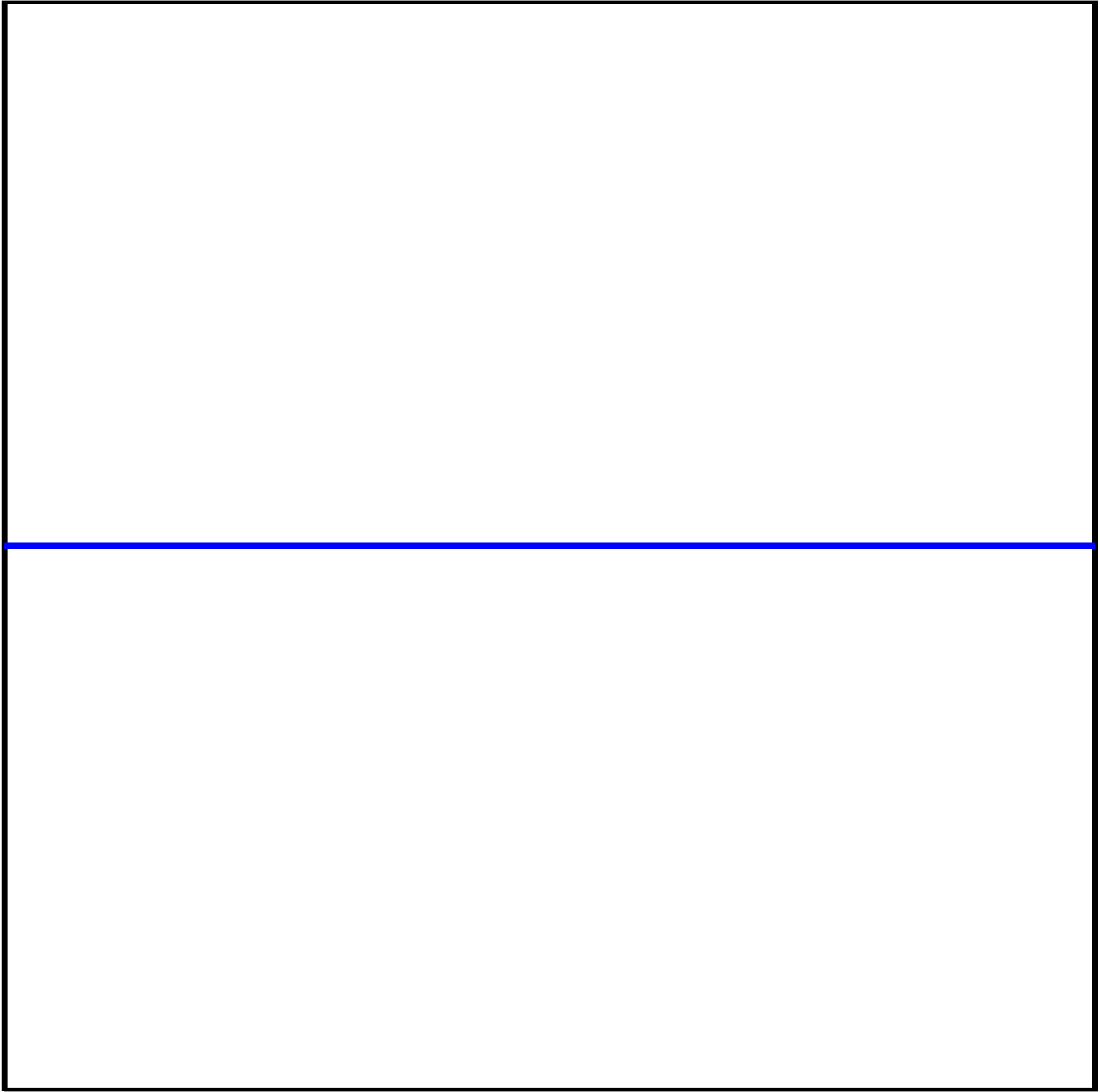}
&\includegraphics[height=1.6cm]{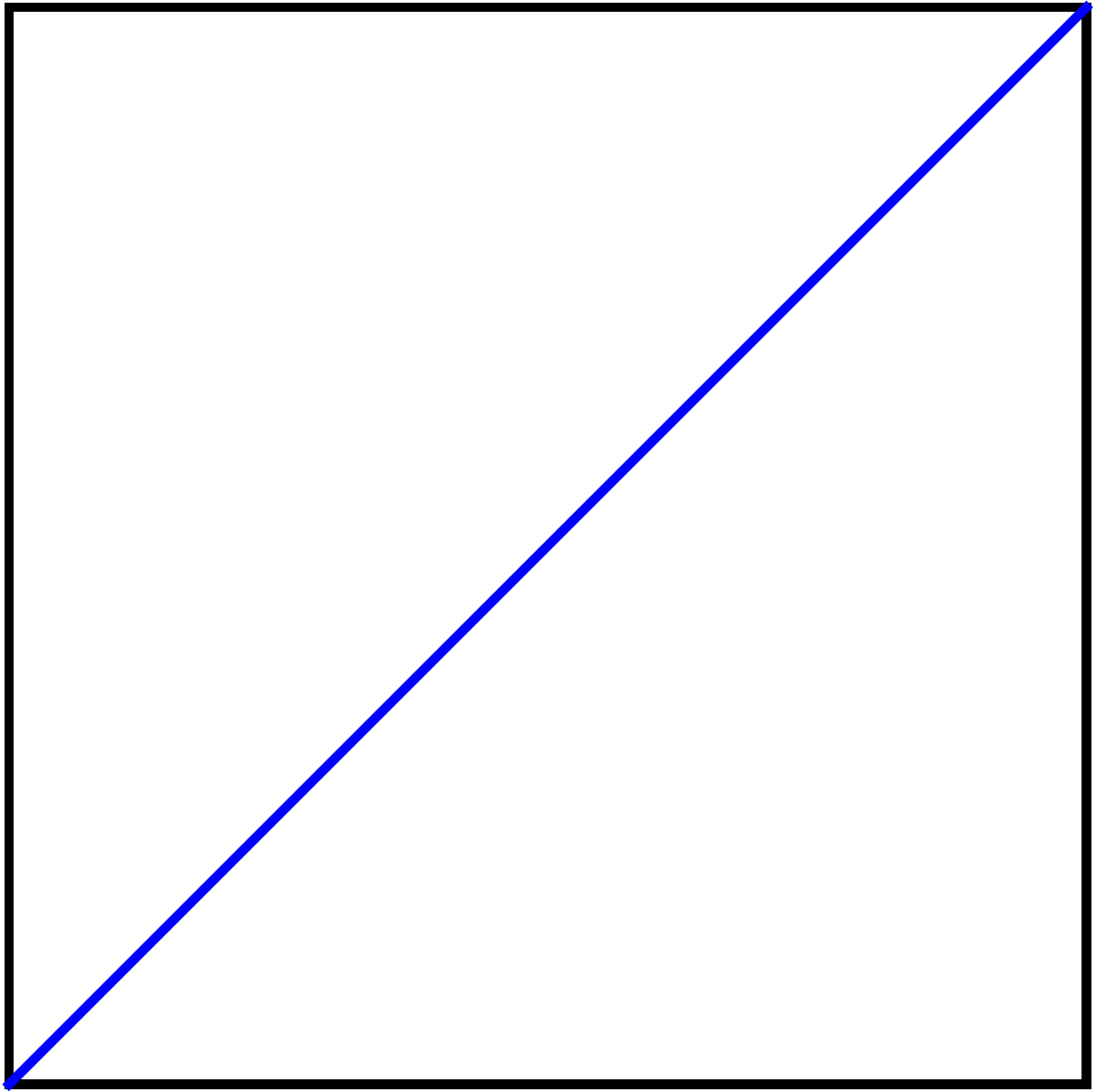}
&\includegraphics[height=1.6cm]{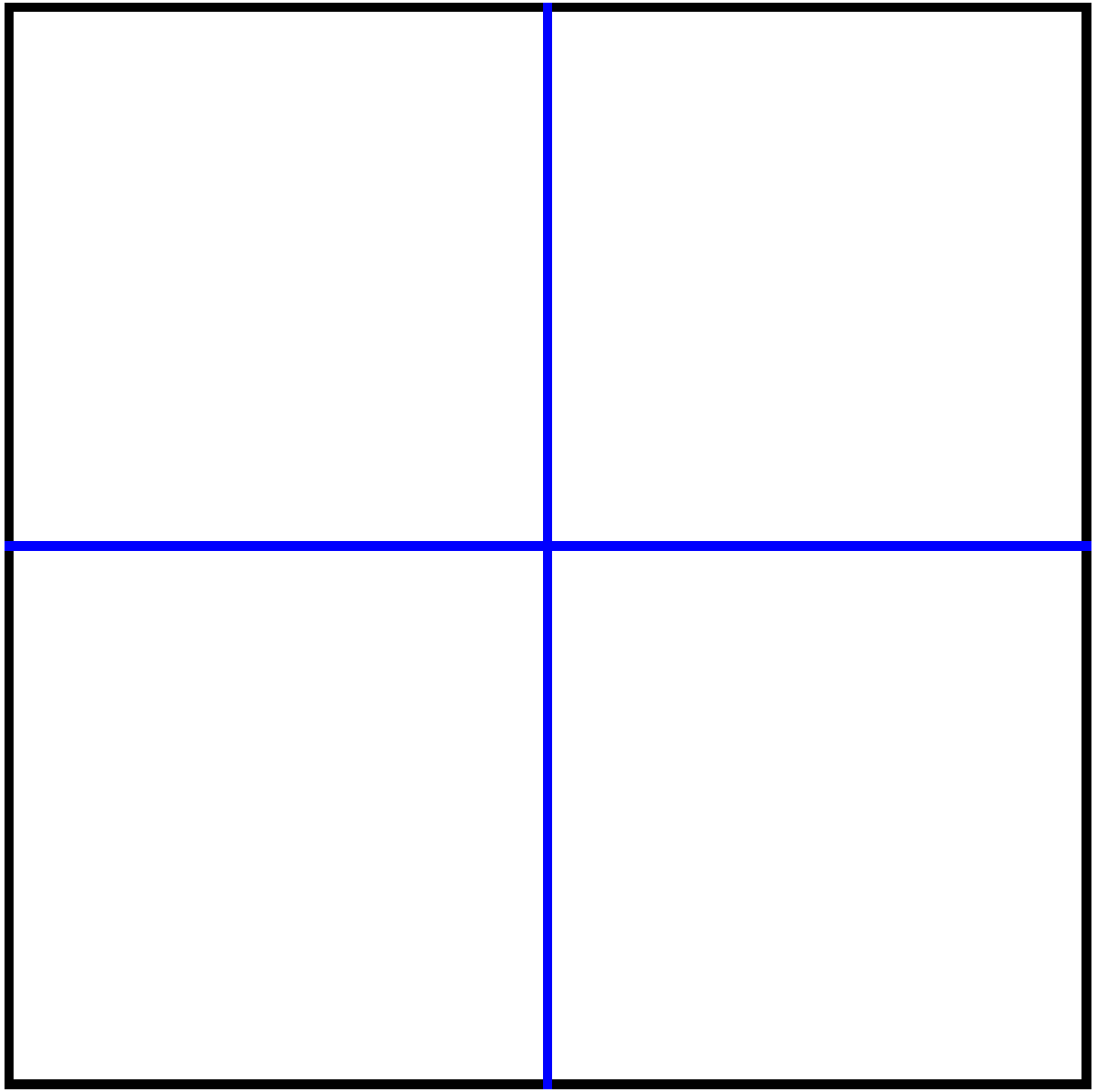}\end{tabular}}
\ \ \subfloat[$3$- and $5$-partition.\label{chapBH.fig.MS}]{\begin{tabular}{p{1.25cm}p{1.25cm}}\includegraphics[height=1.6cm,angle=90]{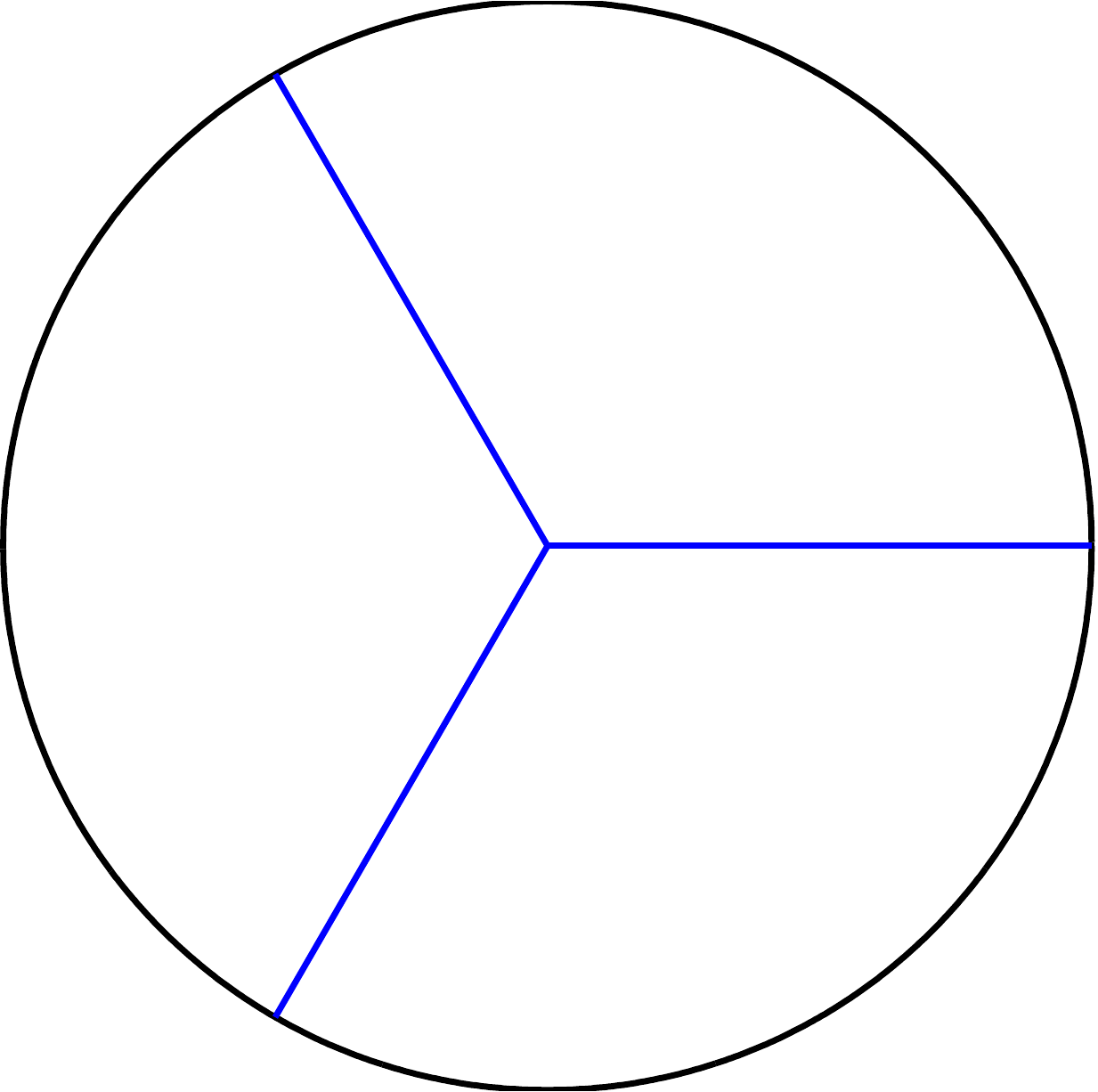}
&\includegraphics[height=1.6cm,angle=90]{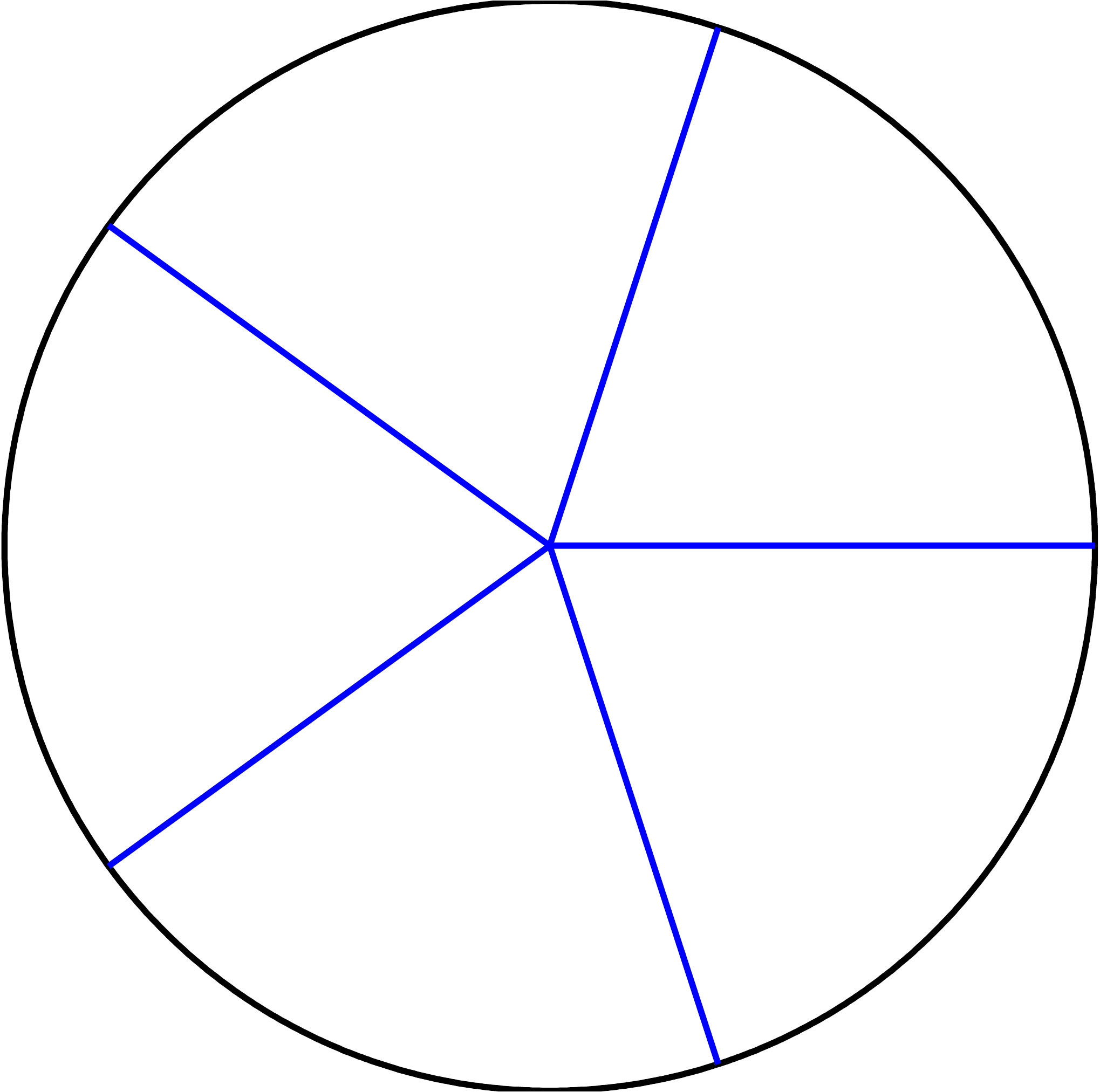}\end{tabular}}
\ \ \subfloat[{$\mathcal D^{\sf perp}$ and $\mathcal D^{\sf diag}$.\label{chapBH.fig.carrecand1}}]{\begin{tabular}{p{1.25cm}p{1.25cm}}\includegraphics[height=1.6cm]{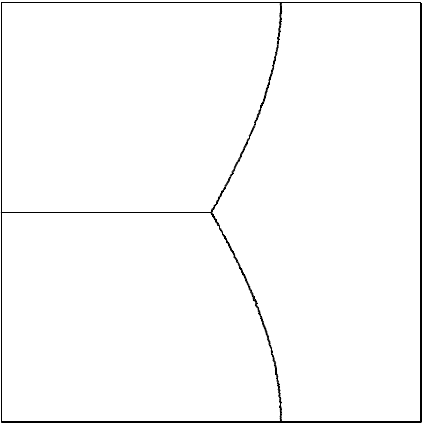}&
\includegraphics[height=1.6cm]{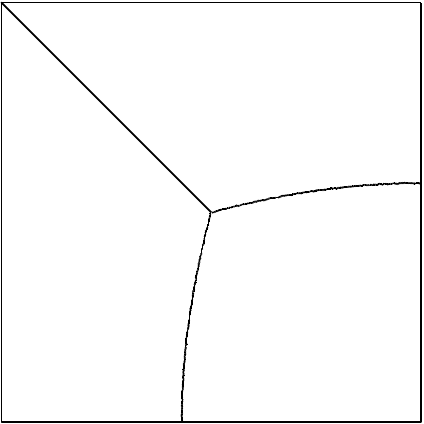}\end{tabular}}
\caption{Partitions of the disk and the square.}
\end{center}
\end{figure}
For other $k$'s, the question is open. Numerical simulations in \cite{MR2836255,MR3093548} permit to exhibit candidates to be minimal $k$-partition of the disk for $k=3,5$. Nevertheless we have no proof that the minimal $ 3$-partition of the disk is the ``Mercedes star'' (see Figure~\ref{chapBH.fig.MS}).

Let us now discuss the $3$-partitions of a square. It is not difficult to see that $\mathfrak L_3$ is strictly less than $ L_3$. 
Numerical computations\footnote{see \href{http://w3.bretagne.ens-cachan.fr/math/simulations/MinimalPartitions/}{\sf http://w3.ens-rennes.fr/math/simulations/MinimalPartitions/}} in \cite{MR2598097} produce natural candidates for a symmetric minimal $3$-partition.
Two candidates $\mathcal D^{\sf perp}$ and $\mathcal D^{\sf diag}$ are obtained numerically by choosing the symmetry axis (perpendicular bisector or diagonal line) and represented in Figure \ref{chapBH.fig.carrecand1}. Numerics suggests that there is no candidate of type [b] or [c], that the two candidates $\mathcal D^{\sf perp}$ and $\mathcal D^{\sf diag}$ have the same energy and that the center is the unique singular point of the partition inside the square. Once this last property is accepted, one can perform the spectral analysis of an Aharonov-Bohm operator\index{Aharonov-Bohm operator} (see Section~\ref{chapBH.s8}) with a pole at the center. This point of view is explored numerically in a rather systematic way by Bonnaillie-No\"el--Helffer \cite{MR2836255} and theoretically by Noris-Terracini \cite{MR2815036} (see also \cite{MR3270167}). 
This could explain why the two partitions $\mathcal D^{\sf perp}$ and $\mathcal D^{\sf diag}$ have the same energy.
\begin{figure}[h!bt]
\begin{center}
\includegraphics[height=1.25cm]{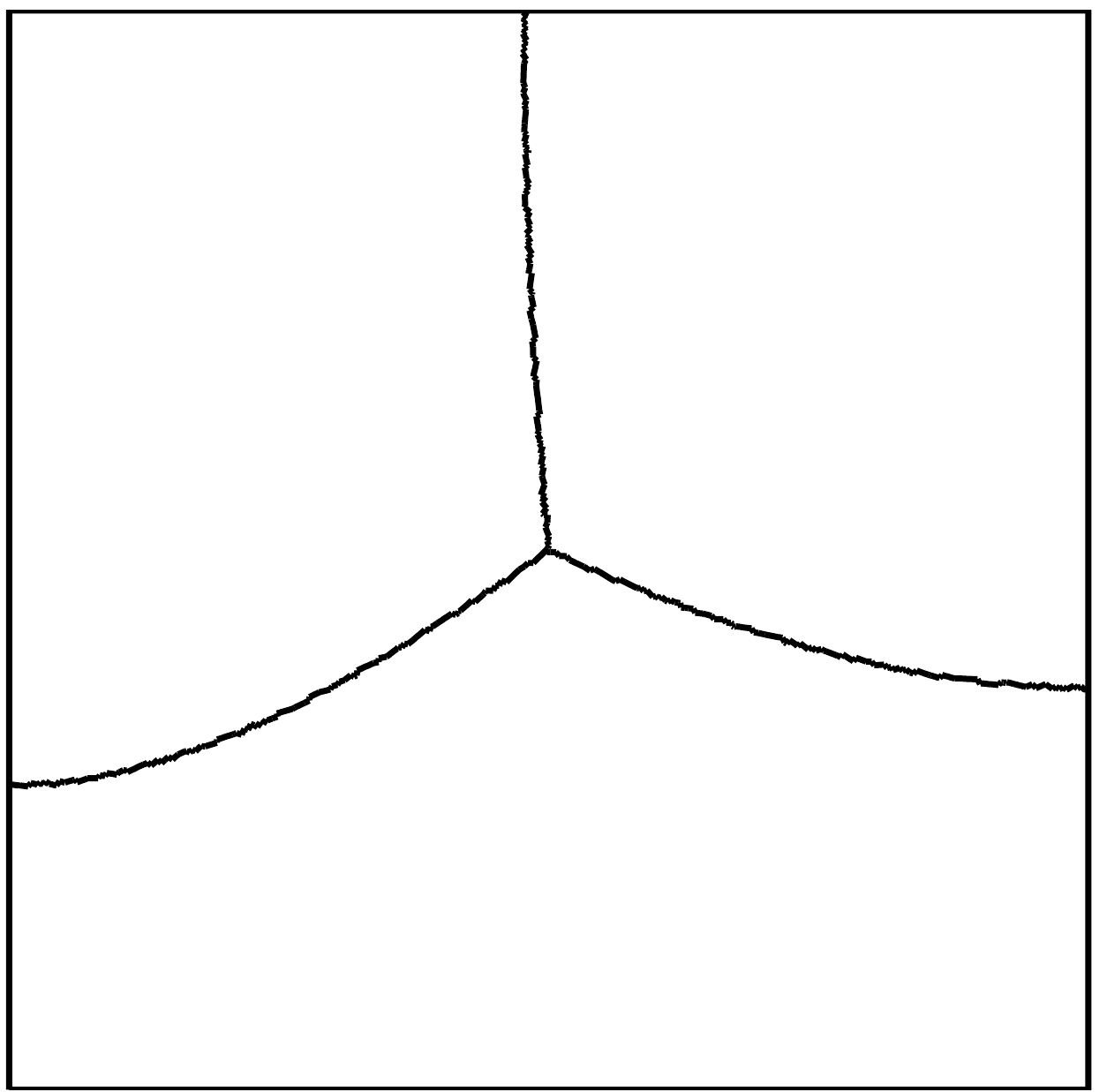}
\includegraphics[height=1.25cm]{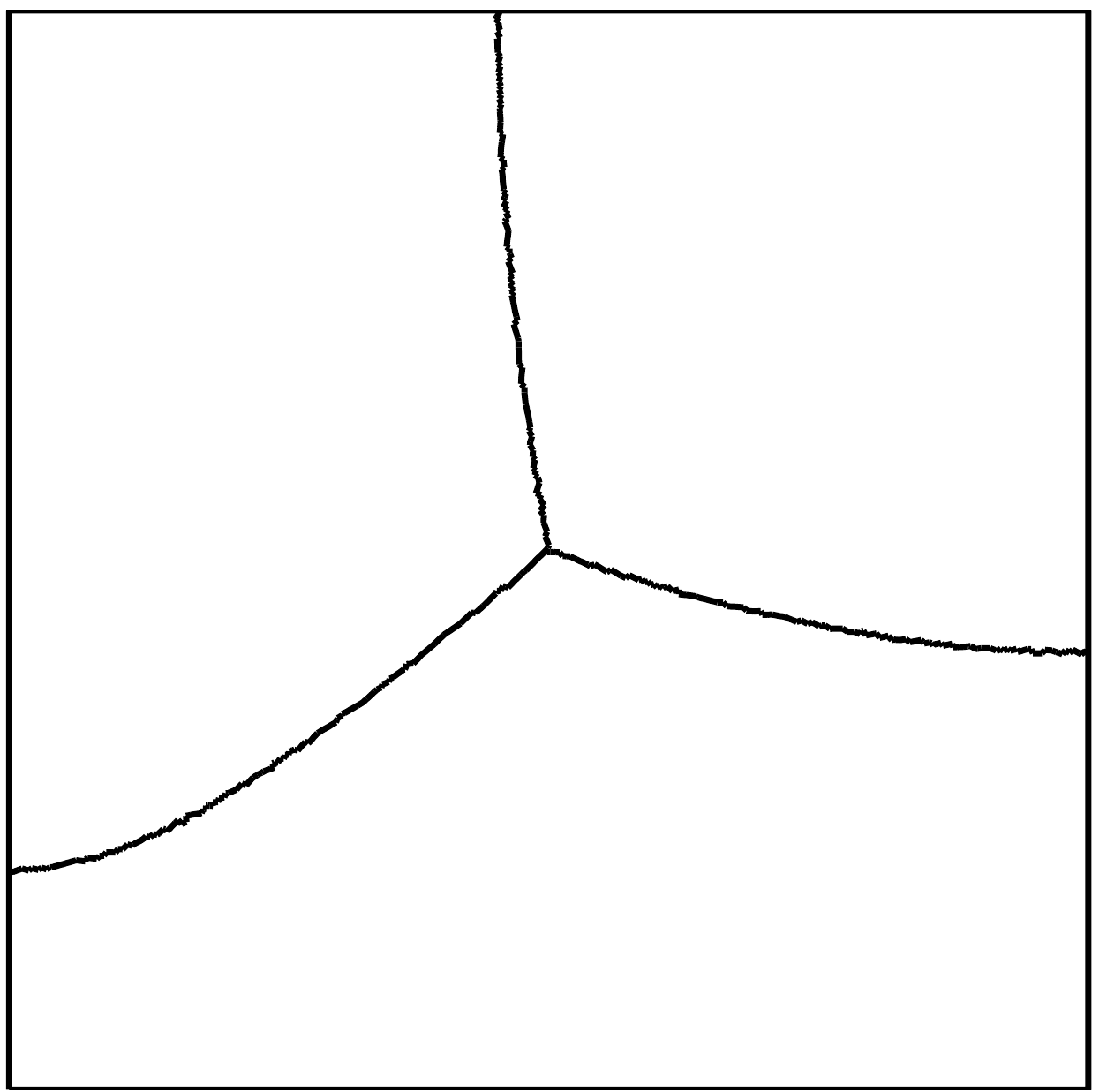}
\includegraphics[height=1.25cm]{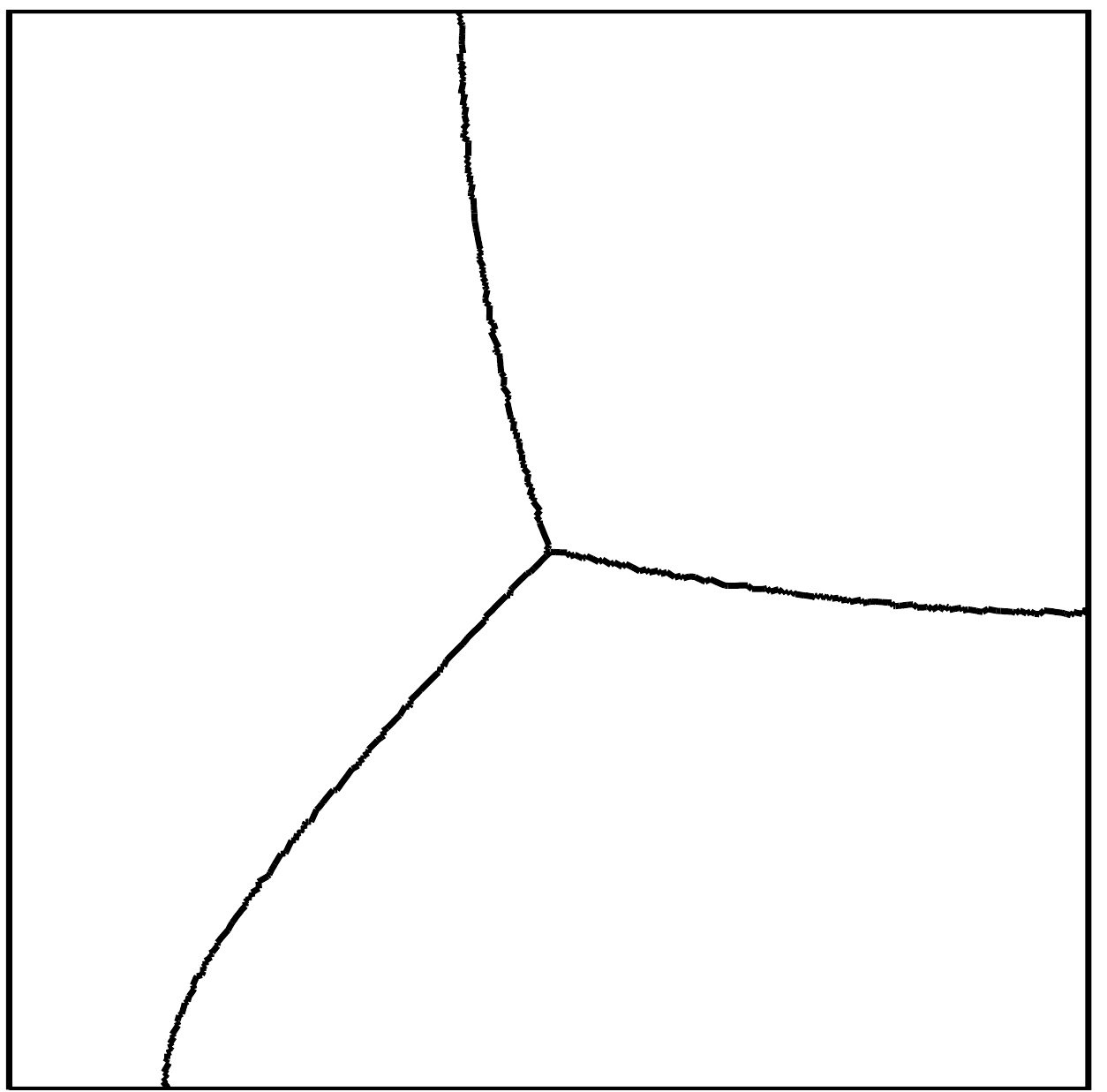}
\includegraphics[height=1.25cm]{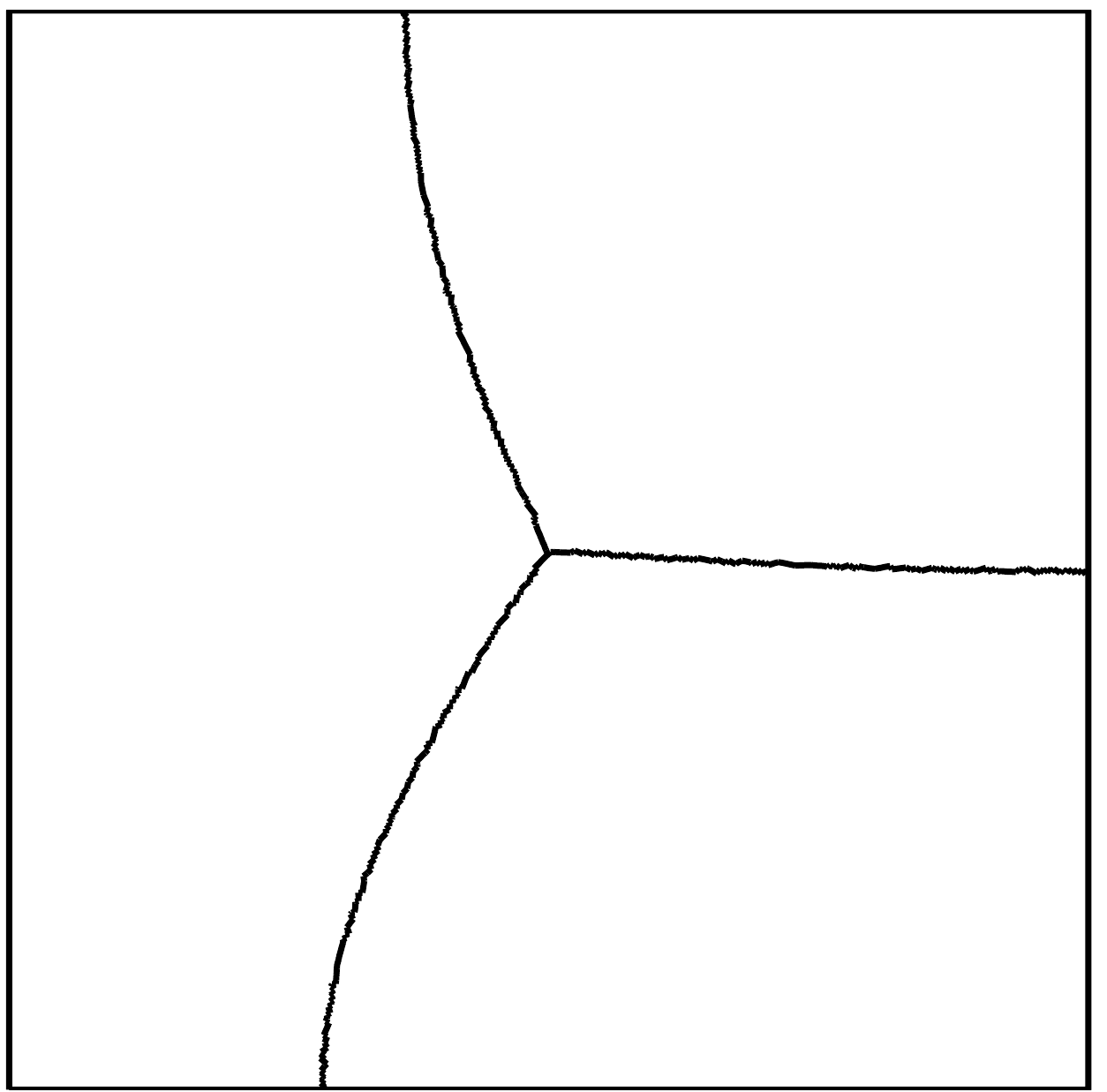}
\includegraphics[height=1.25cm]{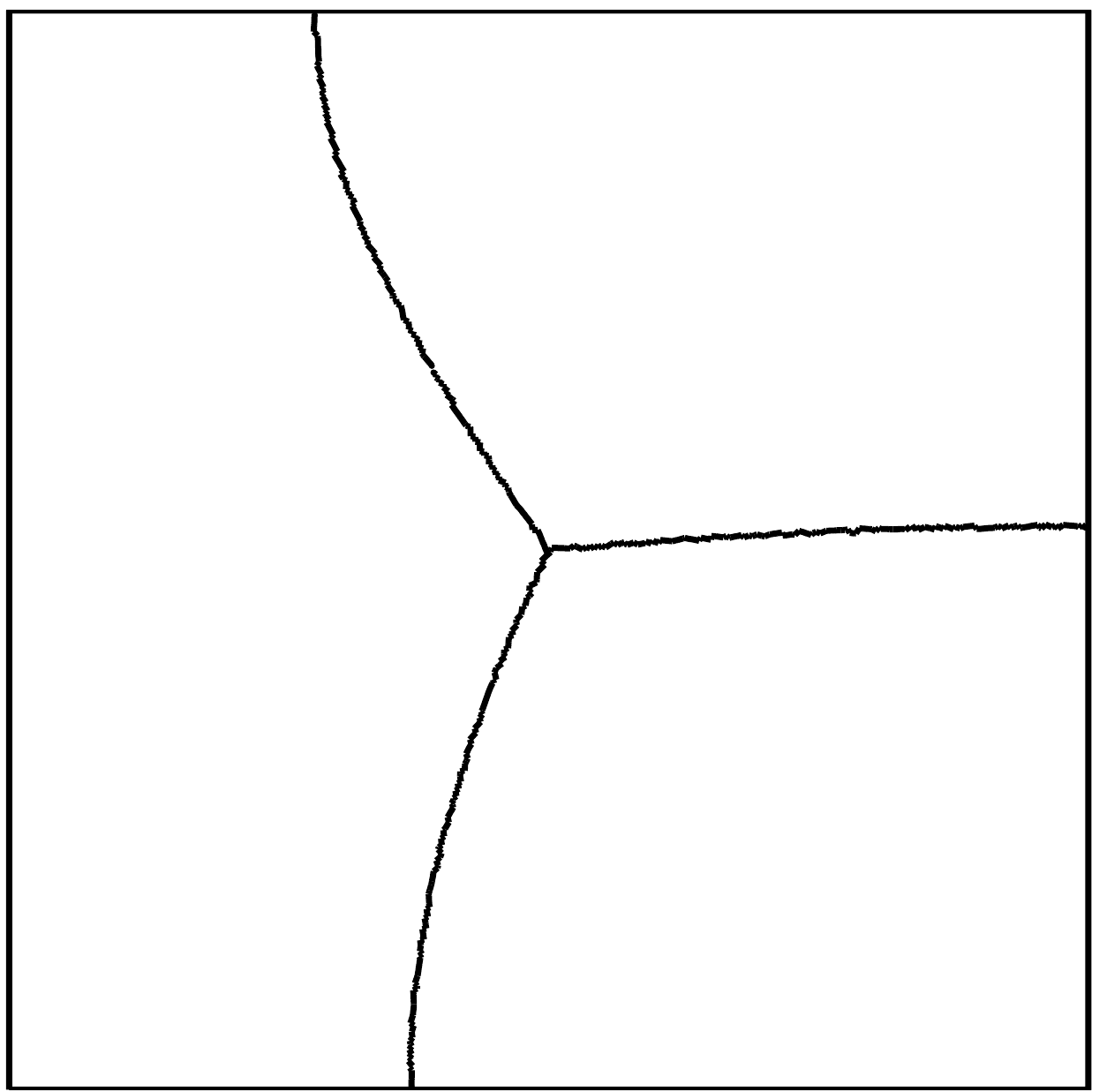}
\includegraphics[height=1.25cm]{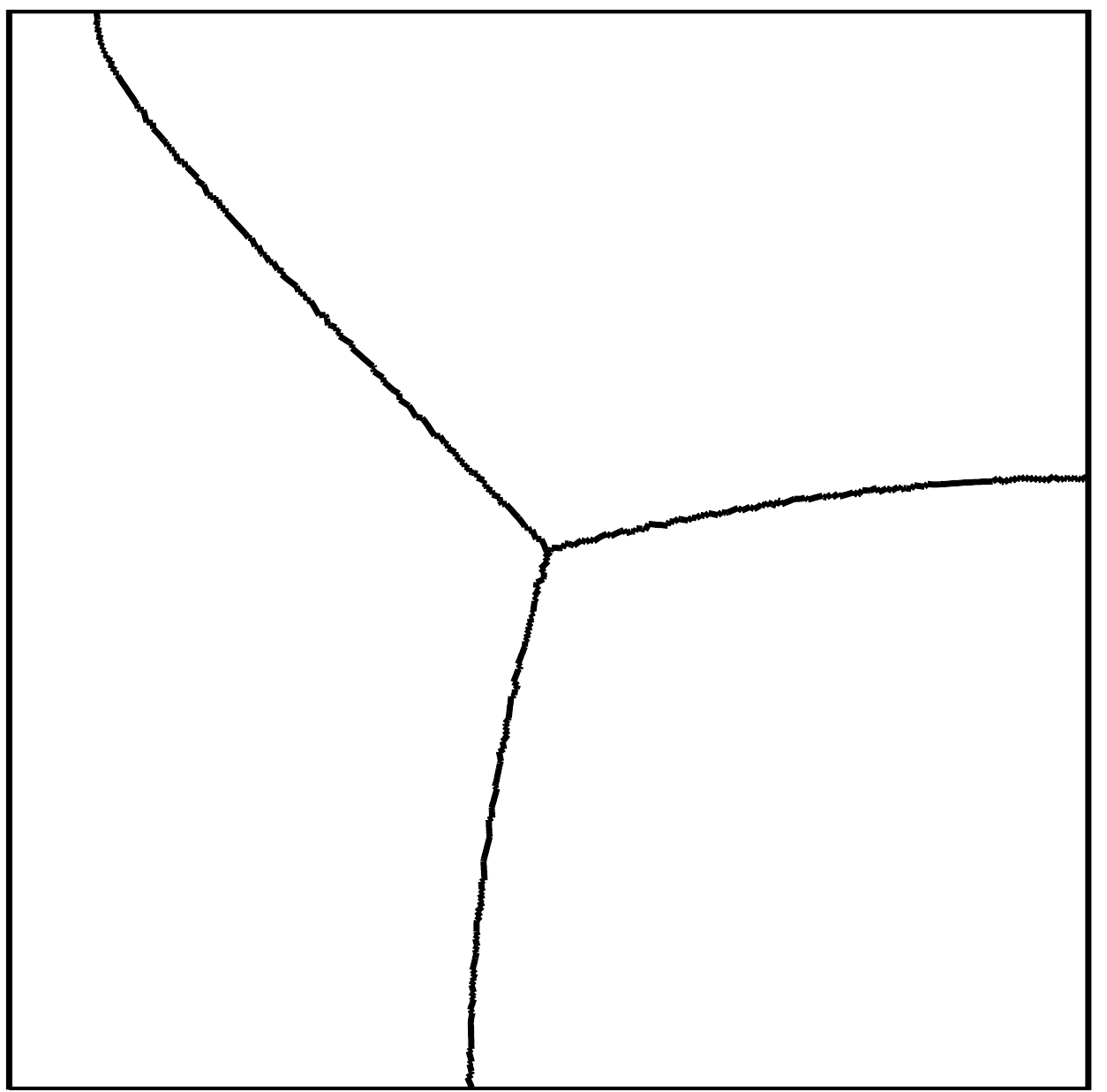}
\includegraphics[height=1.25cm]{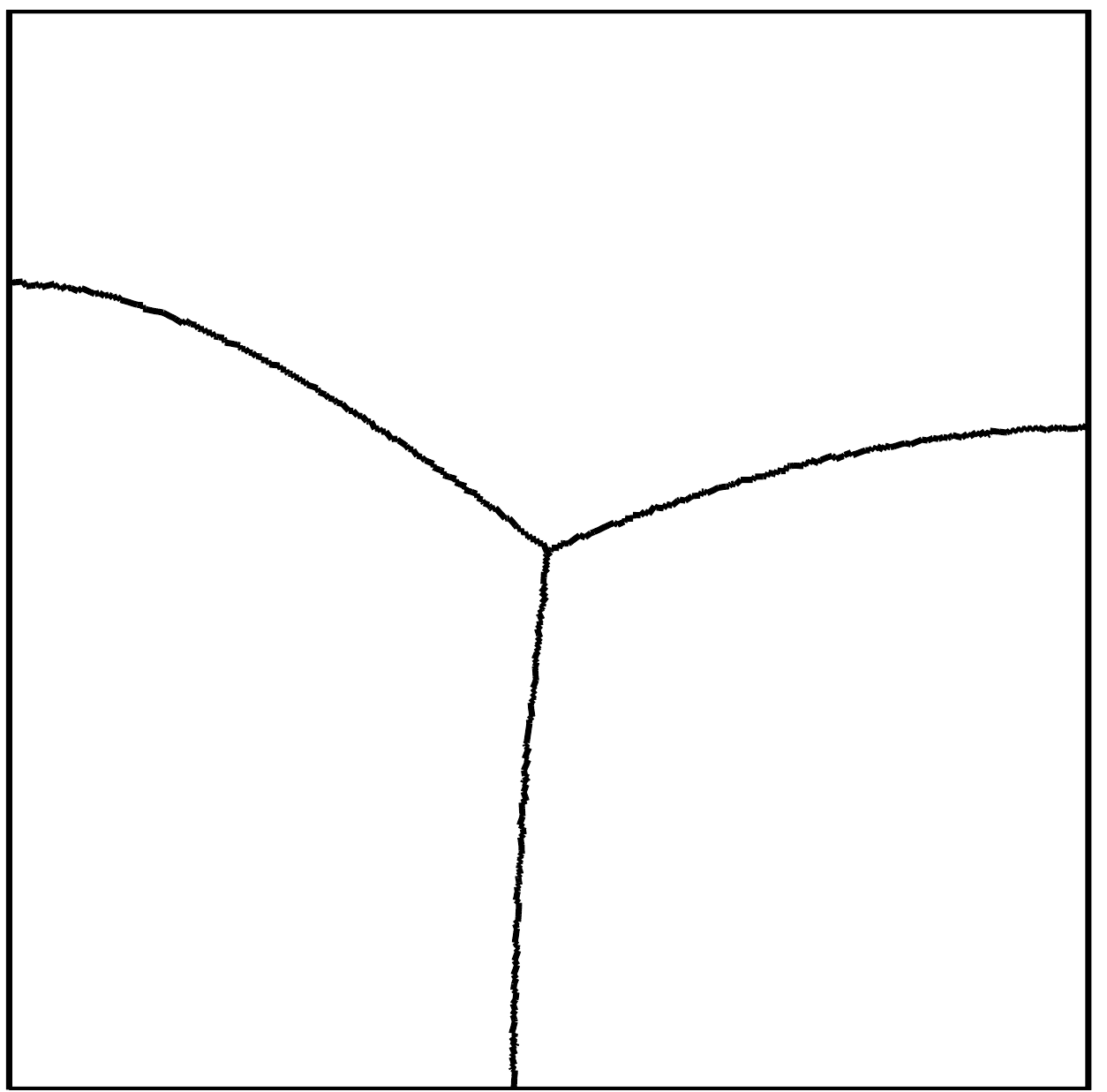}
\includegraphics[height=1.25cm]{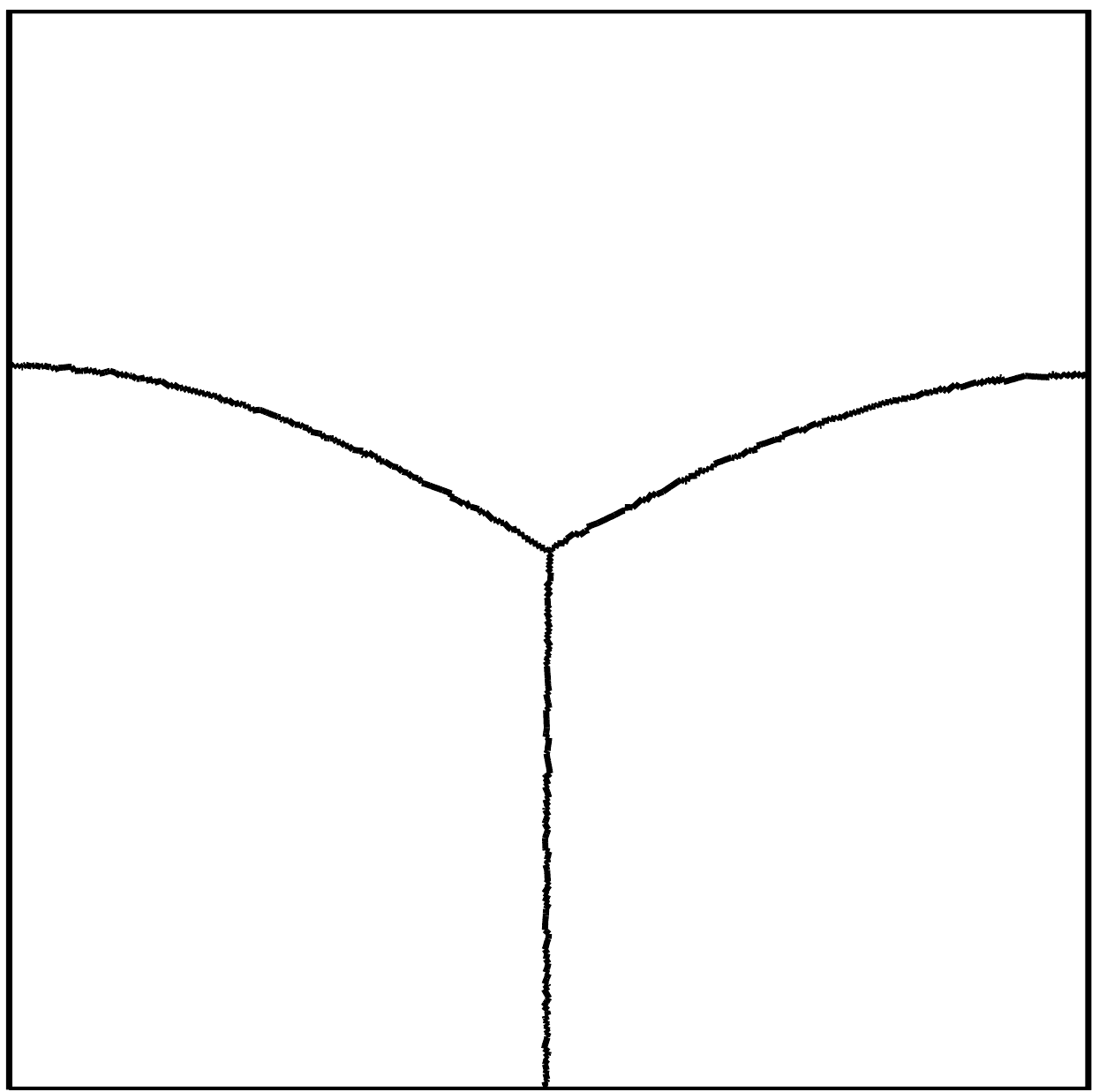}
\includegraphics[height=1.25cm]{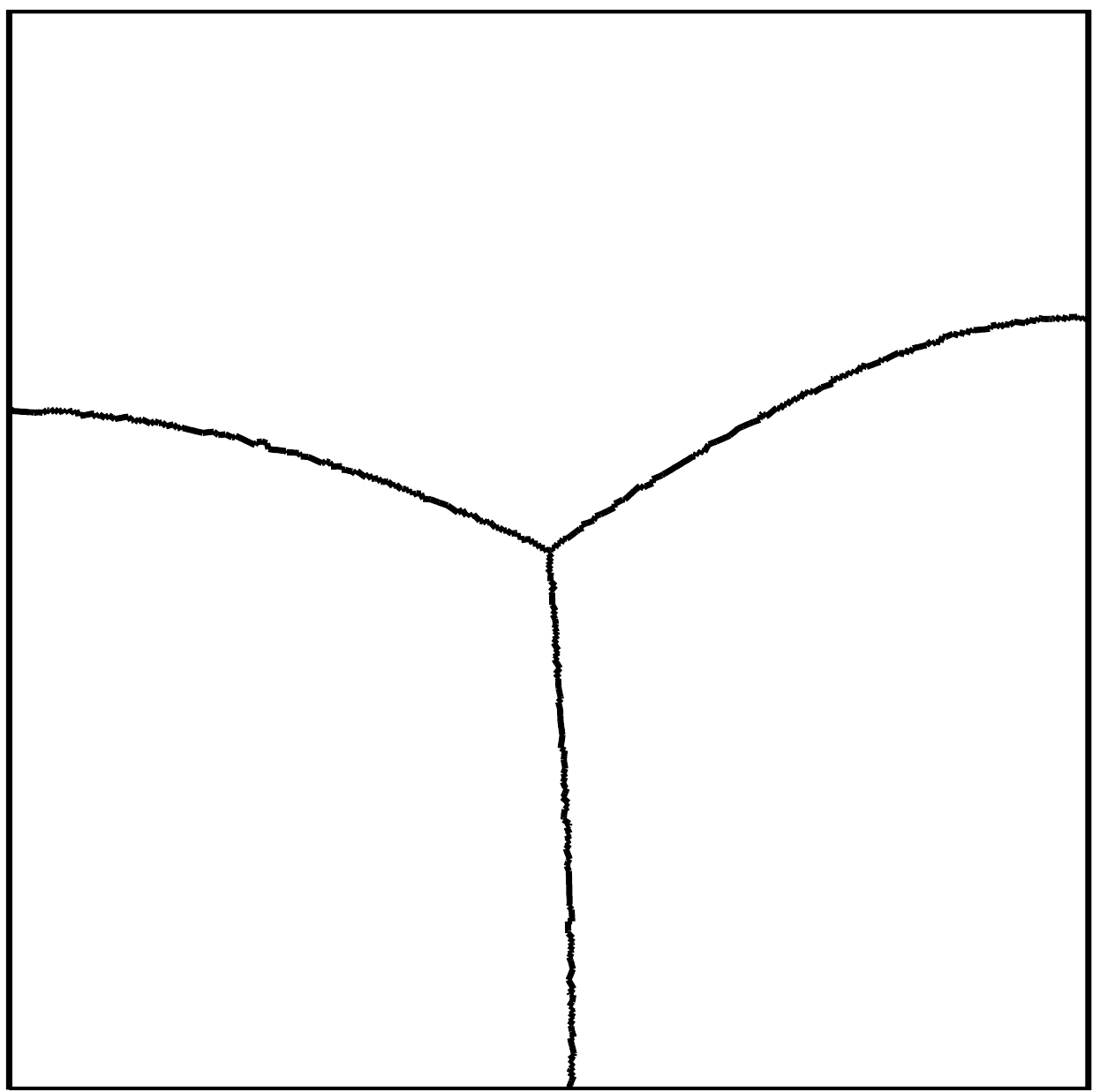}
\caption{A continuous family of $3$-partitions with the same energy.\label{chapBH.fig.famille3part}}
\end{center}
\end{figure}
Moreoever this suggests that there is a continuous family of minimal $3$-partitions of the square. This is illustrated in Figure~\ref{chapBH.fig.famille3part}. In the formalism of the Aharonov-Bohm operator, the basic remark is that this operator has an eigenvalue of multiplicity $2$ when the pole is at the center. We refer to \cite{MR2591197, MR2836255} for further discussion.

\section{Aharonov-Bohm operators and minimal partitions }\label{chapBH.s8}
The introduction of Aharonov-Bohm operators in this context is an example of ``physical mathematics''. There is no magnetic field in our problem and it is introduced artificially. But the idea comes from \cite{MR1690957}, which was motivated by a problem in superconductivity in non simply connected domains introduced by Berger and Rubinstein in \cite{MR1690956}.

\subsection{Aharonov-Bohm effect}
The Aharonov-Bohm effect \cite{MR0110458} is one of the basic effects explained by quantum mechanics but usually refers to an experiment related to scattering theory. According to Google Scholar, there is a huge literature in Mathematics devoted to the analysis of this effect
 starting with Ruisjenaars \cite{MR701261}.  Another effect is related to bound states and gives in some sense a refined version of the diamagnetic effect.
In the non simply connected 2D-cases, when the magnetic field is identically $0$,  the circulations (modulo $2\pi \mathbb Z^d$) around each hole appear consequently as  the unique relevant quantities. The limiting case when the holes are points will be our most important case. If we consider in $\Omega$ the Dirichlet realization $H_{{\bf A},V}$ of 
 $$
 \sum_j (D_{x_j} -A_j)^2 + V\,,\quad\mbox{ with }D_{x_j} = -i \partial_{x_j}\,, 
$$
the celebrated diamagnetic inequality due to Kato says:
\begin{equation}\label{diamD}
\inf \;\sigma \left(H_{{\bf A},V}  \right)  \geq \inf \;\sigma \left( H_{{\bf 0},V} \right)\;.
\end{equation}
This inequality admits a kind of converse, showing its optimality (Lavine-O'Caroll-Helffer) (see the presentation in \cite{MR2662319}).
\begin{proposition}\label{Prop2.6}~\\
Suppose that $\Omega \Subset \mathbb R^2$,  ${\bf A}\in C^1(\overline{\Omega})$ and  $V\in L^\infty(\Omega)$.
Let $\lambda_{{\bf A},V}$ be the ground state of $H_{{\bf A},V}$, then
 the three properties are equivalent
\begin{enumerate}
\item $H_{{\bf A},V}$ and $H_{{\bf 0},V}$ are unitary equivalent ; 
\item $\lambda_{{\bf A},V}=\lambda_{{\bf 0},V}$ ;
\item ${\bf A}$ satisfies the two conditions 
${\rm curl}\  {\bf A}=0\;$ and
$\frac{1}{2\pi} \int_\gamma {\bf A} \in \mathbb Z$
 on any closed path $\gamma$ in $\Omega\,$, where $\int_\gamma {\bf A}$ denotes the circulation of ${\bf A}$ along $\gamma\,$. 
\end{enumerate}
\end{proposition}

\subsection{Aharonov-Bohm operators} 
Let $\Omega$ be a planar domain and ${\bf p}=(p_{1},p_{2})\in\Omega$. Let us consider the Aharonov-Bohm Laplacian\index{Aharonov-Bohm operator} in a punctured domain $\dot\Omega_{{\bf p}}:=\Omega\setminus\{{\bf p}\}$ with a singular magnetic potential and normalized flux $\alpha$. We first introduce
$${\mathbf A}^{\bf p}({\bf x})=(A_{1}^{\bf p}({\bf x}),A_{2}^{\bf p}({\bf x})) = \frac{({\bf x}-{\bf p})^\perp}{|{\bf x}-{\bf p}|^2},\qquad\mbox{ with }\quad {\bf y}^\perp=(-y_{2},y_{1})\,.$$
This magnetic potential satisfies
$${\rm curl}\ {\mathbf A}^{\bf p}({\bf x})=0\quad\mbox{ in } \dot\Omega_{{\bf p}}\,.$$
If ${\bf p} \in \Omega$, its circulation along a path of index $1$ around ${\bf p}$ is $2\pi$ (or the flux created by ${\bf p}$). If ${\bf p} \not\in \Omega$, $\mathbf A^{\bf p} $ is a gradient and the circulation along any path in $\Omega$ is zero. From now on, we renormalize the flux by dividing it by $2\pi$.\\
The Aharonov-Bohm Hamiltonian with singularity ${\bf p}$ and flux $\alpha$, written for brevity $H^{AB}(\dot \Omega_{{\bf p}},\alpha)$, is defined by considering the Friedrichs extension starting from $ C_0^\infty(\dot \Omega_{{\bf p}})$ and the differential operator
\begin{equation}
-\Delta_{{\boldsymbol{\alpha}} {\bf A}^{\bf p}} := (D_{x_{1}} - \alpha A_{1}^{\bf p})^2 + (D_{x_{2}}-\alpha A_{2}^{\bf p})^2\,. 
\end{equation}
This construction can be extended to the case of a configuration with $\ell$ distinct points ${\bf p}_1,\dots, {\bf p}_\ell$ (putting a flux $\alpha_{j}$ at each of these points). We just take as magnetic potential
$$ {\bf A}_{\boldsymbol{\alpha}}^{\bf P} = \sum_{j=1}^\ell \alpha_j {\mathbf A}^{{\bf p}_j}\,, \qquad\mbox{ where }\quad {\bf P}=({\bf p}_1,\dots,{\bf p}_\ell)\quad\mbox{ and }\quad{\boldsymbol{\alpha}}=(\alpha_{1},\ldots,\alpha_{\ell}),$$
and consider the operator in $\dot \Omega_{\bf P}:=\Omega \setminus \{{\bf p}_1,\dots,{\bf p}_\ell \}$. 
Let us point out that the ${\bf p}_j$'s can be in $\mathbb R^2\setminus \Omega$, and in particular in $\partial \Omega$. It is important to observe that
if ${\boldsymbol{\alpha}} ={\boldsymbol{\alpha}}'$ modulo $\mathbb Z^\ell $, then $H^{AB}(\dot \Omega_{{\bf P}},{\boldsymbol{\alpha}})$ and $H^{AB}(\dot \Omega_{{\bf P}},{\boldsymbol{\alpha}}')$ are unitary equivalent.

\subsection{The case when the fluxes are $1/2$}
Let us assume for the moment that there is a unique pole $\ell=1$ and suppose that the flux $\alpha$ is $1/2$. For brevity, we omit $\alpha$ in the notation when it equals $1/2$. Let $ K_{{\bf p}}$ be the antilinear operator 
\begin{equation}\label{defKp}
 K_{{\bf p}} = {\rm e}^{i \theta_{{\bf p}}} \; \Gamma\,, 
 \end{equation}
 where $\Gamma$ is the complex conjugation operator $\Gamma u = \bar u\,$ and 
$\theta_{\bf p}$ is such that
$d\theta_{\bf p}= 2 {\bf A} ^{\bf p}\,.$\\
We note that, because the normalized flux of $2 {\bf A} ^{\bf p}$ belongs to $\mathbb Z$ for any path in $\dot \Omega_{{\bf p}}$, the function  ${\bf x} \mapsto \exp i \theta_{\bf p}({\bf x})$ is $C^\infty$. A function $ u$ is called $K_{{\bf p}}$-real, if $K_{{\bf p}} u =u\,.$ The operator  $H^{AB}(\dot \Omega_{{\bf p}})=H^{AB}(\dot \Omega_{{\bf p}},\frac 12)$ is preserving the $ K_{{\bf p}}$-real functions. Therefore we can consider a basis of $K_{{\bf p}}$-real eigenfunctions. Hence we only analyze the restriction of $H^{AB}(\dot \Omega_{{\bf p}})$ to the $ K_{{\bf p}}$-real space $ L^2_{K_{{\bf p}}}$ where
$$
 L^2_{K_{{\bf p}}}(\dot{\Omega}_{{\bf p}})=\{u\in L^2(\dot{\Omega}_{{\bf p}}) \;:\; K_{{\bf p}}\,u =u\,\}\,.
$$
If there are several poles ($\ell>1$) and ${\boldsymbol{\alpha}} = (\frac 12,\dots, \frac 12)$, we can also construct the antilinear operator $ K_{\bf P}$, where $\theta_{\bf p}$  in \eqref{defKp} is replaced by
\begin{equation}\label{chapBH.defTheta}
\Theta_{{\bf P}} = \sum_{j=1}^\ell \theta_{{\bf p}_j}\,.
\end{equation}

\subsection{Nodal sets of $K_{{\bf P}}$-real eigenfunctions\index{nodal domains}}
As mentioned previously under the half-integer flux condition, we can find a basis of $K_{\bf P}$-real eigenfunctions. It was shown in \cite{MR1690957} and \cite{MR1994668} that the $ K_{{\bf P}}$-real eigenfunctions have a regular nodal set (like the eigenfunctions of the Dirichlet Laplacian) with the exception that, at each singular point ${\bf p}_j$ ($j=1,\dots,\ell$), an odd number $\nu({\bf p}_j)$ of half-lines meet. 
So the only difference with the notion of regularity introduced for minimal partitions is that some $\nu({\bf p}_j)$ can be equal to $1$.\\
The ground state of $H^{AB}(\dot \Omega_{{\bf P}})$ has a very particular structure (see \cite{MR1690956}, \cite{MR1690957} and \cite{HHO2O}).
\begin{proposition}[Slitting property]
If $\mathcal N$ denotes the  zero set of a $K_{{\bf P}}$-real eigenfunction of $H^{AB}(\dot \Omega_{{\bf P}})$ corresponding to the lowest eigenvalue, then 
$\overline{\Omega} \setminus \mathcal N$ is connected. 
\end{proposition}
This is illustrated in Figure~\ref{fig.AB}. In particular, in the case of one pole, the proposition says that  the zero set of a $K_{{\bf P}}$-real  groundstate consists of a line joining the pole and the exterior boundary.
\begin{figure}[h!t]
\begin{center}
\includegraphics[width=2.5cm]{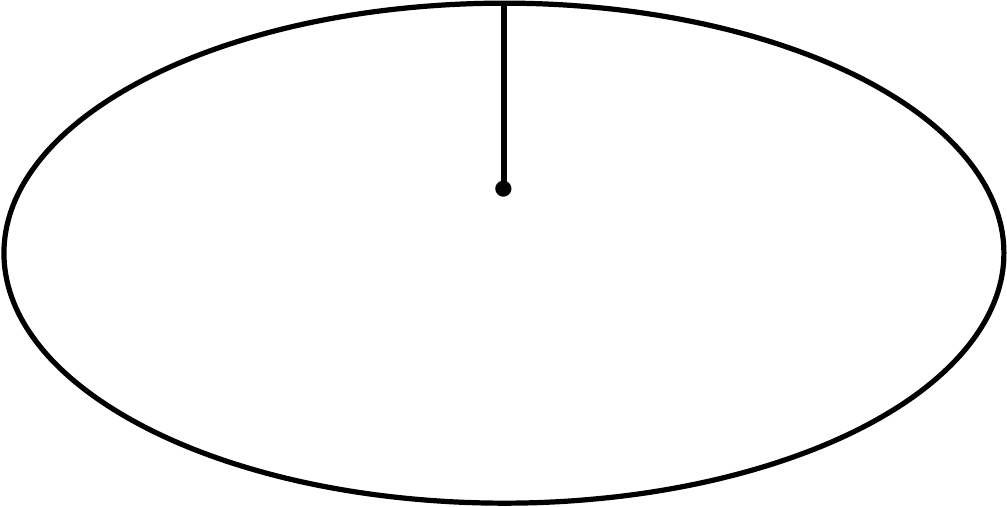}
\includegraphics[width=2.5cm]{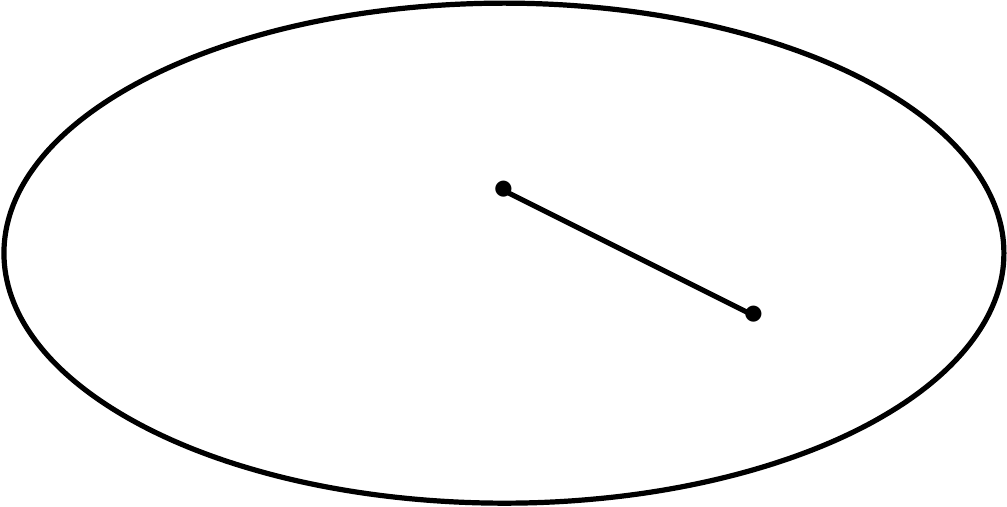}
\includegraphics[width=2.5cm]{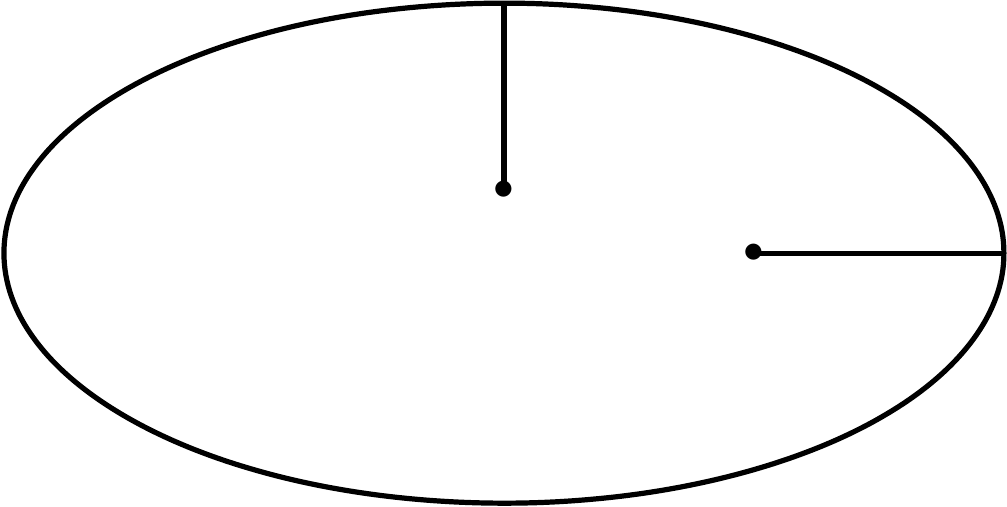}
\includegraphics[width=2.5cm]{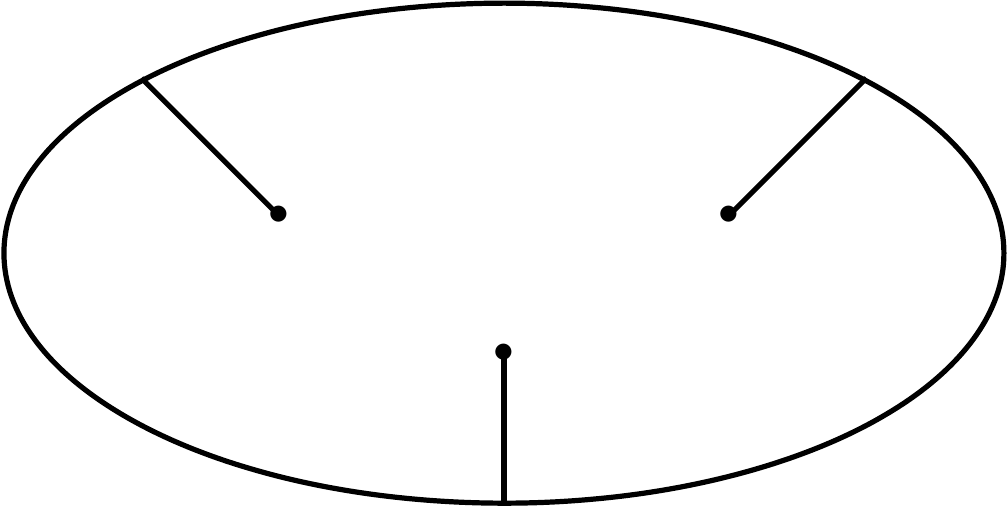}
\includegraphics[width=2.5cm]{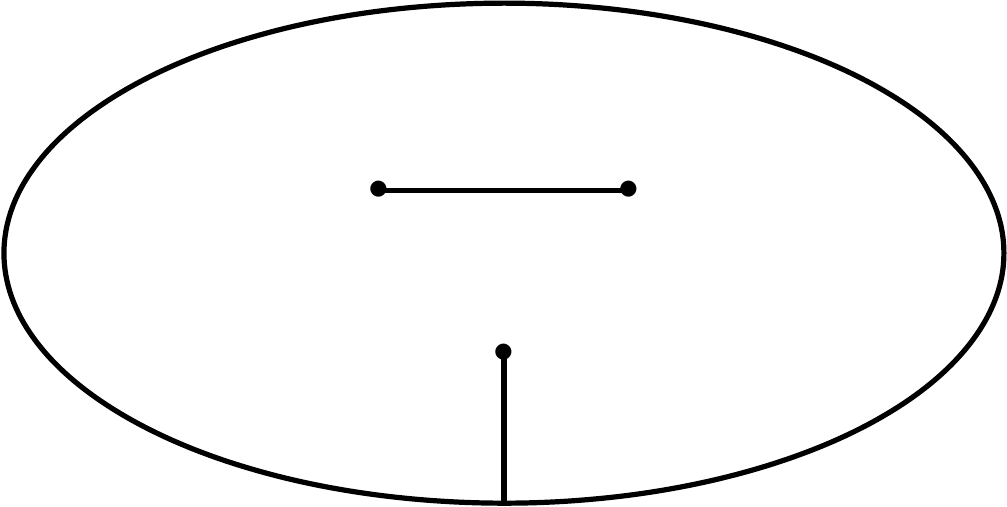}
\includegraphics[width=2.5cm]{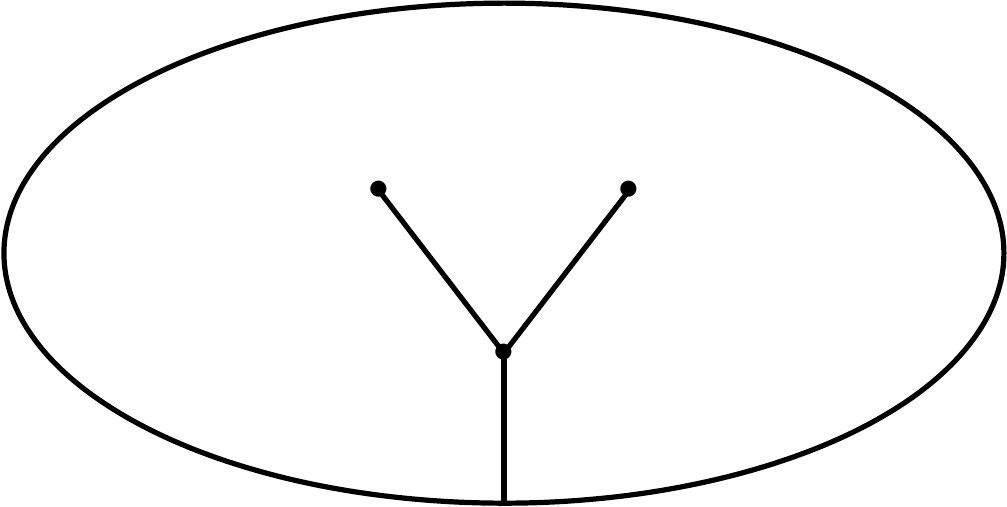}
\caption{Possible topological types of nodal sets in function of the number $\ell $  of poles  ($\ell =1,2,3$).\label{fig.AB}}
\end{center}
\end{figure}

It is also proven in \cite{MR1690957}
\begin{proposition}[Multiplicity]
   The multiplicity $m$ of the first eigenvalue  satisfies
    \begin{equation}\label{eqn:mult}
      m\le\begin{cases}
        2\,,&\mbox{ for } \ell=1,2\,,\\
        \ell\,,& \mbox{ for } \text{$\ell $ odd,  $\ell \geq3\,$},\\
        \ell -1\,,\quad& \mbox{ for }  \text{$\ell $ even, $\ell \geq 4\,$}.
      \end{cases}
    \end{equation}
\end{proposition}
 These two propositions are also true in the case of the Dirichlet realization of  a Schr\"odinger operator  in the form $H^{AB}(\dot \Omega_{{\bf P}}) + V$.  They are  actually also proved \cite{HHO2O}  for the Neumann problem and in the case of more general holes.\\

Coming back to more general eigenfunctions,  we have:
\begin{proposition}
The zero set of a $K_{{\bf P}}$-real eigenfunction of $H^{AB}(\dot \Omega_{{\bf P}})$ is the boundary set of a regular partition if and only if $\nu({\bf p}_j) \geq 2$ for $j=1,\dots, \ell$.
\end{proposition}
Let us illustrate the case of the square with one singular point. Figure~\ref{chapBH.fig.exvecpABcarre} gives the nodal lines of some eigenfunctions of the Aharonov-Bohm operator: there are always one or three lines ending at the singular point (represented by a red point). Note that only the fourth picture gives a regular and nice partition.

\begin{figure}[h!t]
\begin{center}\begin{tabular}{cccccccc}
\includegraphics[height=1.5cm]{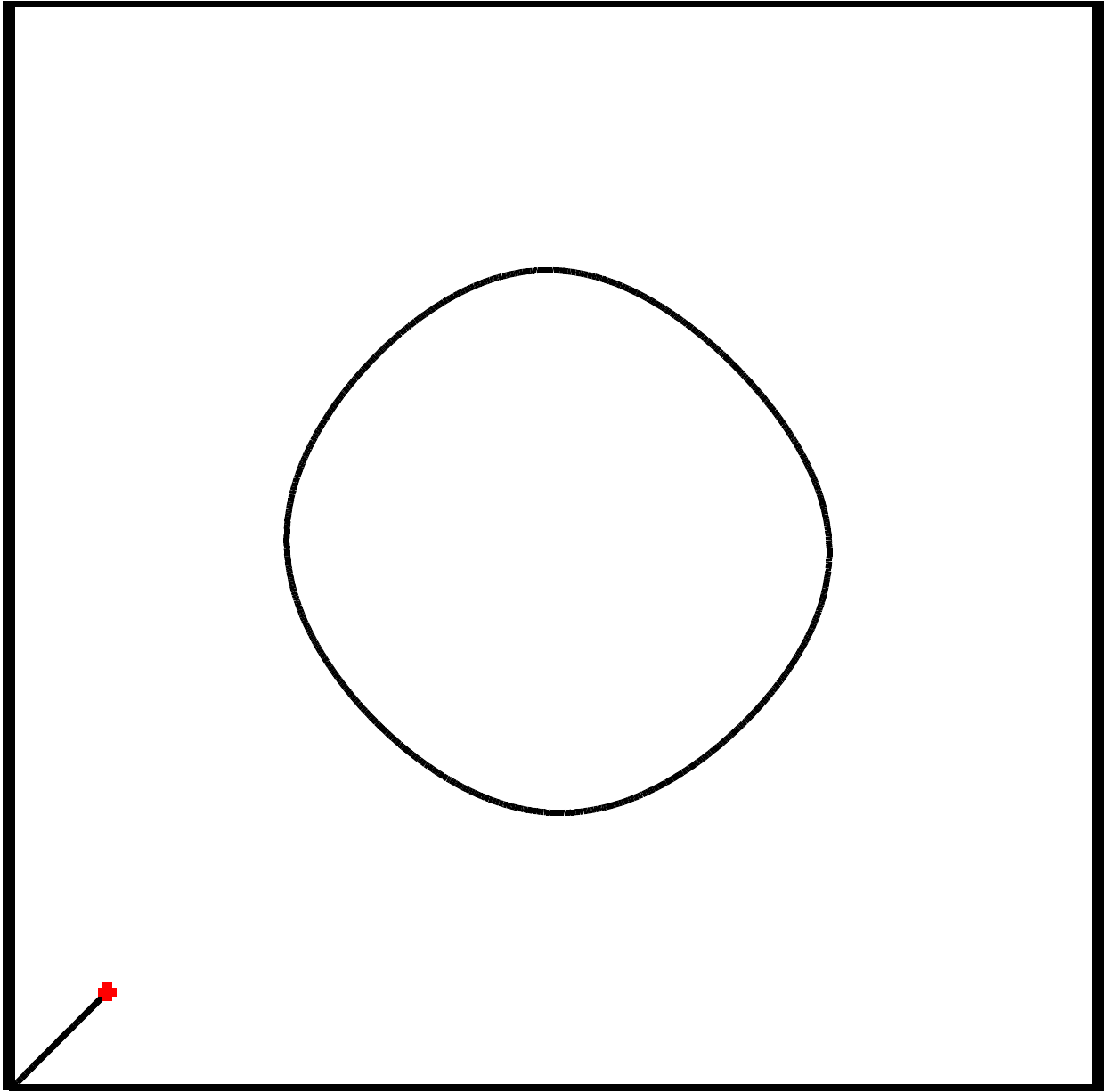}
&\includegraphics[height=1.5cm]{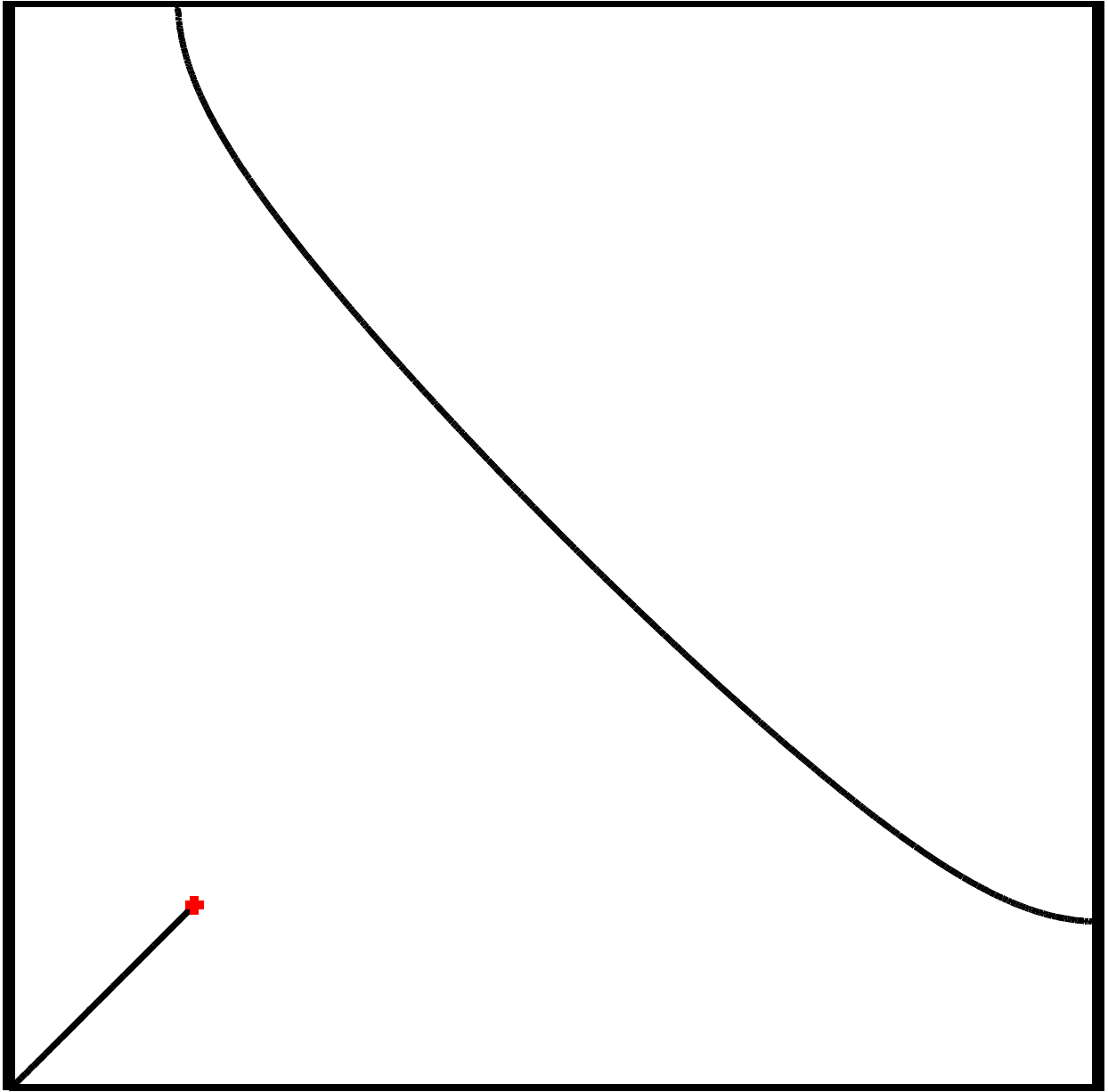}
&\includegraphics[height=1.5cm]{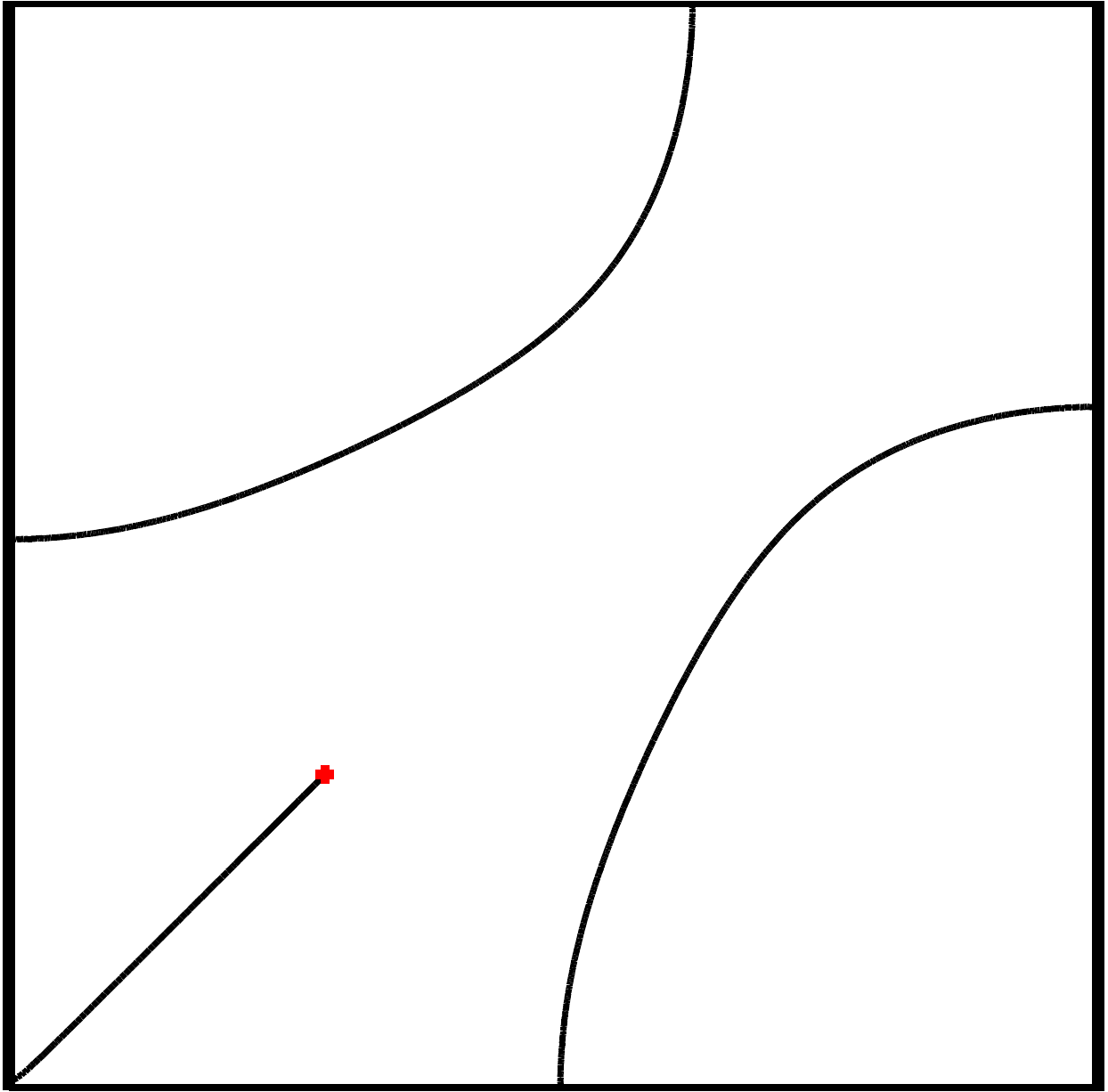}
&\includegraphics[height=1.5cm]{carre_t25.pdf} 
&\includegraphics[height=1.5cm]{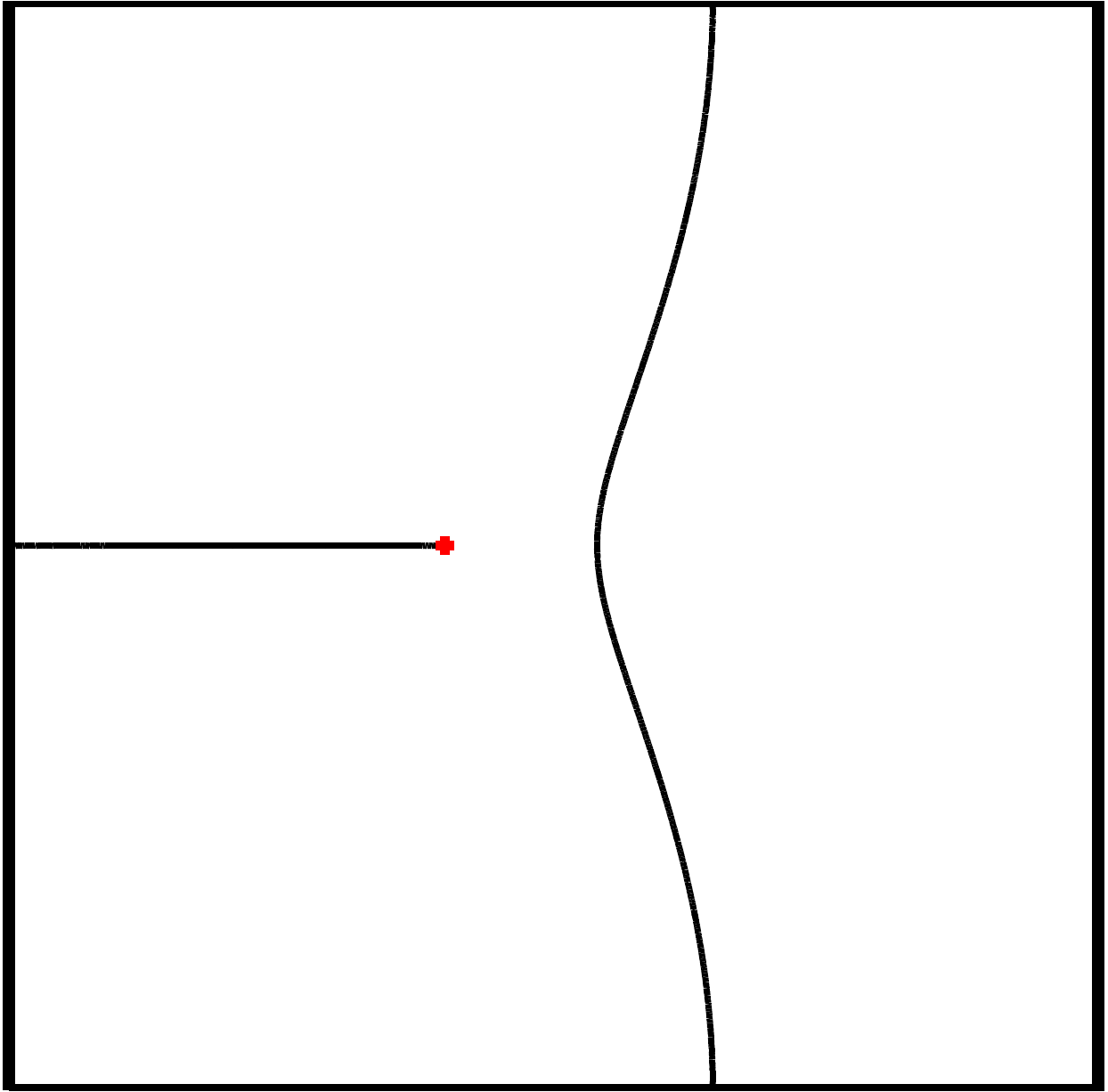}
&\includegraphics[height=1.5cm]{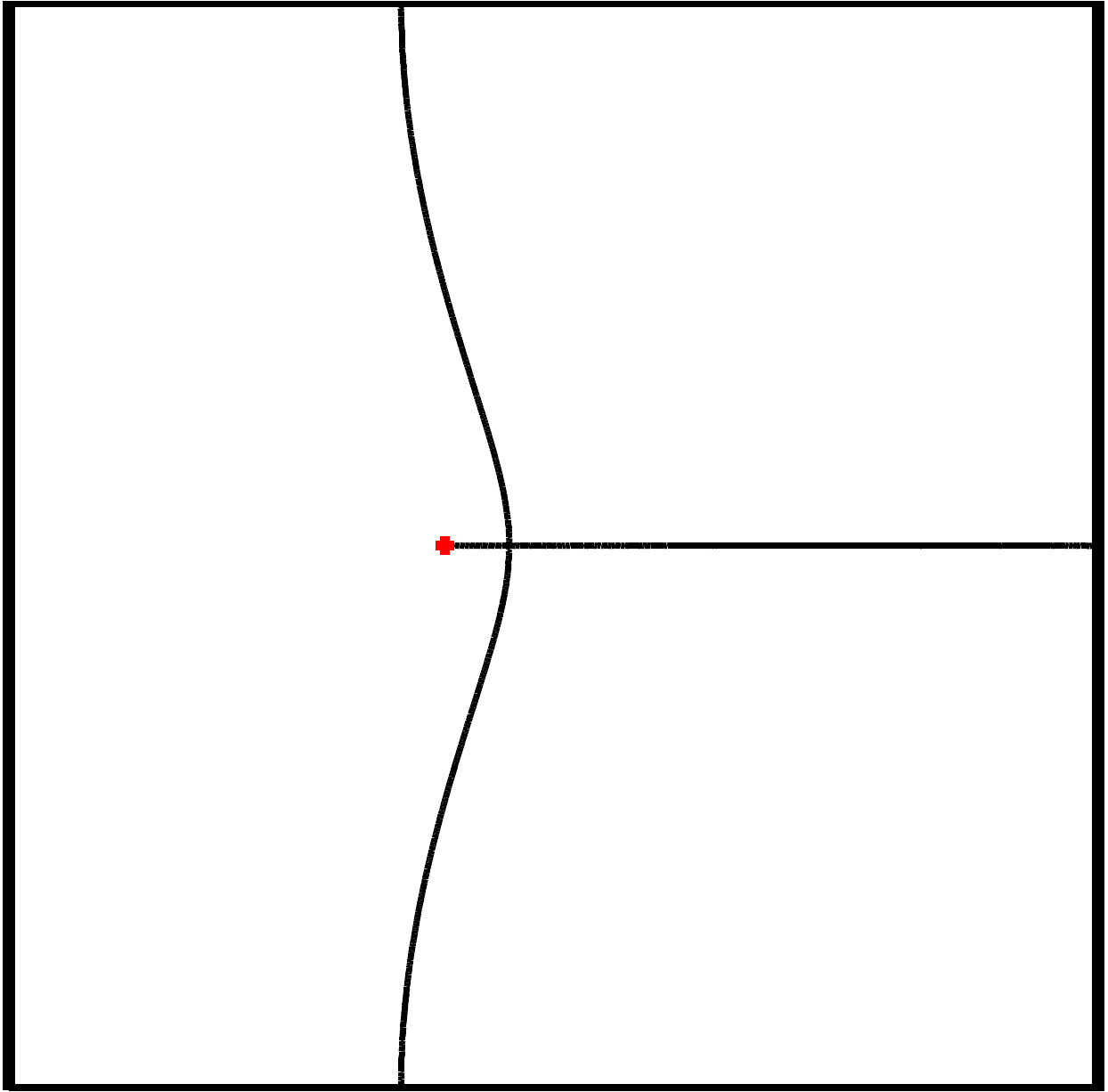}
&\includegraphics[height=1.5cm]{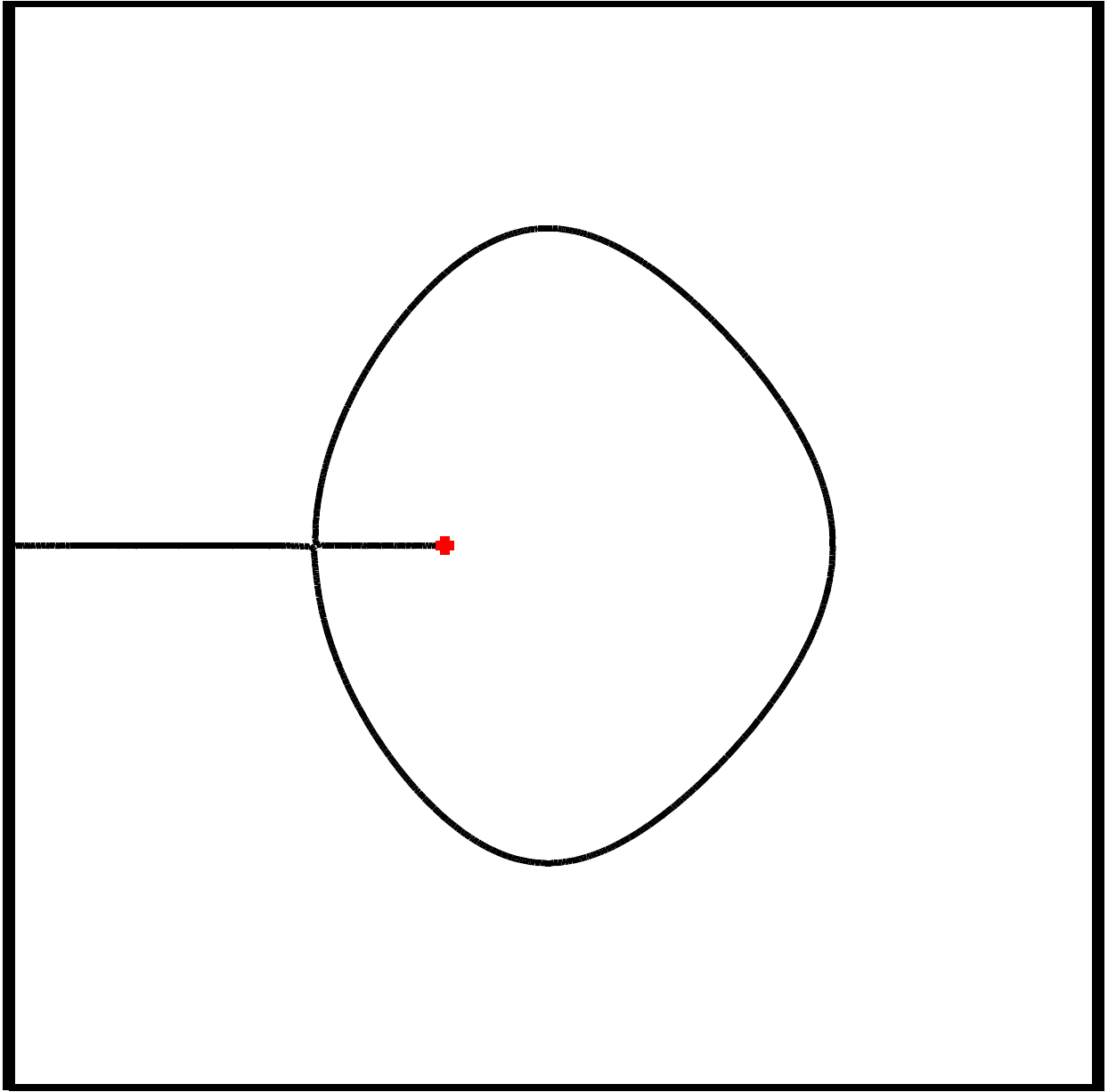}
&\includegraphics[height=1.5cm]{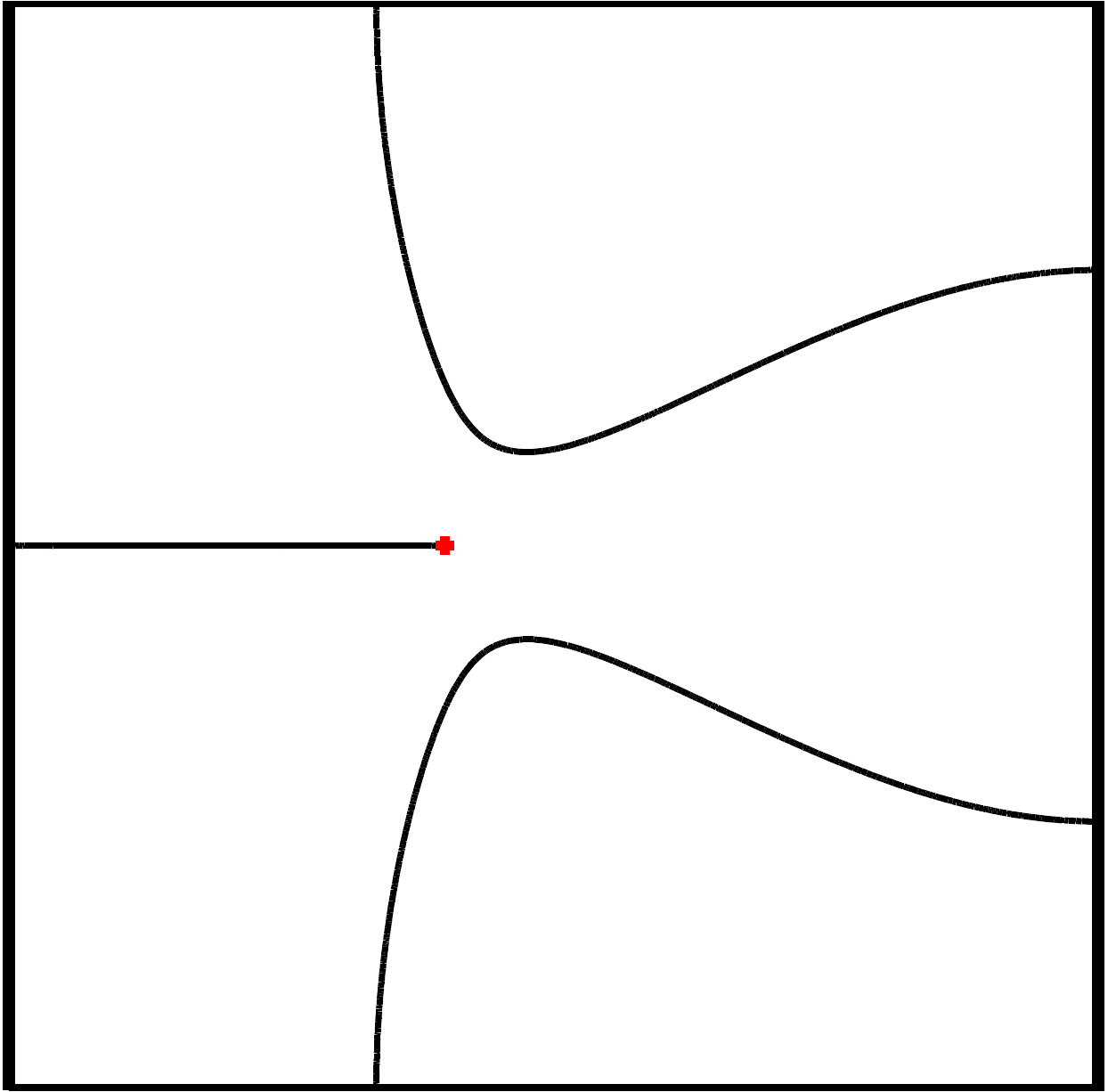}
\end{tabular}
\caption{Nodal lines of some Aharonov-Bohm eigenfunctions on the square.\label{chapBH.fig.exvecpABcarre}}
\end{center}
\end{figure}
\begin{figure}[h!bt]
\begin{center}
\begin{tabular}{ccccc}
\includegraphics[height=1.5cm]{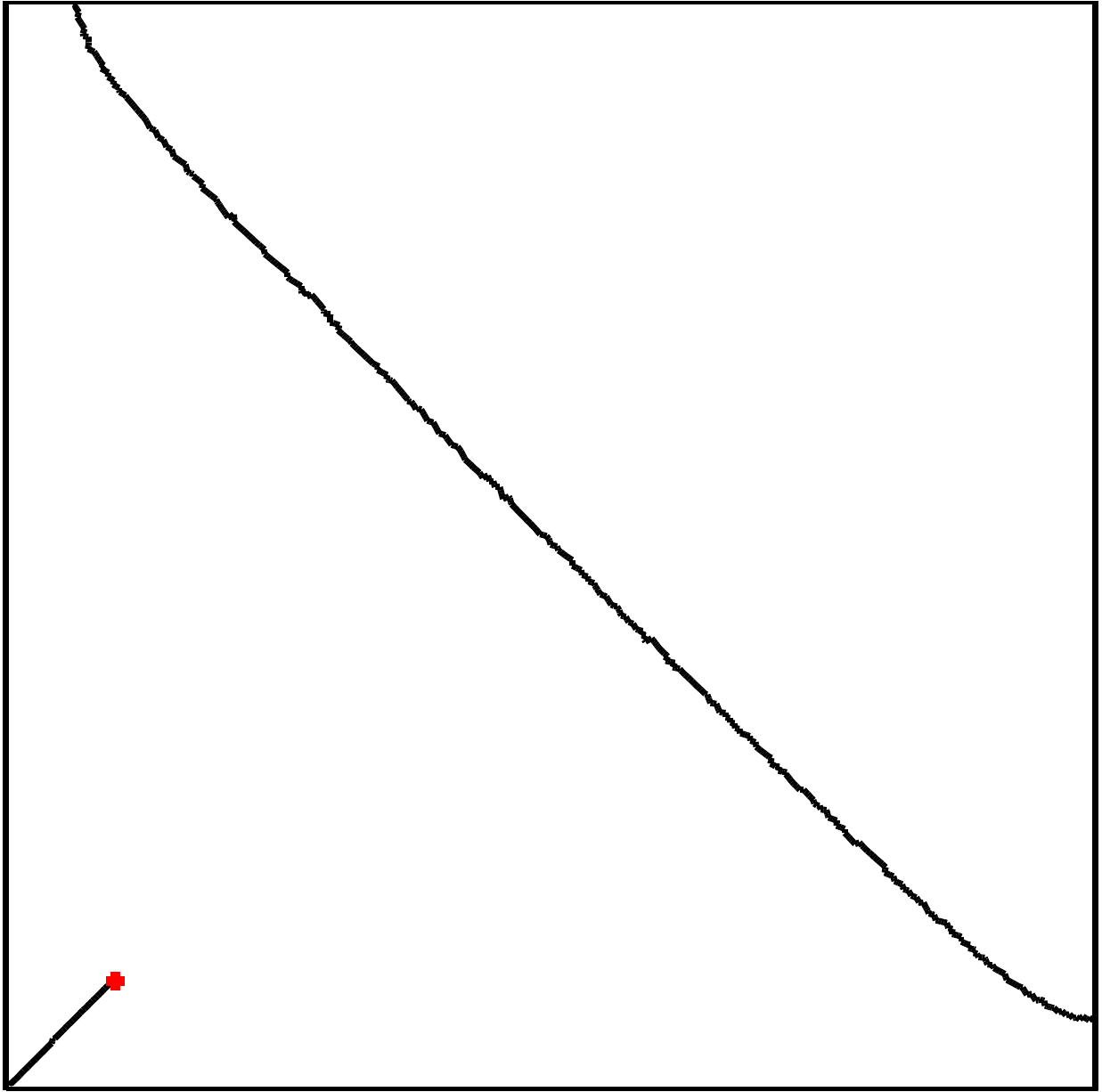}&
\includegraphics[height=1.5cm]{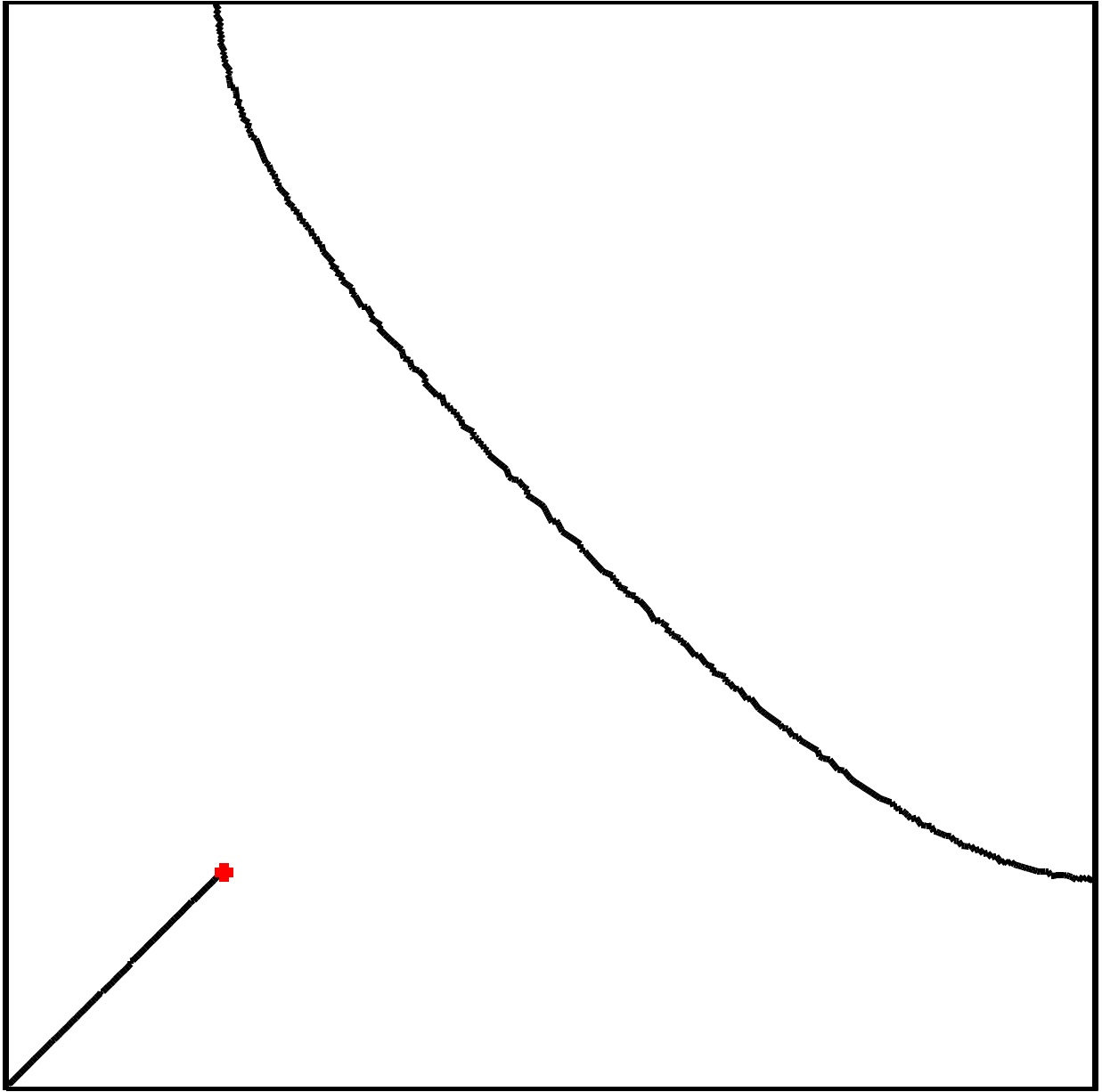}&
\includegraphics[height=1.5cm]{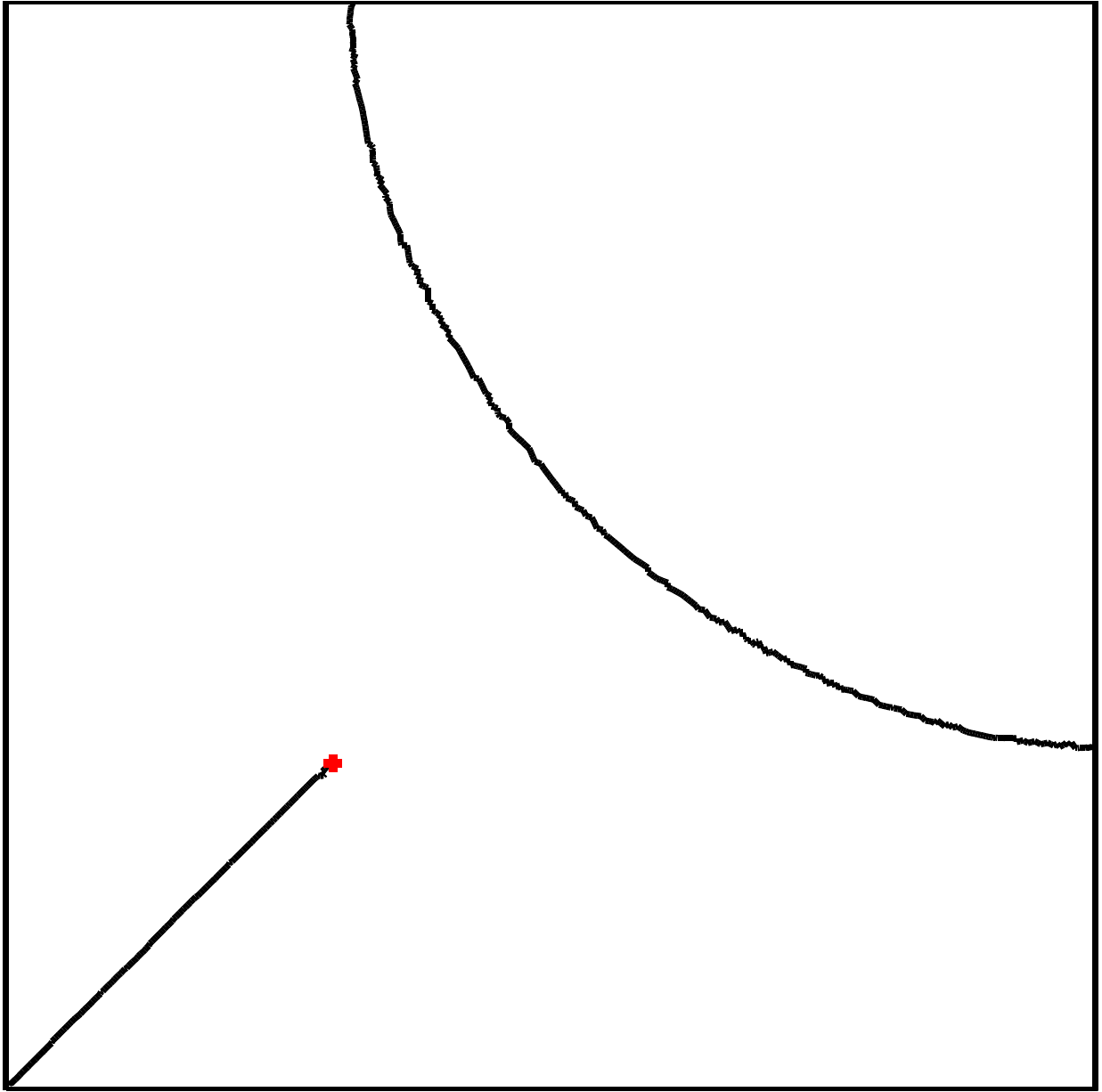}&
\includegraphics[height=1.5cm]{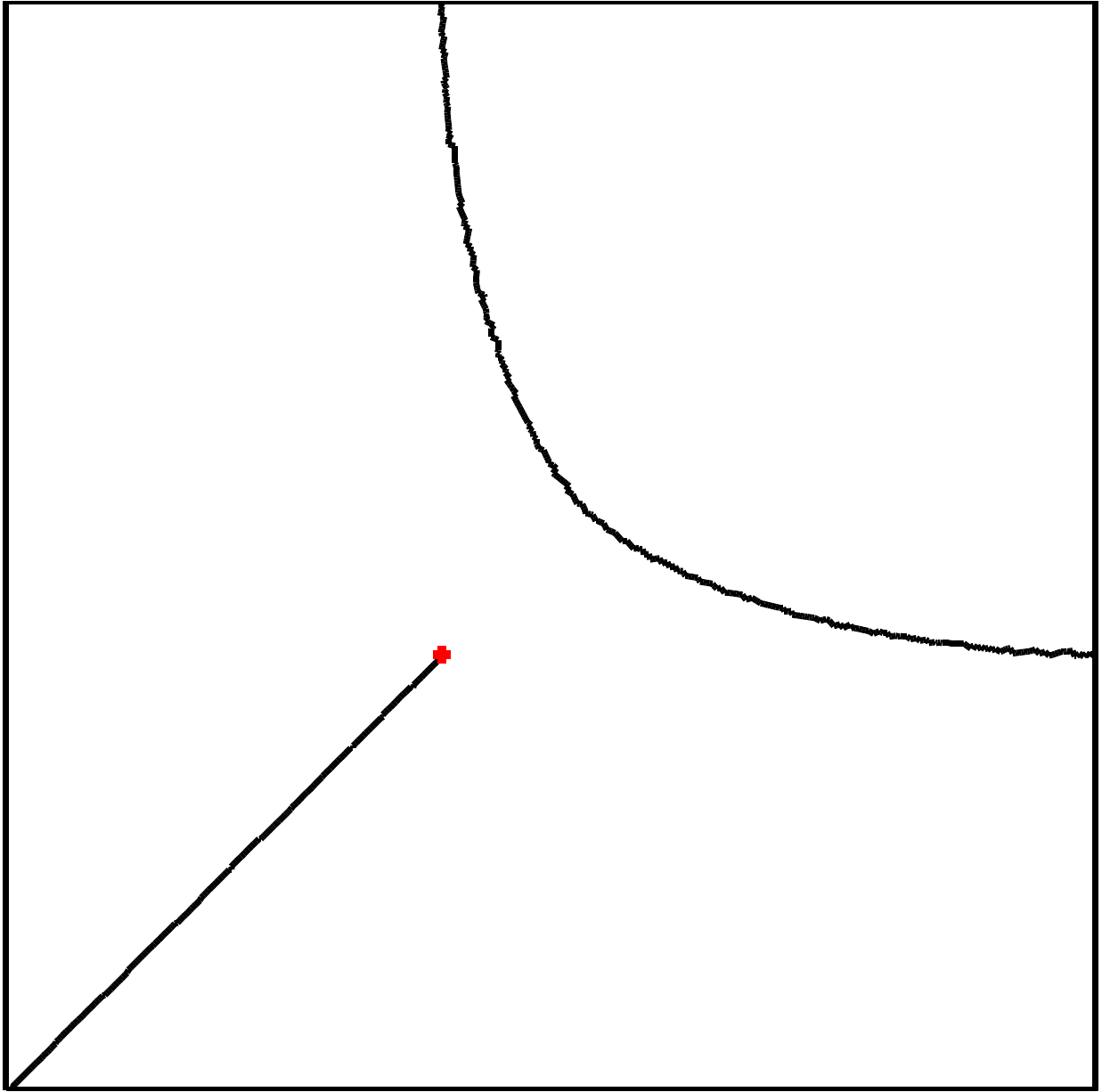}&
\includegraphics[height=1.5cm]{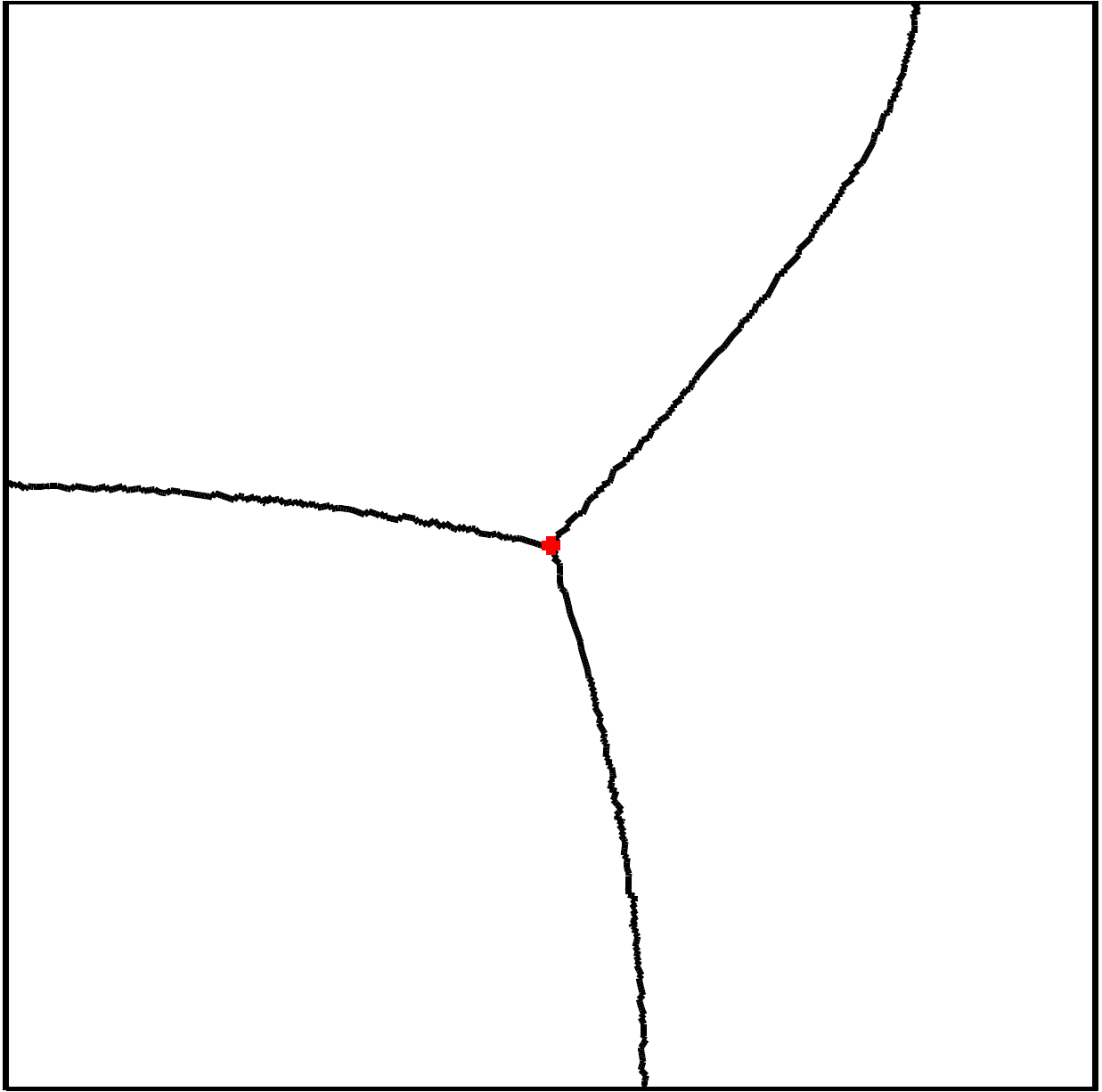}
\end{tabular}
\caption{Nodal lines for the third Aharonov-Bohm eigenfunction in function of ${\bf p}$ on the diagonal.\label{chapBH.fig.VP3AB} }
\end{center}
\end{figure}
Our guess for the punctured square (${\bf p}$ at the center) is that any nodal partition of a third $K_{{\bf p}}$-real eigenfunction gives a minimal $ 3$-partition.
Numerics shows that this is only true if the square is punctured at the center (see Figure~\ref{chapBH.fig.VP3AB} and \cite{MR2836255} for a systematic study). Moreover the third eigenvalue is maximal there and has multiplicity two (see Figure~\ref{chapBH.fig.ABsquare}).
\subsection{Minimal partitions and Aharonov-Bohm operators}
 Helffer--Hoffmann-Ostenhof prove a magnetic characterization of minimal $k$-partitions (see~\cite[Theorem 5.1]{MR3120736}):
\begin{theorem}\label{chapBH.thchar}
Let $\Omega$ be simply connected and $\mathcal D$ be a minimal $k$-partition of $\Omega$. Then $\mathcal D$ is the nodal partition\index{nodal domains} of some $k$-th $ K_{{\bf {\bf P}}}$-real eigenfunction of $H^{AB} ( \dot{\Omega}_{\bf P})$ with $\{{\bf p}_1,\ldots, {\bf p}_\ell\}= X^{\sf odd}(\partial \mathcal D )\,$.
\end{theorem}{
\begin{proof}
We return to the proof that a bipartite minimal partition is nodal for the Laplacian. Using the $\varphi_j$ whose existence was recalled for minimal partitions, we can find a sequence $\varepsilon_j =\pm 1$ such that $\sum_j \varepsilon_j \exp ( \frac i 2 \Theta_{{\bf P}} ({\bf x}) )\, \varphi_j({\bf x})$ is an eigenfunction of $H^{AB} ( \dot{\Omega}_{\bf P})$, where $\Theta_{{\bf P}}$ was defined in \eqref{chapBH.defTheta}.
\end{proof}

\subsection{Continuity with respect to the poles}
In the case of a unique singular point, \cite{MR2815036}, \cite[Theorem 1.1]{MR3270167} establish the continuity with respect to the singular point up to the boundary.
\begin{theorem}\label{chapBH.thm.BNNT}
Let $\Omega$ be simply connected,  $\alpha\in [0,1)$, $\lambda_{k}^{AB}({\bf p},\alpha)$ be the $k$-th eigenvalue of $H^{AB}(\dot\Omega_{{\bf p}},\alpha)$. Then the function ${\bf p}\in\Omega\mapsto \lambda_{k}^{AB}({\bf p},\alpha)$ admits a continuous extension on $\overline\Omega$ and
\begin{equation}\label{chapBH.contpole}
\lim_{{\bf p}\to\partial\Omega}\lambda_{k}^{AB}({\bf p},\alpha) = \lambda_{k}(\Omega)\,,\qquad\forall k\geq1\,.
\end{equation}
\end{theorem}
The theorem implies that the function ${\bf p}\mapsto \lambda_{k}^{AB}({\bf p},\alpha)$ has an extremal point in $\overline \Omega$. Note also that $\lambda_k^{AB}({\bf p},\alpha)$ is well defined for ${\bf p}\not\in \Omega$ and is equal to $\lambda_k(\Omega)$. One can indeed find a solution $\phi$ in $\Omega$ satisfying $d\phi ={\bf A}_{{\bf p}}\,$, and $u \mapsto \exp(i \alpha \phi) \, u$ defines the unitary transform intertwining $H(\Omega)$ and $H^{AB}(\dot \Omega_{{\bf p}}, \alpha)\,$.
\begin{figure}[h!bt]
\begin{center}
\subfloat[${\bf p}\mapsto \lambda^{AB}_{k}({\bf p})$, ${\bf p}\in\Omega$, $1\leq k\leq 5$\,.\label{fig7a}]{\begin{tabular}{cccccc}
$k=1$ & $k=2$ & $k=3$ & $k=4$ & $k=5$\\
\includegraphics[width=2.8cm]{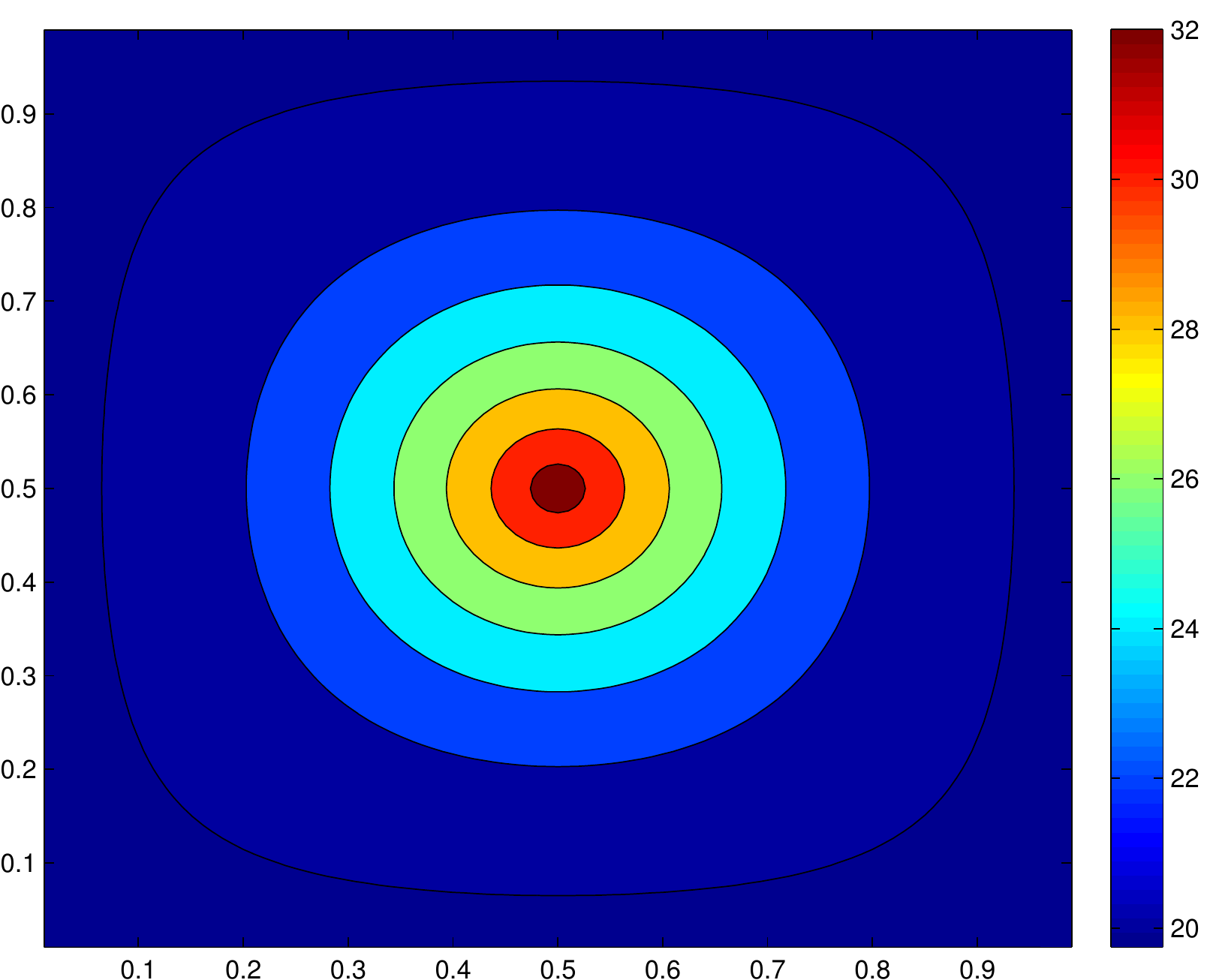}
&\includegraphics[width=2.8cm]{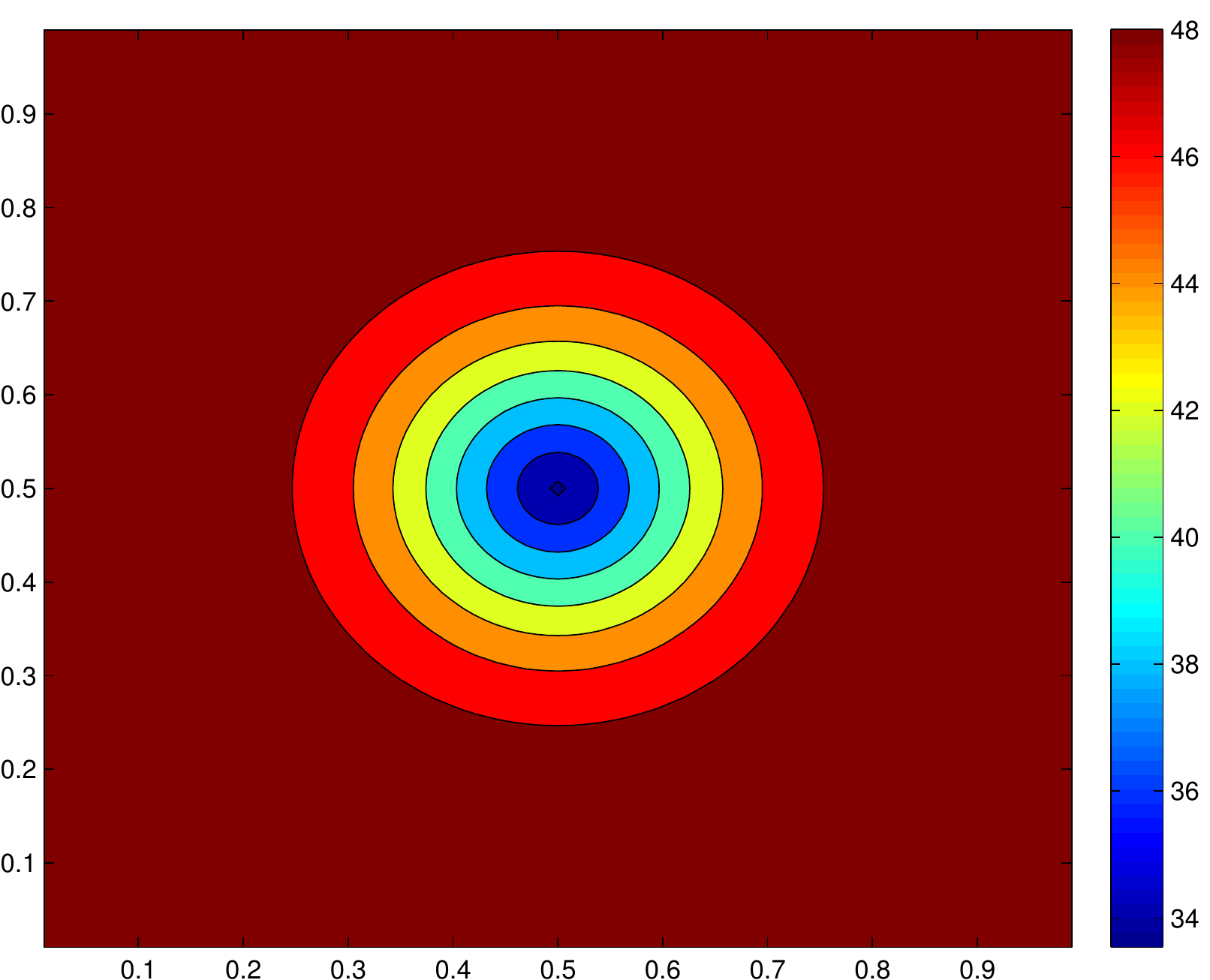}
&\includegraphics[width=2.8cm]{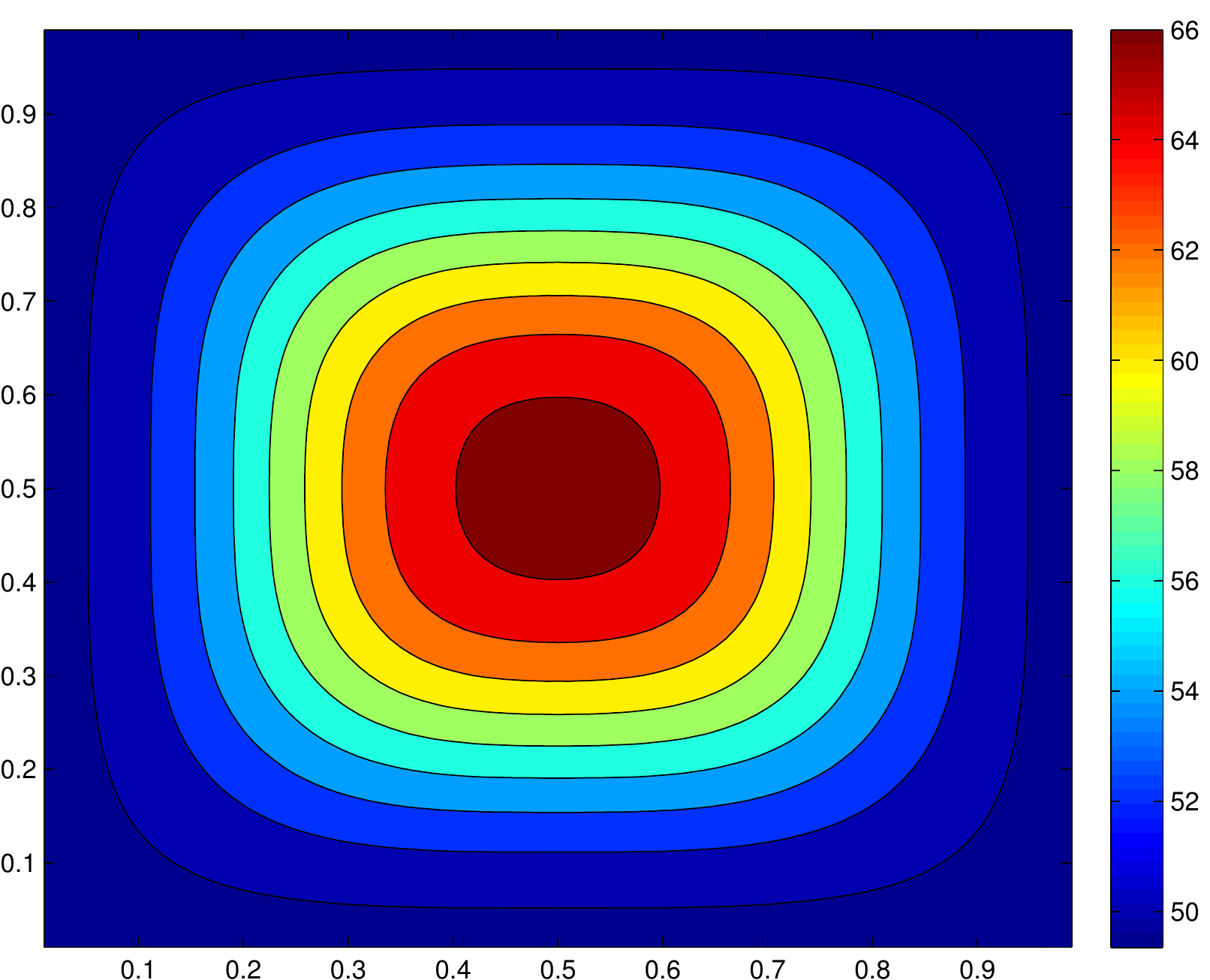}
&\includegraphics[width=2.8cm]{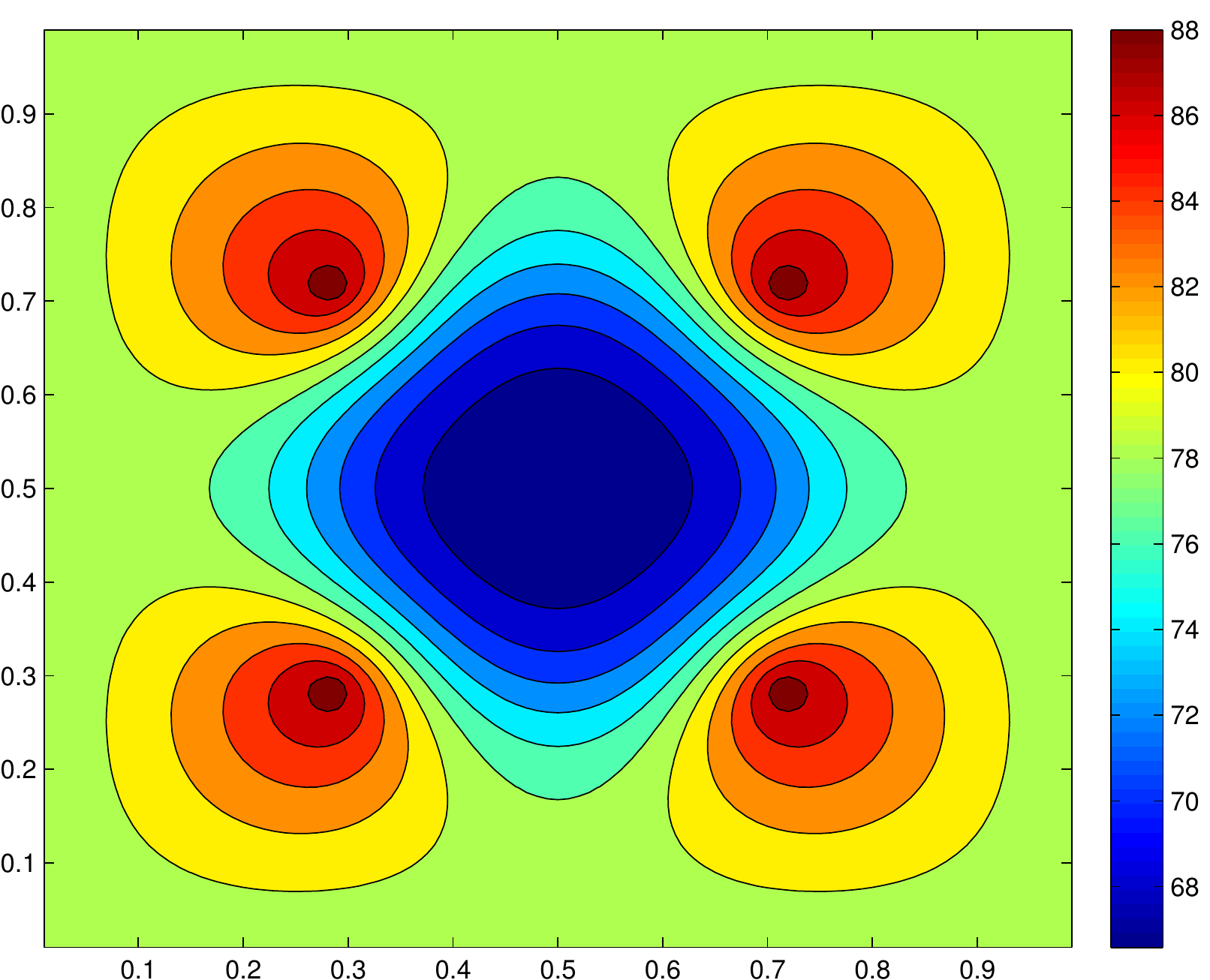}
&\includegraphics[width=2.8cm]{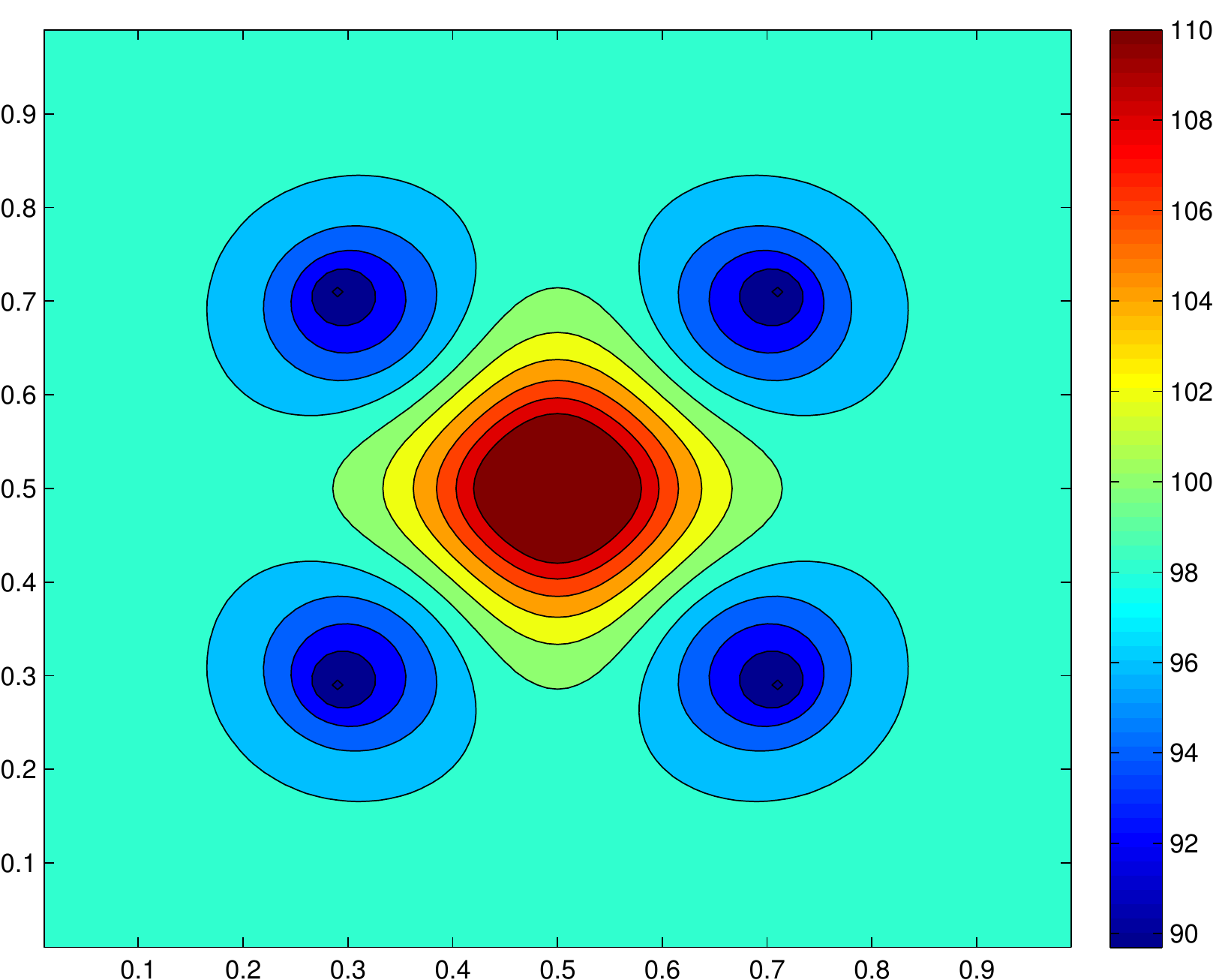}
\end{tabular}}\\
\subfloat[$p\mapsto \lambda^{AB}_{k}({\bf p})$ with ${\bf p}=(p,p)\,$, $1\leq k \leq 9$.\label{chapBH.fig.diagcarre}]{ \quad\quad\includegraphics[height=4cm]{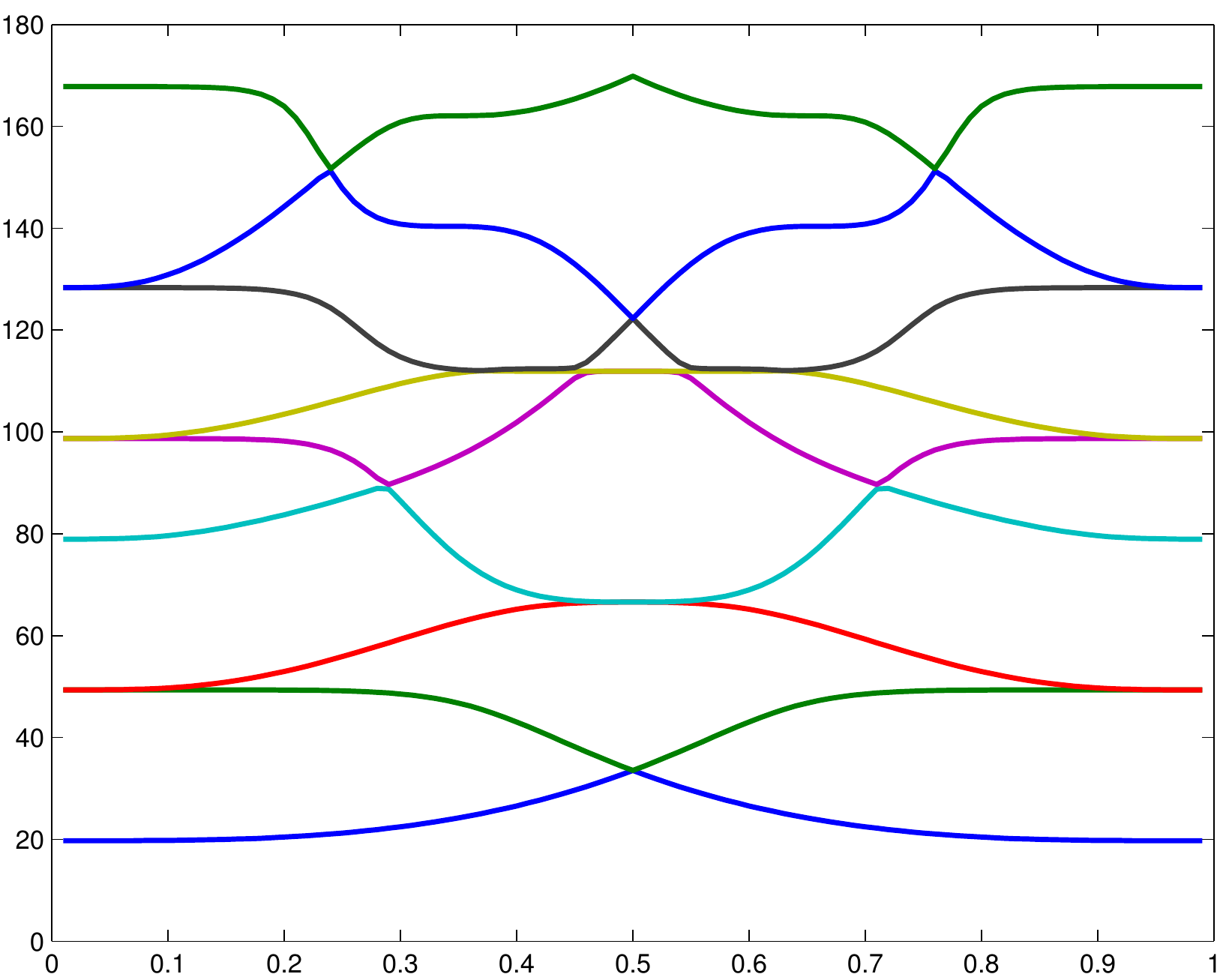}\quad\quad}
\subfloat[$p\mapsto \lambda^{AB}_{k}({\bf p})$ with ${\bf p}=(p,\frac12)\,$, $1\leq k \leq 9$.\label{chapBH.fig.medcarre}]{ \quad\quad\includegraphics[height=4cm]{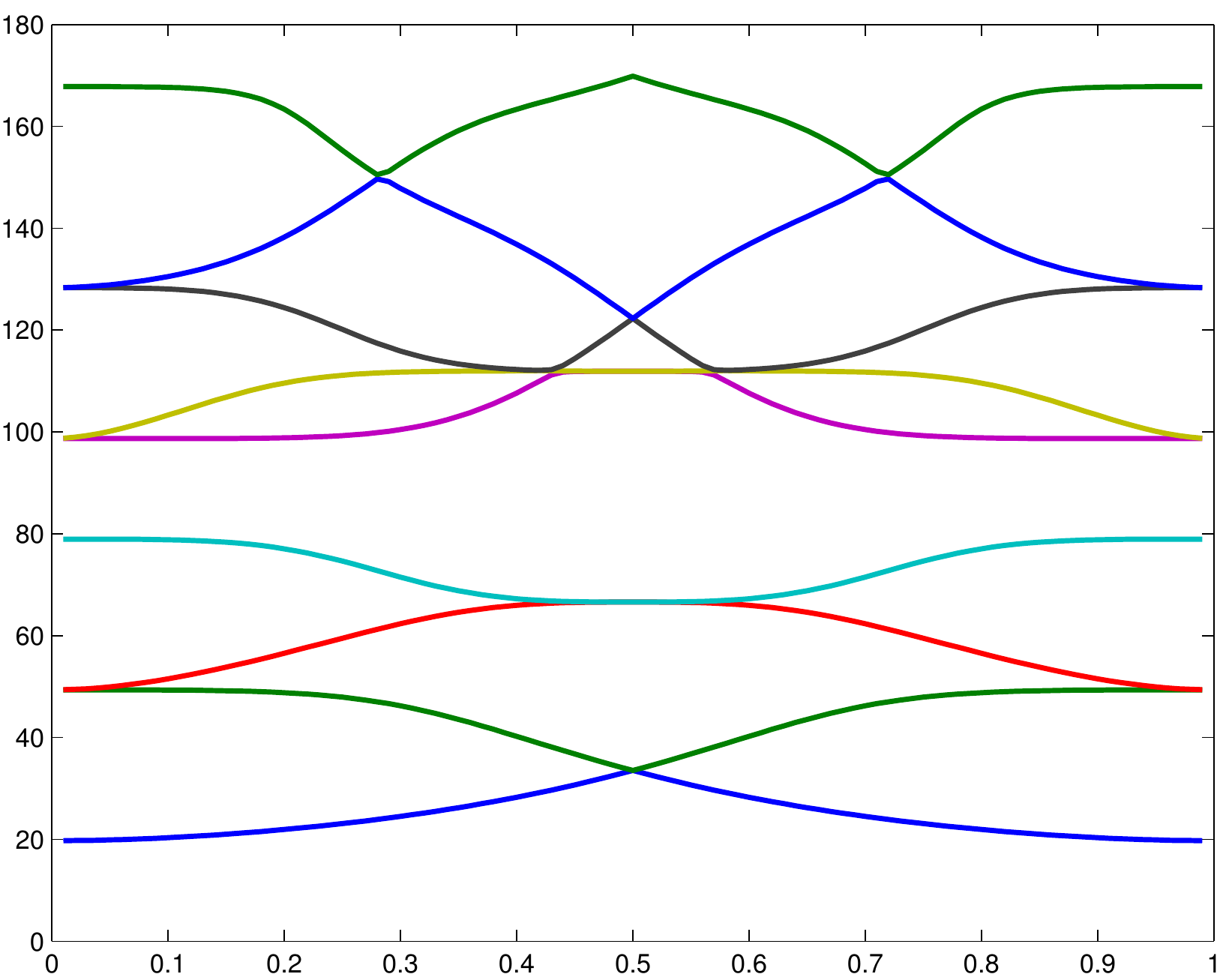}\quad\quad}
\caption{Aharonov-Bohm eigenvalues $\lambda^{AB}_{k}({\bf p})$ on the square as functions of the pole $\bf p$.\label{chapBH.fig.ABsquare}}
\end{center}
\end{figure}
Figure~\ref{fig7a} gives the first eigenvalues of $H^{AB}(\dot\Omega_{{\bf p}})$ in function of ${\bf p}$ in the square $\Omega=[0,1]^2$ and demonstrates \eqref{chapBH.contpole}. When ${\bf p} = (1/2,1/2)$, the eigenvalue is extremal and always double (see in particular Figures~\ref{chapBH.fig.diagcarre} and \ref{chapBH.fig.medcarre} which represent the first eigenvalues when the pole is either on a diagonal line or on a bisector line).

Let us analyze what can happen at an extremal point (see \cite[Theorem 1.1]{MR2815036}, \cite[Theorem 1.5]{MR3270167}).
\begin{theorem}\label{chapBH.thm.BNNT2}
Suppose $\alpha=1/2$. For any $k\geq1$ and ${\bf p}\in\Omega$, we denote by $\varphi_{k}^{AB,{\bf p}}$ an eigenfunction associated with $\lambda_{k}^{AB}({\bf p})\,$.
\begin{itemize}[label=--,itemsep=-2pt]
\item If $\varphi_{k}^{AB,{\bf p}}$ has a zero of order $1/2$ at ${\bf p}\in\Omega$, then either $\lambda_{k}^{AB}({\bf p})$ has multiplicity at least $2$, or ${\bf p}$ is not an extremal point of the map ${\bf x}\mapsto \lambda_{k}^{AB}({\bf x})$.
\item If ${\bf p}\in\Omega$ is an extremal point of ${\bf x}\mapsto \lambda_{k}^{AB}({\bf x})\,$, then either $\lambda_{k}^{AB}({\bf p})$ has multiplicity at least $2$, or $\varphi_{k}^{AB,{\bf p}}$ has a zero of order $m/2$ at ${\bf p}$, $m\geq3$ odd.
\end{itemize}
\end{theorem}
This theorem gives an interesting necessary condition for candidates to be minimal partitions. Indeed, knowing the behavior of the eigenvalues of Aharonov-Bohm operator, we can localize the position of the critical point for which the associated eigenfunction can produce a nice partition (with singular point where an odd number of lines end). We observe in Figure~\ref{chapBH.fig.ABsquare} that the eigenvalue is never simple at an extremal point.

When there are several poles, the continuity result of Theorem~\ref{chapBH.thm.BNNT} still holds. We will briefly address this result (see \cite{MR3390812} for the proof and more details). This is rather clear in $\Omega^\ell \setminus \mathcal C$, where $\mathcal C$ denotes the ${{\bf P}}$'s such that ${\bf p}_i \neq {\bf p}_j$ when $i\neq j$. It is then convenient to extend the function ${\bf P} \mapsto \lambda^{AB}_k({\bf P}, {\boldsymbol{\alpha}} )$ to $(\mathbb R^2)^\ell$. We define $\lambda^{AB}_k({\bf P}, {\boldsymbol{\alpha}})$ as the $k$-th eigenvalue of $H^{AB}(\dot\Omega_{\tilde{\bf P}}, \tilde{\boldsymbol{\alpha}})$, where the $m$-tuple $\tilde{\bf P}=(\tilde{\bf p}_1,\dots,\tilde{\bf p}_{m})$ contains once, and only once, each point appearing in ${\bf P}=({\bf p}_1,\dots,{\bf p}_{\ell})$ and where $\tilde{\boldsymbol{\alpha}}=(\tilde\alpha_1,\dots,\tilde\alpha_M)$ with
$\tilde\alpha_k=\sum_{j,\,{\bf p}_j=\tilde{\bf p}_k}\alpha_j,$ for $1\leq k\leq m\,.$
\begin{theorem} \label{chapBH.thm.EigCont}
If $k\ge 1$ and $\boldsymbol{\alpha} \in \mathbb R^\ell$, then the function ${\bf P}\mapsto \lambda^{AB}_k ({\bf P},{\boldsymbol{\alpha}} )$ is continuous in $\mathbb R^{2\ell}$.
\end{theorem}
This result generalizes Theorems~\ref{chapBH.thm.BNNT} and \ref{chapBH.thm.BNNT2}. It implies in particular continuity of the eigenvalues when one point tends to $\partial\Omega$, or in the case of coalescing points. For example, take $\ell=2$, $\alpha_1=\alpha_2 = 1/2$, ${\bf P}=({\bf p}_1,{\bf p}_2) $ and suppose that ${\bf p}_1$ and ${\bf p}_2$ tend to some ${\bf p}$ in $\Omega$. One obtains in this case that $\lambda^{AB}_k({\bf P}, {\boldsymbol{\alpha}} )$ tends to $\lambda_k(\Omega)$.

\begin{remark}
More results on the Aharonov-Bohm eigenvalues as function of the poles can be found in \cite{MR2836255,MR2815036,MR3270167, MR3426097, MR3390812,MR3641642,1612.01330,1706.05247,1611.06750,1707.04416}. We have only emphasized in this section on the results which have direct applications to the research of candidates for minimal partitions.\\
In many of the papers analyzing minimal partitions, the authors refer to a double covering argument. Although this point of view (which appears first in \cite{MR1690957} in the case of domains with holes) is essentially equivalent to the Aharonov approach, it has a more geometrical flavor. \end{remark}

\section{On the asymptotic\index{asymptotic} behavior of minimal $k$-partitions}\label{chapBH.sec.klarge}
The hexagon has fascinating properties and appears naturally in many contexts (for example the honeycomb).  Let $\hexagon$ be a regular hexagon of unit area. If we consider polygons generating a tiling, the ground state energy $\lambda(\hexagon)$ gives the smallest value (at least in comparison with the square, the rectangle and the equilateral triangle). In this section we analyze the asymptotic behavior of minimal $k$-partitions as $k \to +\infty\,$.

\subsection{The hexagonal conjecture}\label{chapBH.ss9.1}
\begin{conjecture}\label{chapBH.ConjAs1}
The limit of ${ {\mathfrak L_k(\Omega)}/{k}}$ as \index{hexagonal conjecture}
${ k\to +\infty}$ exists and $$
 |\Omega|\lim_{k\to +\infty} \frac{\mathfrak L_k(\Omega)}{k}=\lambda(\hexagon)\;.
$$
\end{conjecture}
Similarly, one has
\begin{conjecture}\label{chapBH.ConjAs2}
The limit of ${ {\mathfrak L_{k,1}(\Omega)}/{k}}$ as
${ k\to +\infty}$ exists and
\begin{equation}
 |\Omega|\lim_{k\to +\infty} \frac{\mathfrak L_{k,1}(\Omega)}{k}=\lambda(\hexagon)\;.
\end{equation}
\end{conjecture}
These conjectures, that we learned from M. Van den Berg in 2006, are also mentioned in Caffarelli-Lin \cite{MR2304268} for $\mathfrak L_{k,1}$ and  imply  that the limit is independent of $\Omega\,$.
Of course the optimality of the regular hexagonal tiling appears in various contexts in physics.  By keeping the hexagons belonging to the intersection of $\Omega$ with the hexagonal tiling and using the monotonicity of $\mathfrak L_k$ for the inclusion, it is easy to show the upper bound in Conjecture \ref{chapBH.ConjAs1},
\begin{equation}\label{chapBH.upperboundhexa}
 |\Omega|\lim\sup_{k\to +\infty} \frac{\mathfrak L_k(\Omega)}{k} \leq \lambda(\hexagon)\;.
\end{equation}
We recall that the Faber-Krahn inequality\index{Faber-Krahn (inequality)} \eqref{chapBH.eq.FK} gives a weaker lower bound
\begin{equation}\label{chapBH.lowerboundhexa}
|\Omega| \frac{\mathfrak L_k(\Omega)}{k} \geq |\Omega| \frac{\mathfrak L_{k,1}(\Omega)}{k} \geq \lambda(\Circle) \,.
\end{equation}
Note that Bourgain \cite{MR3340367} and Steinerberger \cite{MR3272823} have recently improved the lower bound by using an improved Faber-Krahn inequality together with considerations on packing property by disks.\\
The inequality $\mathfrak L_{k,1}(\Omega) \leq \mathfrak L_k(\Omega)$ together with the upper bound \eqref{chapBH.upperboundhexa} shows that the second conjecture implies the first one. 
Conjecture \ref{chapBH.ConjAs1} has been explored in \cite{MR2598097} by checking numerically non trivial consequences of this conjecture. Other recent numerical computations devoted to $\lim_{k\to +\infty} \frac{1}{k} \mathfrak L_{k,1}(\Omega)$ and to the asymptotic structure of the minimal partitions are given by Bourdin-Bucur-Oudet \cite{MR2566585}.

The hexagonal conjecture leads to a natural corresponding hexagonal conjecture for the length of the boundary set, namely
\begin{conjecture}\label{chapBH.ConjAs3}
Let $\ell (\hexagon)= 2 \sqrt{2\sqrt{3}} $ be the length of the boundary of \hexagon\,. Then
\begin{equation}\label{chapBH.hcl}
\lim_{k\to +\infty} \frac{ |\partial\mathcal{D}_k|}{\sqrt{k}} = \frac{1}{2} \ell (\hexagon) \sqrt{ |\Omega|}\,.
\end{equation}
\end{conjecture}
This point is discussed in \cite{MR3151084} in connection with the celebrated theorem of Hales \cite{MR1797293} proving the honeycomb conjecture.
\subsection{Lower bound for the number of singular points}
It has been established since 1925 (A. Stern, H. Lewy, J. Leydold, B\'erard-Helffer, see \cite{BH3} and references therein), that there are domains for which there exist an  infinite sequence of eigenvalues of the Laplacian  for which the corresponding eigenvalues have a fixed number of nodal domains and critical points in the zeroset. The next result (see \cite{MR3381015},  \cite{MR3639652}) shows that the situation is quite different for minimal partitions.
 \begin{theorem}
For any sequence  $(\mathcal D_k)_{k\in \mathbb N} $ of regular minimal $k$-partitions, we have
\begin{equation}\label{chapBH.eq.liminfk}
\liminf_{k\to \infty}\frac{\sharp X^{\sf odd}(\partial \mathcal D_k )}k >0\,.
\end{equation}
 \end{theorem}
Although inspired by the proof of Pleijel's theorem, this proof
includes (for any $k$) a lower bound in the Weyl's formula for the eigenvalue $\mathfrak L_k$ of the Aharonov-Bohm operator $ H^{AB}(\dot{\Omega}_{\bf P}) $ associated with the odd singular points of $\mathcal D_k$. The proof gives an explicit but very small lower bound in \eqref{chapBH.eq.liminfk} which is independent of the sequence.
This is to compare with the upper bound \eqref{chapBH.eulergrandk} which gives
$$\limsup_{k\to\infty}\frac{ \sharp X^{\sf odd}(\partial \mathcal D_k )}k \leq 2\,.$$
\begin{remark}
The hexagonal conjecture in the case of a compact Riemannian manifold is the same. We refer to \cite{MR3151084} for the details, the idea being that, for $k$ large,  the local structure of the manifold plays the main role, like for Pleijel's formula (see \cite{MR690651}).
In \cite{MR3421911} the authors analyze numerically the validity of the hexagonal conjecture in the case of the sphere for $\mathfrak L_{k,1}$. Using Euler's formula, one can conjecture that there are $(k-12)$ hexagons and $12$ pentagons for $k$ large enough. In the case of a planar domain one expects curvilinear hexagons inside $\Omega$, around $\sqrt{k}$ pentagons close to the boundary (see \cite{MR2566585}) and a few number of other polygons. Let us mention also the recent related results of D.~Bucur and collaborators \cite{1707.00605,1703.05383}.
\end{remark}
\section{Open problems}
We finish this survey by recalling other  open problems.

\subsection{Problems related to Pleijel's theorem}
Let us come back to the strong version of Pleijel's theorem (see Subsection \ref{chapBH.ss2.5}). 
The relation of the number of nodal domains and the properties of the Pleijel constant 
$$Pl(\Omega):=\lim\sup_{n\to \infty}\frac{\mu(\varphi_n)}{n},$$ 
appearing in \eqref{chapBH.bourgain3} is still very mysterious. Here $\Omega\subset {\mathbb R}^d$ is a bounded domain and we consider a Dirichlet Laplacian. But, as observed in Remark \ref{RemPl}  one could  also consider the analog problems for Schr\"odinger operators. We mention some natural questions, but some of them might seem rather ridiculous, we just  want to demonstrate how little is understood. 
\begin{enumerate}
\item Find some bounded domains $\Omega\subset {\mathbb R}^d$ so that one can work out the Pleijel constant $Pl(\Omega)$  explicitly. For rectangles $\mathcal R(a,b)$ with sidelengths $a, b$ with $\frac{a^2}{b^2}$ irrational it is known that \break  $Pl(\Omega)=\frac{2}{\pi}$, see {\it e.g.} \cite{MR3381015}. Based on numerical work \cite{PhysRevLett.88.114101}, Polterovich \cite{MR2457442} moreover conjectured that $Pl(\Omega)\le \frac{2}{\pi}$. For the harmonic oscillator $ H^{\sf osc}=-\Delta+\sum_{j=1}^d a_j^2x_j^2$ on ${\mathbb R}^d$ with $a_1,\ldots,a_{d}$ rationally independent, Charron \cite{1512.07880} showed that $\lim\sup_{n\to \infty}\frac{\mu(\varphi_n)}{n}=\frac{d!}{d^d}$.

\item It is not known whether $Pl(\Omega) >0$ always holds and even whether $\lim \sup_{n\to \infty} \mu(\varphi_n)=+\infty $ in general. Take a bounded domain $\Omega\subset {\mathbb R}^2$. If we consider an example where we can work out the eigenvalues $\lambda_n$ and the associated eigenfunctions $\varphi_n$ explicitely then we have always that $Pl(\Omega)>0$. In general almost nothing is known. There are some very subtle families of manifolds for which it has been shown that $\lim\sup_{n\to \infty}\mu(\varphi_n)=+\infty\,$, see \cite{MR3102912,1510.02963,MR3466853,MR3584194}. But for membranes 
this question is wide open. \\
Here is a simple problem: \\
{\it Prove or disprove that there exists an integer $K$ such that,  for all bounded (perhaps simply connected) domains,  there exists an eigenfunction $\varphi_{k}$ associated with $\lambda_{k}$ with $k\leq K$ and $\mu(\varphi_k)\ge 3$. }\\
A related problem is the following:\\
\centerline{\it  Find $\Omega\subset {\mathbb R}^2$ so that $\mu(\varphi_k)=2$ for $1<k \le 5\,$.}\\
 It is not at all clear that such a membrane exists. There are examples where $\lambda_2$ has multiplicity 3, in the case of the sphere $\mathbb S^2$ for instance, but also  for not simply connected domains  in $\mathbb R^2$ (see the paper by M. Hoffmann-Ostenhof, T. Hoffmann-Ostenhof and N. Nadirashvili \cite{MR1605269} on the nodal line conjecture).
 
For the higher dimensional case, Colin de Verdi\`ere \cite{MR932800} has  shown that one can have arbitrarily high multiplicity of the second eigenvalue for certain Riemannian eigenvalue problems. 
\item Take any $\Omega$ (with Dirichlet or Neumann boundary condition) and consider 
\begin{equation}
\mathfrak \mathcal \mathfrak N(\Omega)=\{k\in {\mathbb N} \;:\; \exists \text{ an eigenfunction } \varphi \text{ with }\mu(\varphi)=k\}.
\end{equation} 
For problems where one can work out the eigenvalues and the corresponding eigenfunctions explicitly,  we have usually $\mathfrak N(\Omega)={\mathbb N}$. For the circle $\mathbb S^1$ we have $\mathfrak N=\{1\}\cup 2{\mathbb N}$. This is also the case for the torus $T(a,b)$ where $T(a,b)$ is the rectangle $\mathcal R(a,b)$ with periodic boundary condition and with $(a/b)^2$ irrational. If this assumption $a,b$ does not hold then  there are examples of $\mu(\varphi)=3$, see \cite{1504.03944}.
Here comes a problem:\\
\centerline{\it  Prove or disprove that for bounded   $\Omega\subset {\mathbb R}^2$  either $\mathfrak N ={\mathbb N}$ or $\mathfrak N =\{1\}\cup 2 {\mathbb N}\,$. }\\
For instance can it happen that there is a domain so that no eigenfunction has four nodal domains? Analog questions can be also asked for the $d$-dimensional case, Schr\"odinger operators and for Laplace Beltrami operators on bounded manifolds.
\end{enumerate}

\subsection{Open problems on minimal  partitions} 
We do not come back to the Mercedes star conjecture for the disk, which was discussed in Subsection~\ref{chapBH.sexamples} but there are related problems for which we have natural guesses for minimal $3$-partitions and, more generally minimal $k$-partitions. Consider the equilateral triangle and the regular hexagon. In both cases the boundary set of the $3$-partition should consist of three straight segments which start in the middle of the side and meet at the center with the angle $2\pi/3$. For the regular hexagon it is similar, the three straight segments start from the middle of three sides which do not neighbor each other and meet in the center of the hexagon.  
A related natural guess is available for the non-nodal minimal  spectral $5$-partition for the disk. There one expects that the minimal partition created by five segments which start form the origin and meet there  with angle $2\pi/5$ (see Figure~\ref{chapBH.fig.MS}). For all those examples there is strong numerical evidence, see \cite{1612.07296,1702.01567}. 
We now mention most of  the cases of non-nodal minimal spectral partitions for which the topology of minimal  $k$-partition  
is known. The simplest case is $\mathbb S^1$. Here we know everything: all minimal $2k$-partitions are nodal and the
minimal $(2k+1)$-partitions just are given by $(0,\frac{2\pi}{2k+1})$ and its rotations by multiples of $\frac{2\pi}{2k+1}$. 
To see this just go to the double covering and look at $(0,4\pi)$ with periodic boundary conditions, or consider  $(D_{x} - \frac 12)^2$ on $(0,2\pi)$ with periodic condition.\\
We consider the Laplacian operator on the rectangle 
$\mathcal R(1,b)=(0,1)\times (0,b)$. 
Two cases will be mentioned.\\ 
First we take periodic boundary conditions  on the interval (0,1) by identifying $x_{1}=0$ with $x_{1}=1$. For $x_{2}=0$ and $x_{2}=b$ we
take Neumann boundary conditions. The spectrum of this operator $H^{{\sf per},N}$ is 
\begin{equation*} 
\sigma(H^{{\sf per},N}) =\Big\{\pi^2(4m^2+\frac{n^2}{b^2}), {(m,n)\in \mathbb Z \times \mathbb N }\Big\}. 
\end{equation*}
Since $\lambda_1=0$ and $\lambda_2=\lambda_3$ for $b<1/2$, $\mathfrak L_3$ is associated to a non-nodal partition. If we let $b>0$ sufficiently small we are "near" the case of $\mathbb S^1$. One then can go to the double covering as for the circle and consider the eigenvalues.  For $b\le (2\sqrt 5)^{-1}$, the minimal 3-partition is then given by $D_1=(0,1/3)\times (0,b),\; D_2=(1/3,2/3)\times (0,b), D_3=(2/3,1)\times (0,b)$, see \cite{MR3050161}. It was only possible with Neumann boundary conditions. For Dirichlet one would expect also such a minimal partition. For larger odd $k$ one gets similar results (requiring smaller $b$). Hence the problem is: \\
\centerline{\it Prove such a result with Dirichlet boundary conditions and improve the conditions on the $b's$.}\\

Next we consider the torus \cite{MR3232810, MR3684573}. Equivalently we consider periodic boundary conditions on the intervals $(0,1)$ and $(0,b)$. Again one can expect, for small $b$, a situation "near" to the circle. Things are more involved but again one can reduce the problem to spectral problems on suitable coverings. There are many nice problems coming up here  which are discussed in  \cite{MR3684573}.
\\

Finally we come back to $\mathbb S^2$ (see Section \ref{chapBH.ss6.3}). Since only $\lambda_1(\mathbb S^2)$ and $\lambda_2(\mathbb S^2)$ are Courant sharp, any $\mathfrak L_k(\mathbb S^2)$ corresponds to a non-nodal minimal partition for $k>2$. In \cite{MR2664708} it was shown (see Figure~\ref{chapBH.fig.3partsphere}) that, up to rotation, the minimal $3$-partition is unique and  that  its  boundary  set consists of three half great circles which connect the north-pole and the 
south-pole and meet each other with angles $2\pi/3$. 
\begin{figure}[h!t]
\begin{center}
\includegraphics[width=2.5cm]{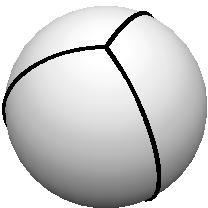}
\caption{Minimal $3$-partition of the sphere.\label{chapBH.fig.3partsphere}}
\end{center}
\end{figure}
There is an open question in harmonic analysis known as the Bishop conjecture (solved for $1$-minimal $2$-partitions) \cite{MR1155854}:\\
\centerline{\it Show that the minimal  $3$-partition  of $\mathbb S^2$ is a $1$-minimal $3$-partition.}\\

In \cite{MR2664708} possible candidates for minimal $4$-partitions were discussed and one natural guess is the regular tetrahedron. More generally, one might ask  the question whether platonic solids  (corners  on the sphere connected by segments of great circles) are possible candidates. First we note that the octahedron would lead to a bipartite partition hence the minimal $8$-partition must be something else. But the tetrahedron, the cube, the dodecahedron and the icosahedron would be possible candidates. There is a related isoperimetric problem whose relation to spectral minimal partitions might be interesting. The problem  is to find the least perimeter partition of $\mathbb S^2$ into $k$ regions of equal area. Up to now the cases $k=3,4,12$, \cite{MR1460978, MR2679060, Hales}, have been solved. For $k=3$ the perimeter partition is also the spectral minimal 3-partition. We  can then propose the following problem: \\
{\it Prove or disprove that for $k=4,6,12,20$ (each case would be interesting) that the minimal k-partitions correspond to the platonic solids.}

\begin{remark}  
 Let us also  mention that if the nodal line conjecture holds (see \cite{MR1605269} and references therein), say for simply connected domains, this would imply a lot on the geometry of minimal spectral partitions.    
\end{remark} 

Finally,  we want to mention some very interesting recent results by Smilansky and collaborators, \cite{MR3000497} for the membrane case in $\Omega\subset \mathbb R^d$ and \cite{MR2909765} for the discrete case. We just indicate their results. They consider the generic case, simple eigenvalues and eigenfunctions which have no higher order zeros, and investigate for the $k$-th eigenfunction $\varphi_k$ the {\em nodal deficiency} $\;d_k =k-\mu(\varphi_k)$. They define a functional, somehow a bit in the spirit of Equation \eqref{chapBH.eq.Lkdef} and show that the critical point of this functional corresponds to a nodal partition. Moreover the Morse index of the critical point turns out to be the deficiency index. It would be very interesting to investigate whether this approach can be extended to the non-generic case.

\bigskip {\noindent\bf Acknowledgements.}
{\small We would like to thank particularly our  collaborators  P. B\'erard, B. Bogosel, P. Charron, C.~L\'ena, B. Noris, M. Nys, M. Persson Sundqvist, S. Terracini, and G. Vial. 
We would also like to thank D. Bucur, T. Deheuvels, E. Harrell, A. Henrot, D.~Jakobson,  J. Lamboley, J. Leydold,  \'E. Oudet, M. Pierre, I. Polterovich, T. Ranner for their interest, their help and for useful discussions. \\
}

\bibliographystyle{abbrv_links}

\def\cprime{$'$} \def\cprime{$'$}

\end{document}